\documentclass[final,onfignum,onetabnum]{siamart171218}

\newcommand*{\defeq}{\stackrel{\text{def}}{=}}
\usepackage{algorithm}
\usepackage{algpseudocode}
\newcommand{\comment}[1]{}
\usepackage{ifpdf}
\usepackage[numbers, sort&compress ]{natbib}
\numberwithin{equation}{section} 
\numberwithin{figure}{section}
\numberwithin{table}{section}

\usepackage{multirow}

\usepackage{lipsum}
\usepackage{amsfonts}
\usepackage{graphicx}
\usepackage{epstopdf}
\ifpdf
  \DeclareGraphicsExtensions{.eps,.pdf,.png,.jpg}
\else
  \DeclareGraphicsExtensions{.eps}
\fi

\numberwithin{theorem}{section}

\newcommand{\TheTitle}{Derivative-Free Optimization of Noisy Functions via Quasi-Newton Methods} 
\newcommand{\TheAuthors}{A. S. Berahas, R. H. Byrd and J. Nocedal}

\headers{DFO of Noisy Functions via Quasi-Newton Methods}{\TheAuthors}

\title{{\TheTitle}\thanks{Submitted to the editors 3/27/2018.
\funding{The first author was supported by  National Science Foundation grant DMS-0810213. The second author was supported by National Science Foundation grant DMS-1620070. The third author was supported by  Department 
       of Energy grant DE-FG02-87ER25047.  }}}

\author{
  Albert S. Berahas\thanks{Department of Industrial Engineering and Management Sciences, Northwestern University, Evanston, IL
    (\email{albertberahas@u.northwestern.edu}).}
  \and
  Richard H. Byrd\thanks{Department of Computer Science, University of Colorado, Boulder, CO (\email{richard@boulder.edu}).}
  \and
  Jorge Nocedal\thanks{Department of Industrial Engineering and Management Sciences, Northwestern University, Evanston, IL
    (\email{j-nocedal@northwestern.edu}).}
}

\usepackage{amsopn}

\usepackage{soul}

\begin{document}

\maketitle

\begin{abstract}
  This paper presents a finite difference quasi-Newton method for the minimization of noisy functions. The method takes advantage of the scalability and power of BFGS updating, and employs an adaptive procedure for choosing the differencing interval $h$ based on the noise estimation techniques of Hamming \cite{hamming2012introduction} and Mor\'e and Wild \cite{more2011estimating}. This noise estimation procedure and the selection of $h$ are inexpensive but not always accurate, and to prevent failures the algorithm incorporates a recovery mechanism  
  that takes appropriate action in the case when the line search procedure is unable to produce an acceptable point. A novel convergence analysis is presented that considers the effect of a noisy line search procedure. Numerical experiments comparing the method to a function interpolating trust region method are presented.
\end{abstract}



\section{Introduction}
The BFGS method has proved to be a very successful technique for nonlinear continuous optimization, and recent work has shown that it is also very effective for nonsmooth optimization \cite{lewis2013nonsmooth,keskar2017limited,lewis2009nonsmooth}---a class of problems for which it was not designed for and for which few would have predicted its success. Moreover,  as we argue in this paper,  the BFGS method with finite difference approximations to the gradient can be the basis of a very effective algorithm for the derivative-free optimization  of  noisy objective functions.

It has  long been recognized in some quarters \cite{gould-private} that one of the best methods for (non-noisy) derivative-free optimization is the standard BFGS method with finite difference gradients.
However, its application in the noisy case has been considered problematic.
 The fact that finite differences can be unstable in the presence of noise  has motivated the development of alternative methods, such as direct search methods \cite{hooke1961direct,KoldLewiTorc03,DennTorc91,lewis2000direct,Wrig96,jones1993lipschitzian,audet2006mesh} and function interpolating trust region methods \cite{powell2006newuoa,Powe02,conn1997convergence,conn1998derivative,conn2009introduction,MaraNoce00a,wild2008orbit,maggiar2015derivative}, which use function evaluations at well spread out points---an indispensable feature when minimizing noisy functions. These algorithms do not, however, scale well with the number of variables. 

In contrast, the L-BFGS method for deterministic optimization is able to build useful quadratic models of the objective function at a cost that is linear in the dimension $n$ of the problem. Motivated by this observation, we propose a finite difference quasi-Newton approach for minimizing functions that contain noise.  
To ensure the reliability of finite difference gradients, we determine the differencing interval $h$ based on an estimate of the noise level (i.e., the standard deviation of the noise). For this purpose, we follow the difference-table technique pioneered by Hamming \cite{hamming2012introduction}, as improved and extended by Mor\'e and Wild \cite{more2011estimating}. This technique samples the objective function $f$ at a small number of equally spaced points  along a random direction, and estimates the noise level from the columns of the difference table. Our optimization algorithm is adaptive, as it re-estimates the noise level and differencing interval $h$ during the course of the optimization, as necessary. 

An important ingredient in the method is the line search, which in our approach serves the dual purpose of computing the length of the step (when the interval $h$ is adequate) and determining when the differencing interval $h$ is not appropriate and should be re-estimated. When the line search is unable to find an acceptable point, the method triggers a \emph{recovery procedure} that chooses between several courses of action. The method must, in general, be able to distinguish between the case when an unsuccessful step is due to a poor gradient approximation (in which case $h$ may need to be increased), due to nonlinearity (in which case $h$ should be maintained and the steplength $\alpha_k$ shortened), or due to the confusing effects of noise.

We establish two sets of convergence results for strongly convex objective functions; one for a fixed steplength strategy and one in which the steplength is computed by an Armijo backtracking line search procedure. The latter is novel in the way it is able to account for noisy function evaluations  during the line search.  In both cases, we prove linear convergence to a neighborhood of the solution,  where the size of the neighborhood is determined by the level of noise.

The results of our numerical experiments suggest that the proposed algorithm is competitive, in terms of function evaluations, with a well-known function interpolating trust region method, and that it scales better with the dimension of the problem and parallelizes easily.
Although the reliability of the noise estimation techniques  can be guaranteed only when the noise in the objective function is i.i.d., we have observed that the algorithm is often effective on problems in which this assumption is violated.

Returning to the first paragraph in this section, we  note that it is the power of quasi-Newton methods, in general, rather than the specific  properties of BFGS that have proven to be so effective in a surprising number of settings, including the subject of this paper. Other quasi-Newton methods  could also prove to be useful. The BFGS method is  appealing because it is simple, admits a straightforward extension to the large-scale setting, and is supported by a compact and elegant convergence theory.

The paper is organized into 6 sections. We  conclude this section with some background and motivation for this work. In Section \ref{sec:det_case} we compare the performance of  finite difference L-BFGS and a function interpolating trust region method on smooth objective functions that do not contain noise.  Section~\ref{sec:noise} presents the algorithm for the minimization of noisy functions, which is analyzed in Section~\ref{analysis}. Section~\ref{numerical} reports the results of numerical experiments, and  Section~\ref{finalr}  summarizes the main findings of the paper.

\subsection{Background}

Although not a mainstream view, the finite difference BFGS method is regarded by some researchers as one of the most effective methods for derivative-free optimization of smooth functions that do not contain noise, and countless users have employed it knowingly or unknowingly in that setting. To cite just one example, if a user supplies only function values, the {\tt  fminunc} {\sc matlab} function automatically invokes a standard finite difference BFGS method. 

  Nevertheless, much research has been performed in the last two decades  to design  other approaches for derivative-free optimization \cite{rios2013derivative,conn2009introduction}, most prominently  direct search methods \cite{hooke1961direct,KoldLewiTorc03,DennTorc91,lewis2000direct,Wrig96,jones1993lipschitzian,audet2006mesh} and  function interpolating trust region methods \cite{powell2006newuoa,Powe02,conn1997convergence,conn1998derivative,conn2009introduction,MaraNoce00a,wild2008orbit,maggiar2015derivative}. Earlier approaches include the Nelder-Mead method \cite{NeldMead65}, simulated annealing \cite{kirkpatrick1983optimization} and genetic algorithms \cite{holland1975adaptation,bethke1978genetic}. 
  
Function interpolating trust region methods are  more robust in the presence of noise than other techniques for derivative-free optimization. This was demonstrated by Mor\'e and Wild \cite{more2009benchmarking}  who report that the {\sc newuoa} \cite{powell2006newuoa}  implementation of the function interpolation approach was   more reliable and efficient in the minimization of both smooth and noisy functions than  the direct search method implemented in  {\sc appspack}  \cite{gray2006algorithm} and the Nelder-Mead method {\sc nmsmax}  \cite{Higham:MCT}.  
Their experiments show that direct search methods are  slow and unable to scale well with the dimension of the problem; their main appeal is that they are robust due to their expansive exploration of $\mathbb{R}^n$, and are easy to analyze and implement. In spite of its efficiency,  the function interpolation approach is limited by the high  linear algebra  cost of the iteration; straightforward implementations require $\mathcal{O}(n^6)$ operations, and although this can be reduced to 
$\mathcal{O}(n^4)$ operations by updating factorizations, this cost is still high for large dimensional problems \cite{conn2009introduction,powell2006newuoa}.

 An established finite difference BFGS algorithm is the \emph{Implicit Filtering} method of Kelley \cite{kelley2011implicit,choi2000superlinear}, which is designed for the case when noise decays as the iterates approach the solution. That method enjoys deterministic convergence guarantees to the solution, which are possible due to the assumption that noise can be diminished at any iteration, as needed.  In this paper we assume that noise in the objective function does not decay to zero, and establish convergence to a neighborhood of the solution, which is the best we can hope for in this setting. 
 
Barton \cite{barton1992computing} describes a procedure for updating the finite difference interval $h$  in the case when the noise is multiplicative (as is the case of roundoff errors). He assumes that a bound on the noise is known, and notes that the optimal choice of $h$ depends (in the one dimensional case) on $|f(x)|/|f''(x)|$. Since estimating this ratio can be expensive, he updates $h$ by observing  how many digits change between $f(x_k+h)$ and $f(x_k)$, beyond the noise level. If this change is too large, $h$ is decreased at the next iteration; if it is too small $h$ is increased. He tested this technique successfully in conjunction with finite difference quasi-Newton methods. In this paper, we assume that the noise level is not known and must be estimated, and discuss how to safeguard against inaccurate estimates of the noise level.

\section{Optimization of Smooth Functions}
\label{sec:det_case}

Before embarking on our investigation of noisy objective functions we consider the case when noise is not present, and compare the performance of an L-BFGS method that uses finite differences to approximate the gradient (FDLM) and a function interpolating trust region method, abbreviated from now on as \emph{function interpolating} (FI) method.
This will allow us to highlight the strengths of each approach, and help set the stage for the investigation of the noisy case. 

 
We write the deterministic optimization problem as 
\begin{equation}
\label{prob1}
\min_{x\in \mathbb{R}^n} f(x),
\end{equation}
where $f: \mathbb{R}^n \rightarrow \mathbb{R}$ is twice continuously differentiable. 
In our numerical investigation, we employ the FI method of Conn, Scheinberg and Vicente \cite{conn2009introduction}. It begins by evaluating  the function at $2n+1$ points along the coordinate directions, and as the iteration progresses, builds a simple quadratic model with minimum norm Hessian. At every iteration a new function value is computed and stored. Once $(n+1)(n+2)/2$ function values are available, a fully quadratic model $m(x)$ is constructed by interpolation. Every new iterate $x_{k+1}$ is given by
\[
       x_{k+1} = \arg  \min_x \{ m(x)| \, \| x-x_k\|_2 \leq \Delta_k \},
\]
where the trust region radius $\Delta_k$ is updated using standard rules from derivative-free optimization \cite{conn2009introduction}. We chose the method and software described in \cite{conn2009introduction} because it embodies the state-of-the-art  of FI methods and yet is simple enough to allow us to evaluate all its algorithmic components. 

The finite difference L-BFGS  (FDLM) method is given by
\begin{align}	\label{qn}
	x_{k+1} = x_k - \alpha_k H_k \nabla_h f(x_k),
\end{align}
where $H_k$ is an approximation to the inverse Hessian,  $\nabla_h f(x_k)$ is a finite difference approximation to the gradient of $f$, and the steplength $\alpha_k$ is computed by an Armijo-Wolfe line search; see \cite{mybook}. 
 We test both forward  differences (FD) 
\begin{align}		\label{sm_f_h}
	[\nabla_{ h,{\mbox{\tiny FD}}} f(x)]_i = \frac{f(x + h_{\mbox{\tiny FD}}e_i) - f(x)}{h_{\mbox{\tiny FD}}},			
\quad\mbox{where}\quad
		h_{\mbox{\tiny FD}}= \max\{1, |x|\}\, (\epsilon_{\rm \rm m})^{1/2},
\end{align}
	and  central differences (CD)
\begin{align}		\label{sm_c_h}
	[\nabla_{ h,{\mbox{\tiny CD}}} f(x)]_i = \frac{f(x + h_{\mbox{\tiny CD}}e_i) - f(x - h_{\mbox{\tiny CD}}e_i)}{2h_{\mbox{\tiny CD}}}, \quad\mbox{where}\quad
		h_{\mbox{\tiny CD}}=\max\{1, |x|\}\, (\epsilon_{\rm \rm m})^{1/3}.
\end{align}
Here  $\epsilon_{\rm m}$ is machine precision and $e_i \in \mathbb{R}^n$ is the $i$-th canonical vector. (We note in passing that by employing complex arithmetic \cite{squire1998using}, highly accurate derivatives can be obtained through finite differences using a very small $h$. The use of complex arithmetic is, however, not always possible in practice.)


A  sample of results is given in Figure~\ref{detres},  which plots the optimality gap $ (f(x_k)-f^\star )$ versus number of function evaluations. 
These 4 problems are from the Hock-Schittkowski collection; see \cite{schttfkowski1987more} where their description and optimal function value $f^\star$ are given. 
 For the FDLM method, the number of function evaluations at the $k$-th iteration  is $n + t^k_{\rm ls}$ (FD) or  $2n+t^k_{\rm ls}$  (CD), where $t^k_{\rm ls}$ is the number of function evaluations performed  by the Armijo-Wolfe line search at the $k$-th iteration.
  The FI method is described in \cite{conn2009introduction}, and we refer to as {\tt DFOtr}; it  normally computes only one  new function value per iteration. As a benchmark, we also report the results of {\tt NOMAD} \cite{audet2006mesh,abramson2011nomad}, a well known direct-search method. It is clear that {\tt NOMAD} is by far the slowest code.

\begin{figure}[H]
\begin{centering}


\includegraphics[width=0.24\textwidth]{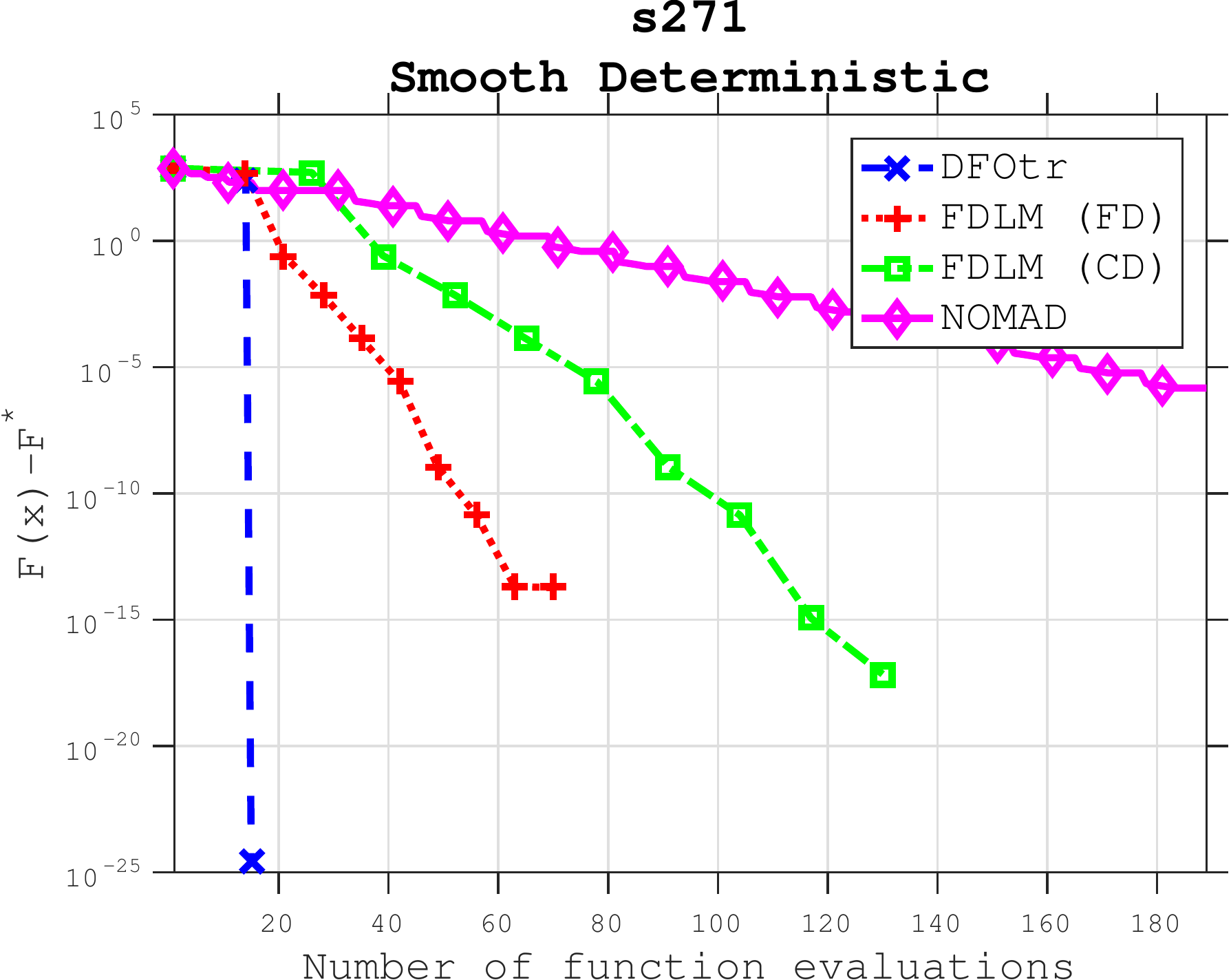}
\includegraphics[width=0.24\textwidth]{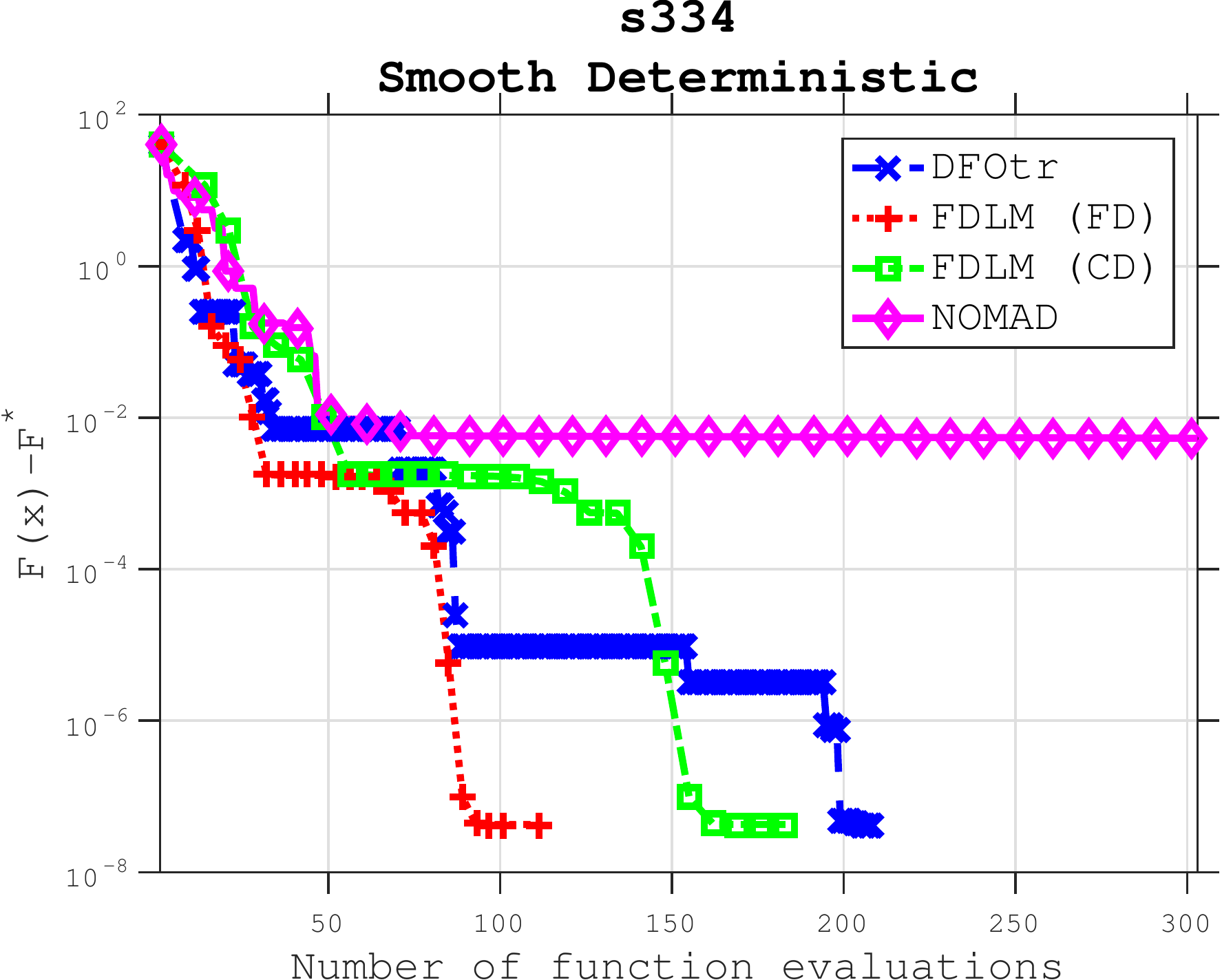}
\includegraphics[width=0.24\textwidth]{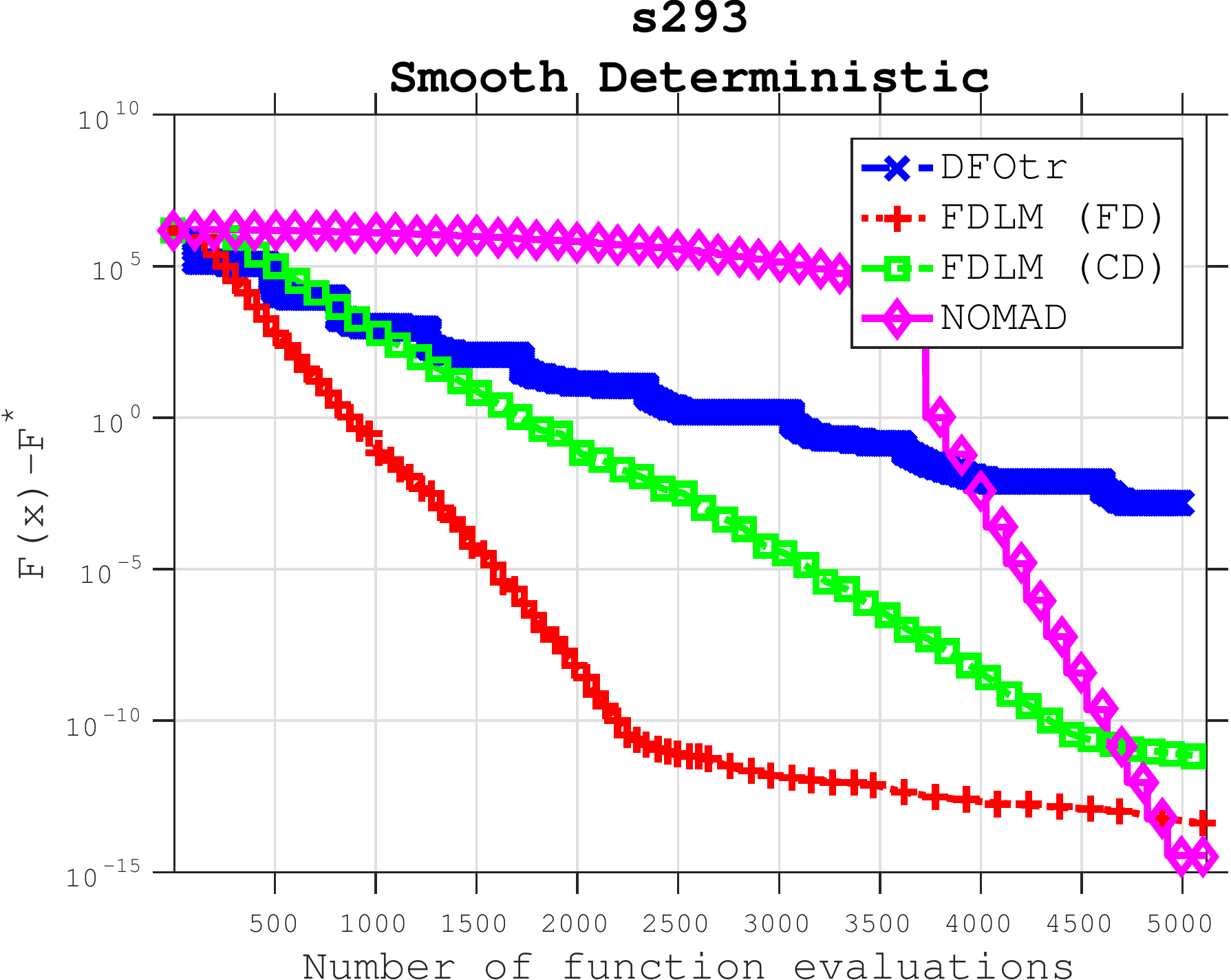}
\includegraphics[width=0.24\textwidth]{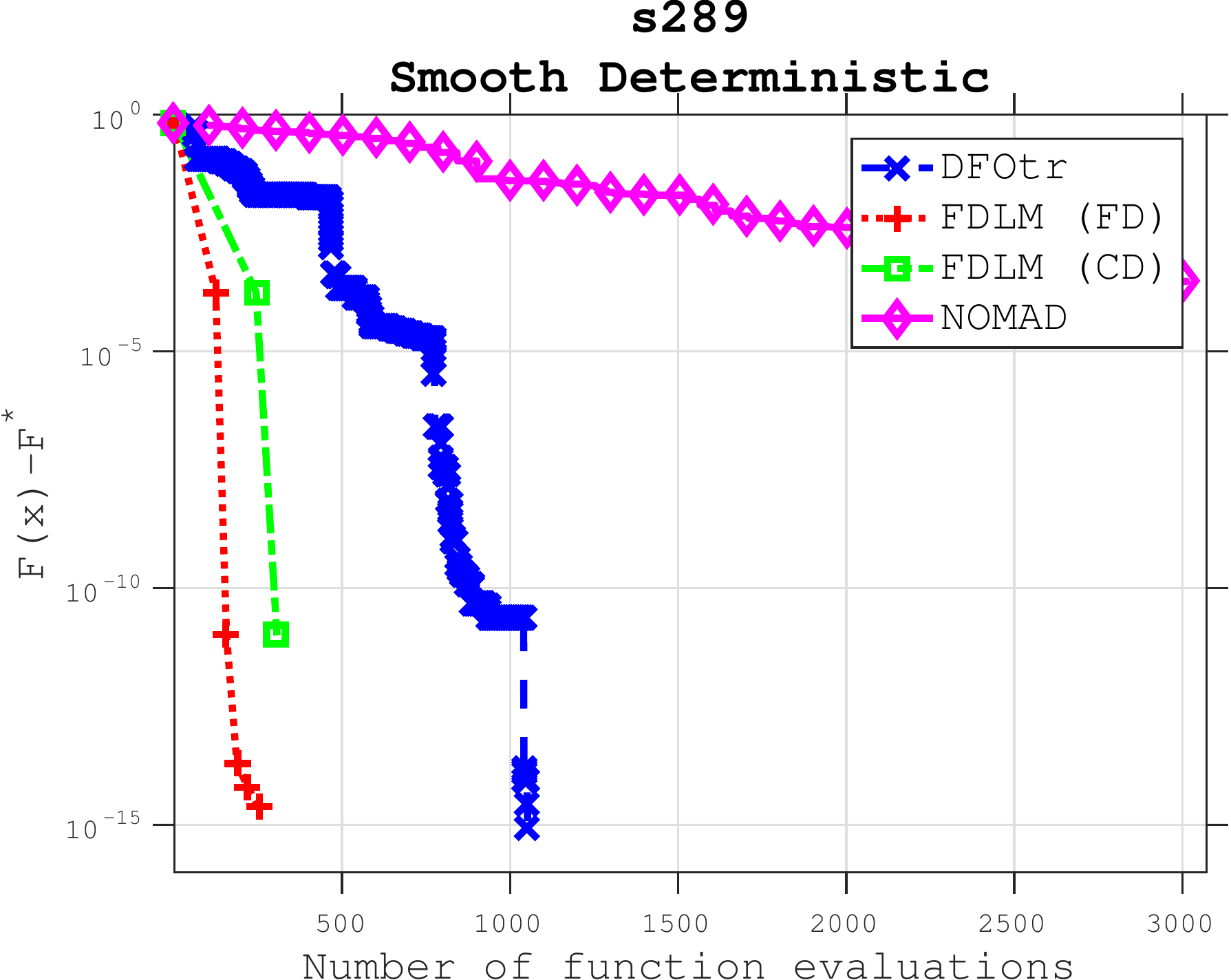}

\par\end{centering}
\caption{\small (Problems without Noise) Performance of the function interpolating trust region method {\tt (DFOtr)} described in \cite{conn2009introduction}, {\tt NOMAD} \cite{audet2006mesh,abramson2011nomad}, and the finite difference L-BFGS method ({\tt FDLM}) using forward or central differences, on 4 problems from the Hock-Schittkowski collection \cite{schttfkowski1987more}.  
}
\label{detres}
\end{figure}

The first problem in Figure~\ref{detres} is quadratic; the FI method will generally have better performance in this case because it has finite termination while the FDLM does not.  The rest of the plots in Figure~\ref{detres} illustrate other behavior of the methods observed in our tests.  Performance and data profiles are reported in Appendix~\ref{extended_schitt_smooth}.
%
%
We do not report  {\sc cpu} times  because the FI code of Conn, Scheinberg and Vicente \cite{conn2009introduction}  is not designed to be efficient in this respect, requiring $\mathcal{O}(n^6)$ flops per iteration, and is thus much slower than the FDLM method as $n$ increases. There exist, much faster FI codes (albeit much more complex)   such as that of  Powell \cite{powell2006newuoa}, whose linear algebra cost is only $\mathcal{O}(n^4)$. Regardless of the implementation, scalability remains one of the main limitations of FI methods. To show that the FDLM approach can deal with problems that are out of reach for FI methods, we report in Table~\ref{rosen} results on the extended Rosenbrock function of various dimensions (as a reference, {\tt DFOtr} requires more than 1,800 seconds for $n =100 $).


\begin{table}[htp!]
\centering
\caption[Caption for LOF]{\small {\sc cpu}  time (in seconds) required by the finite difference L-BFGS method ({\tt FDLM}), using forward and central differences, to reach the accuracy $f-f^\star <10^{-6}$, on the {\tt Extended Rosenbrock} function of various dimensions ($n$). The results were obtained on a workstation with 32GB of RAM and 16 Intel Xeon X5560 cores running at 2.80GH.
 }
\small\medskip
\begin{tabular}{|c||cccccc|}  \hline  
$n$                                                      & 10 & 50 & 100 & 1000 & 2000 & 5000  \\ \hline\hline
\begin{tabular}[c]{@{}c@{}} FD\end{tabular}  & $1.8\times10^{-1}$ & $3.1\times10^{-1}$ & $7.6\times10^{-1}$  & $3.3\times10^{2}$  & $2.9\times10^{3}$ & $5.5\times10^{4}$     \\ \hline
\begin{tabular}[c]{@{}c@{}} CD\end{tabular}  & $2.0\times10^{-1}$   & $4.0\times10^{-1}$ & $1.2$ & $5.6\times10^{2}$  & $5.9\times10^{3}$ & $1.0\times10^{5}$ \\
\hline      
\end{tabular}
\label{rosen}
\end{table}

Overall, our numerical results suggest that  FDLM  is a very competitive method for (non-noisy) derivative-free optimization, particularly for large problems.
Conclusive remarks about the relative performance of the FI and FDLM approaches are, however, difficult to make because there are a variety of FI methods that follow significantly different approaches than the method tested here. For example, the codes by Powell \cite{powell2006newuoa,Powe02} include two trust region radii and employ a different procedure for placing the interpolation points. Some implementations of FI methods start with $\mathcal{O}(n^2)$ function values in order to build a quadratic model; other implementations only require $\mathcal{O}(n)$ function values to start. 


\subsection{Discussion}

One of the appealing features of the function interpolating method is that it can move after only $1$  function evaluation, as opposed to the $\mathcal{O}(n)$ function evaluations required per iteration by the FDLM approach. However, if the model is not accurate (or the trust region is too large) and results in an unsuccessful step, the FI approach may require many function evaluations before the model is corrected (or the trust region is properly adjusted); this can be seen in the second, third, and fourth plots in Figure~\ref{detres}.
 On the other hand, finite difference approximations to the gradients \eqref{sm_f_h}-\eqref{sm_c_h} carry some risks
, as discussed in Section~\ref{discussion}. We did not encounter  difficulties in our tests on smooth functions, but this is an issue that requires careful consideration in general.

The FI and FDLM methods both construct quadratic models of the objective function, and compute a step as the minimizer of the model, which distinguishes them from most direct and pattern search methods. But the two methods differ significantly in nature. FI methods define the quadratic model by interpolating previously evaluated function values and (in some variants) by imposing a minimum-norm change with respect to the previous model. The estimation of the gradient and Hessian is done simultaneously, and the gradient approximation becomes accurate only when the trust region is small; typically near the solution \cite{conn2008geometry}. In FI methods it is the overall quality of the model that matters. The location of the sampling points is determined by the movement of the FI algorithm. These points  tend to lie along a subspace of $\mathbb{R}^n$, which can be harmful to the interpolation process. To guard against this, many FI methods include a procedure (a geometry phase) for spreading the sample points  in $\mathbb{R}^n$ so that the interpolation problem is not 
badly conditioned. 

In contrast to FI methods, the finite difference BFGS method invests significant computation in the estimation of the gradient ($n$ or $2n$ function evaluations) and delegates the construction of the model to BFGS updating. The function evaluations used in the estimation of the gradient parallelize easily, and the linear algebra costs are quite modest, as updating the model and computing a step can be performed in a small multiple of $n$ using the limited memory BFGS approach \cite{mybook}. Practical experience indicates that BFGS and L-BFGS give rise to well-scaled search directions that typically require little or no extra cost during the line search. Placement of the sample points is  along a linearly independent set of directions (during finite differencing) and in this sense the method has some resemblance with pattern search methods -- but their similarities stop there.

It is rare to find in the derivative-free optimization (DFO)  literature comparisons between the finite difference BFGS method and direct search or function interpolating trust region methods, even when testing smooth objective functions without noise. One reason for this omission may be the perception that finite difference based methods are inefficient as they require at least $n$ function evaluations per iteration, whereas other methods for DFO are more frugal in this respect. In an early paper, Powell  \cite{powell1972unconstrained} wrote in regards to derivative-free optimization: ``I believe that eventually the better methods will not use derivative approximations.''
In this paper, we argue that a finite difference L-BFGS method is, indeed, an effective technique  for the minimization of certain classes of noisy functions.

These observations are a good starting point for our discussion of stochastic optimization problems. 


%

\section{Optimization of Noisy Functions}
\label{sec:noise}

We study problems of the form
\begin{align}	\label{objn}
	   \min_{x\in \mathbb{R}^n} f(x) = \phi(x)+  \epsilon(x).
           \end{align}
 We assume that $\phi: \mathbb{R}^n \rightarrow \mathbb{R}$ is a smooth twice continuously differentiable function and $\epsilon(x)$ is a random variable whose distribution is independent of $x$. (The notation $\epsilon(x)$ simply means that at any $x$ we compute the realization of a random variable.) The model \eqref{objn}  covers the case of multiplicative noise, i.e., when $f(x)= \phi(x)(1+ \hat \epsilon(x))$ and $\hat \epsilon(x)$ is a random variable. To establish convergence results, we will assume that $\epsilon(x)$ is i.i.d and bounded, but our algorithm and presentation apply to the general model \eqref{objn}. Specifically, we are also interested in the case when the noise $\epsilon(x)$ is deterministic, as is the case of roundoff errors or when adaptive algorithms are part of the function evaluation.  
 
 \subsection{The Finite Difference Interval}

 Our method relies crucially on the computation of an appropriate finite difference parameter $h$. It has been shown by Mor\'e and Wild \cite{more2012estimating} that if one can estimate the level of noise in $f$, which we denote by $\sigma_f$, one can compute a nearly optimal $h$ for which the error in  approximating  $\nabla f(x)$ 
is $\mathcal{O}(\sigma_f^{1/2})$ and $\mathcal{O}(\sigma_f^{2/3})$ for forward differences and central differences, respectively.   
  
The \emph{noise level} $\sigma_f$ of the function $f$ given in \eqref{objn} is defined as the standard deviation of $\epsilon(x)$, i.e.,
\begin{equation}   \label{sigman}
         \sigma_f= ({\rm Var}\{\epsilon(x)\})^{1/2}. 
 \end{equation}
This quantity can be estimated using the difference table technique proposed by Hamming \cite{hamming2012introduction}, as extended and refined by  Mor\'e and Wild \cite{more2011estimating}. 
We denote our \emph{estimate} of the noise level of $f$ by $\epsilon_f$.

With $\epsilon_f$ in hand, we define the finite differencing interval as suggested in \cite{more2012estimating}. The $i$-th component of the forward difference approximation of the gradient of $f$ at $x$ is given by
	\begin{align}	\label{fd-noise}
	[\nabla_{ h,{\mbox{\tiny FD}}} f(x)]_i = \frac{f(x + h_{\mbox{\tiny FD}}e_i) - f(x)}{h_{\mbox{\tiny FD}}},			
\quad\mbox{where}\quad
		h_{\mbox{\tiny FD}} = 8^{1/4} \left(\frac{\epsilon_f}{\nu_2}\right)^{1/2},
	\end{align}
	and $\nu_2$ is an 
	estimate of $\max_{x \in [x,x+h_{\mbox{\tiny FD}}e_i ]}| \nabla^2 f(x)^T e_i |$.
	The central difference approximation is given by
	\begin{align}	\label{cd-noise}
	[\nabla_{ h,{\mbox{\tiny CD}}} f(x)]_i = \frac{f(x + h_{\mbox{\tiny CD}}e_i) - f(x - h_{\mbox{\tiny CD}}e_i)}{2h_{\mbox{\tiny CD}}}, \quad\mbox{where}\quad
		h_{\mbox{\tiny CD}} = 3^{1/3} \left(\frac{ \epsilon_f}{\nu_3}\right)^{1/3},
	\end{align}
	and $\nu_3$ is an estimate of the third derivative along $e_i$, in an interval of length $2 h_{\mbox{\tiny CD}}$ around $x$.  Since estimating 2nd or 3rd derivatives along each coordinate direction is expensive, in our implementation we perform this estimation once along a random direction, as discussed in Section~\ref{discussion}. In the sequel, we let $\nabla_{ h} f(x)$ denote \eqref{fd-noise} or \eqref{cd-noise} when the distinction is not important.
	
\subsection{Noise Estimation} 

 The noise level $\sigma_f$ of a function measures the uncertainty in the computed function values, and can be estimated using Hamming's  table of function differences \cite{hamming2012introduction}.
To generate this table, Mor\'e and Wild \cite{more2011estimating} first choose a random direction $v \in \mathbb{R}^n$ of unit norm, evaluate the function at  $q+1$  equally spaced points (with spacing $\delta$) along that ray, and compute the function differences 
\begin{subequations} 
\begin{align}
	&\Delta^0 f(x) = f(x), \label{func_diff0}\\
	&\Delta^{j+1} f(x) = \Delta^{j}[\Delta f(x)]=\Delta^j[f(x+\delta v)] - \Delta^j[f(x)], \quad j \geq 0. \label{func_diffk}
\end{align}
\end{subequations}
Let $x_i = x + u_i \delta v$, where $u_i = -q/2 + i$, for  $i=0,\dots,q$, 
denote the sample points centered around $x$.  Using the computed function differences one can build a  table whose entries are
\begin{align}   \label{htable}
	T_{i,j} = \Delta^j f(x_i), \qquad 1\leq j \leq q \quad \text{and } \quad 0 \leq i \leq q-j.
\end{align}
Hamming's  approach  relies on the fact that differences in $\phi$ tend to zero rapidly, while the differences in $\epsilon(x)$ are bounded away from zero  \cite{hamming2012introduction}. As a result, the noise level $\sigma_f$ can be estimated from the mean of the squares of the columns of this difference table \cite{more2011estimating}. Specifically, for each $j$ (where $j$ indexes a column of the difference table and also represents the order of differencing) one defines
\begin{align}  \label{table}
	 s^2_j \defeq \frac{\gamma_j}{q+1-j} \sum_{i=0}^{q-j}T_{i,j}^2, \qquad\mbox{where} \quad \gamma_j = \frac{(j!)^2}{(2j)!}.
\end{align}
Once an appropriate value of $j$ is identified, the \emph{noise estimate} $\epsilon_f$ is defined as
\begin{equation}    \label{ef}
          \epsilon_f \leftarrow s_j.
\end{equation}

Mor\'e and Wild \cite{more2011estimating}  propose an efficient  and reliable procedure, referred to as  {\tt ECnoise}, for determining the free parameters in this procedure, namely: (i)  the order of differencing $j$; (ii) the number of points $q$ where the function is evaluated; and (iii) the spacing $\delta$ between the sample points; see Section~\ref{discussion}.  {\tt ECnoise}  is inexpensive as it typically requires only between $4-10$ function evaluations to estimate the noise level.  
 It is available at: \url{http://www.mcs.anl.gov/~wild/cnoise/}.  
 
 We should note that if the noise is stochastic one can make use of the Central Limit theorem and estimate the variance by sampling the function a few times (10-20) at a given point. This approach will, however, not work if the noise is deterministic and in that case we need to resort to the technique proposed by Hamming \cite{hamming2012introduction}. In our code we assume that the nature of the noise is not known and therefore estimate the noise using the Hamming table since it is applicable in both cases.


\subsection{Specification of the Finite Difference L-BFGS  Method}

We are now ready to  present  the algorithm for the minimization of the function \eqref{objn}.
It invokes two procedures, \texttt{LineSearch} and \texttt{Recovery}, that together produce a new iterate and, if necessary, recompute the estimate of the noise. These two procedures are described in detail after the presentation of the algorithm.  

The algorithm stores the smallest function value obtained during the finite difference gradient computation, \eqref{fd-noise} or \eqref{cd-noise}. Specifically, for forward differences we define
\begin{equation}   \label{fs}
     f_s= \min_{x_i \in S} f(x_i), \quad\mbox{where} \quad  S= \{x_i: x_i= x+ h_{\mbox{\tiny FD}}e_i , \, i=1,\ldots, n\},
\end{equation}
and let $x_s$ denote a point where the minimum is achieved. Since $S$ is a stencil, we refer to $x_s$ as the \emph{best point on the stencil}. 

\begin{algorithm}[]
\caption{Adaptive Finite Difference L-BFGS (FDLM)}
\label{alg1}
{\bf Inputs:} $f$ (objective function), $x_0$ (initial iterate), 
$a_{\rm max}$ (max number of backtracks), 
$k\leftarrow 0$ (iteration counter), $t_{\rm count} \leftarrow 0$ (function evaluation counter), $t_{\rm ls} \leftarrow 0$ (function evaluation counter during {\tt LineSearch} routine), $t_{\rm rec} \leftarrow 0$ (function evaluation counter during {\tt Recovery} procedure), $\zeta \in (0,1)$ (curvature threshold)
\begin{algorithmic}[1]
\State Compute $f_0 = f(x_0)$; set 
$t_{\rm count}=1$
\State Compute an estimate $\epsilon_f$ of the noise  using {\tt ECnoise} \cite{more2011estimating}, at the cost of $t_{\rm ecn}$  function evaluations
\State Update the function evaluation counter: $t_{\rm count}=t_{\rm count}+ t_{\rm ecn}$
\State Compute $h$ via \eqref{fd-noise} (FD) or \eqref{cd-noise}  (CD)
\State Compute $\nabla_{h} f(x_{0})$ using \eqref{fd-noise} or \eqref{cd-noise}, and store $(x_s,f_s)$
\State  $t_{\rm count}=t_{\rm count}+n$ (FD) or $t_{\rm count}=t_{\rm count}+2n$ (CD)
\While{  a convergence test is not satisfied}
\State Compute $d_k = -H_k \nabla_{h} f(x_{k})$ using the L-BFGS approach \cite{mybook}
\State $(x_{+},f_{+},\alpha_k,t_{ls},LS_{\rm flag})=\texttt{LineSearch}(x_{k},f_k,\nabla_{h} f(x_{k}),d_k,a_{\rm max})$
\State  $t_{\rm count}=t_{\rm count}+t_{\rm ls}$
\If{$LS_{\rm flag}=1$} \Comment{\textcolor{blue}{Line search failed}}
\State $(x_{+},f_{+},h,t_{\rm rec})=\texttt{Recovery}(f_k,x_{k},d_k, h,x_s,f_s)$ 
\State  $t_{\rm count}=t_{\rm count}+t_{\rm rec}$
\EndIf
\State $x_{k+1} = x_{+}$ and $f_{k+1}=f_{+}$
\State Compute $\nabla_{h} f(x_{k+1})$ using \eqref{fd-noise} or \eqref{cd-noise}, and store $(x_s,f_s)$
\State Compute curvature pair: $s_{k} = x_{k+1} - x_k$ and $y_{k} = \nabla_{h} f(x_{k+1}) - \nabla_{h} f(x_{k})$
 \State Store  $(s_k, y_k)$ if $s_{k}^Ty_{k}  \geq \zeta \| s_k \| \| y_k \|$
\State  $t_{\rm count}=t_{\rm count}+n$ (FD) or $t_{\rm count}=t_{\rm count}+2n$ (CD)
\State $k=k+1$ 
\EndWhile
\end{algorithmic}
\end{algorithm}

We now discuss the four main components of the algorithm.

\subsubsection{Line Search} The \texttt{LineSearch} function (Line 9 of Algorithm \ref{alg1}) aims to find a steplength $\alpha_k$ that satisfies the  Armijo-Wolfe conditions,
\begin{subequations} 
\begin{align}
	f(x_{k} + \alpha_k d_k) \leq f(x_k) + c_1\alpha_k \nabla_h f(x_k)^T d_k, \qquad\mbox{(Armijo condition)} \label{eq:armijo}\\
	\nabla_{h} f(x_{k} + \alpha_k d_k)^T d_k \geq c_2 \nabla_{h} f(x_{k})^T d_k , \qquad\mbox{(Wolfe condition)} \label{eq:wolfe}
\end{align}
\end{subequations} 
for some constants $0 < c_1 < c_2 < 1$.
In the deterministic setting, when the gradient is exact and the objective function $f$ is bounded below,  there exist intervals of steplengths $\alpha_k$ that satisfy the Armijo-Wolfe conditions \cite[Lemma 3.1]{mybook}. However, when $f$ is noisy, satisfying \eqref{eq:armijo}-\eqref{eq:wolfe} can be problematic. To start, $d_k$ may not be a descent direction for the smooth underlying function $\phi$ in \eqref{objn}, and even if it is, the noise in the objective may cause the line search routine to make incorrect decisions. Therefore, we allow only a small number ($a_{\rm max}$) of line search iterations while attempting to satisfy   \eqref{eq:armijo}-\eqref{eq:wolfe}. We also introduce the following relaxation: if \eqref{eq:armijo}-\eqref{eq:wolfe} is not satisfied at the first trial point of the line search ($\alpha_k=1$), we relax the Armijo condition \eqref{eq:armijo} for subsequent trial values as follows:
\begin{align}
	f(x_{k} + \alpha_k d_k) \leq f(x_k) + c_1\alpha_k \nabla_h f(x_k)^T d_k + 2\epsilon_f. \label{armijo-mod}
\end{align}

There are three possible outcomes of the \texttt{LineSearch} function: (i) a steplength $\alpha_k$ is found that satisfies the  Armijo-Wolfe conditions, where the Armijo condition may have been relaxed as in \eqref{armijo-mod}; (ii) after $a_{\rm max}$ line search iterations, the line search is  able to find a steplength $\alpha_k$ that only satisfies the relaxed Armijo condition \eqref{armijo-mod} but not the Wolfe condition; 
(iii) after $a_{\rm max}$ line search iterations the previous two outcomes are not achieved; we regard this as a line search failure, set $LS_{\rm flag}=1$, and call the \texttt{Recovery} function to determine the cause of the failure and take corrective action. 


\subsubsection{Recovery Mechanism} As mentioned above, the \texttt{Recovery} subroutine is the mechanism by which action is taken when the line search  fails to return an acceptable point.  This can occur due to the confusing effect of noisy  function evaluations,  a poor gradient approximation, or high nonlinearity of the objective function (the least likely case). The procedure is described in Algorithm~\ref{recover}.

The input to the \texttt{Recovery} subroutine is the current iterate $x_k$ and function value $f_k$, the search direction $d_k$, the current finite difference interval $h$, and the best point $x_s$ in the finite difference stencil together with the corresponding function value $f_s$. 

The \texttt{Recovery} routine can take three actions:
\begin{enumerate}
\item Return a new estimate of the differencing interval $h$ and leave the current iterate $x_k$ unchanged. In this case, a new noise estimate is computed along the current search direction $d_k$. 
\item Generate a new iterate $x_{k+1}$ without changing the differencing interval $h$. The new iterate is given by a small perturbation of  $x_k$ or by the best point in the stencil $x_s$. 
\item Leave the current iterate unchanged and compute a new estimate of the noise level, along a random direction, and a new finite difference interval $h$.
\end{enumerate}


\begin{algorithm}[]
\caption{Recovery Routine}
\label{recover}
{\bf Inputs:} $x_k$ (current iterate), $f_k = f(x_k)$ (current function value), $h$ (current finite difference interval), $\gamma_1 \in (0,1)$, $\gamma_2>1$ (finite difference interval acceptance/rejection parameters), $d_k$ (search direction), $(x_s,f_s = f(x_s))$ (best point on the stencil,)
$t_{\rm rec} \leftarrow 0$ (function evaluation counter for this routine)


\begin{algorithmic}[1]
\State Compute new noise estimate ${\epsilon}_f^{d_k}$ \emph{in direction} $d_k$, and new $\bar{h}$ via \eqref{fd-noise} or \eqref{cd-noise}.
\State Update function evaluation counter: $t_{\rm rec}=t_{\rm rec}+f_{\rm ecn}$ 
\If{$\bar{h} < \gamma_1 h$ OR $\bar{h} > \gamma_2 h$}
\State $h = \bar{h}$ and $x_{+} = x_k$, $f_+ = f_k$ \Comment{\textcolor{red}{Case 1}}
\Else 
\State $x_{h} = x_k +  h \frac{d_k}{\|d_k \|}$   \Comment\textcolor{blue}{Compute small perturbation of $x_k$}
\State Compute $f_{h} = f(x_{h})$
\State $t_{\rm rec}=t_{\rm rec}+1$
\If{$x_{h}$ satisfies the Armijo condition \eqref{eq:armijo}} 
\State $x_{+}=x_{h}$, $f_{+} = f_{h}$, $h =  h$	\Comment{\textcolor{red}{Case 2}} 
\Else
\If{$f_{h} \leq f_s$ AND $f_{h} \leq f_k$}
\State $x_{+} = x_{h}$, $f_{+} = f_{h}$, $h =  h$ \Comment{\textcolor{red}{Case 3}}
\ElsIf{$f_k > f_s \text{   AND   } f_{h} >  f_s$}
\State $x_{+} = x_{s}$, $f_{+} = f_s$	, $h =  h$ \Comment{\textcolor{red}{Case 4}}
\Else
\State $x_{+} = x_k$, $f_{+} = f_{k}$ \Comment{\textcolor{red}{Case 5}}
\State Compute new noise estimate $\epsilon_f^{v_k}$ along \emph{random direction} $v_k$ and new  $\bar{h}$ 
\State $ h= \bar h$
\State $t_{\rm rec}=t_{\rm rec}+f_{\rm ecn}$ 
\EndIf
\EndIf
\EndIf
\end{algorithmic}

{\bf Output:} $x_{+},f_{+},h,t_{\rm rec}$
\end{algorithm}

\medskip
To start, the \texttt{Recovery} routine invokes the {\tt ECnoise} procedure to compute a new noise estimate $ \epsilon_f^{d_k}$ along the current search direction $d_k$,  and the corresponding differencing interval $\bar h$ using formula \eqref{fd-noise} or \eqref{cd-noise}. We estimate the noise along $d_k$ because
the {\tt Recovery} procedure may compute a new iterate along $d_k$, and it thus seem natural to explore the function in that direction. If the current differencing interval $h$ differs significantly from the new estimate $\bar h$, then we suspect that  our current noise estimate is not reliable, and return $\bar h$ without changing the current iterate (Case 1). This feature is especially important when the noise is multiplicative since in this case the noise level changes over the course of optimization, and the finite difference interval may need to be updated frequently. 
 
 On the other hand (Line 5), if the two differencing intervals, $h$ and $\bar h$, are similar, we regard them as reliable and assume  that the line search procedure failed due to the confusing effects of noise. We must therefore generate a new iterate by other means. We compute a small perturbation of $x_k$, of size $h$, along the current search direction $d_k$ (Line 6); if this point $x_h$ satisfies the Armijo condition \eqref{eq:armijo}, it is accepted, and the procedure terminates on Line 10 (Case 2). Otherwise we make use of the best point $x_s$ on the stencil. If the function value $f_h$ at the trial point is less than both the current function value $f_k$ and $f_s$, then $x_h$ is accepted and the procedure ends on Line 13 (Case 3). Else, if $f_s$ is smaller than both $f_k$ and $f_h$, we let $x_s$
 be the new iterate and terminate on Line 15 (Case~4). These two cases are inspired by the global convergence properties of pattern search methods, and do not require additional function evaluations.
 
 If none of the conditions above are satisfied, then the noise is re-estimated along a random direction ($v_k \in \mathbb{R}^n$) using {\tt ECnoise}, a new differencing interval is computed, and the {\tt Recovery} procedure terminates without changing the iterate (Case 5).  An additional action that could be taken in this last case is to switch to higher-order differences if this point of the algorithm is ever reached, as discussed in Section~\ref{numerical}.
 
 \subsubsection{ Hessian approximation} Step 8 of Algorithm~\ref{alg1} computes the L-BFGS  search direction. The inverse Hessian approximation $H_k$  is updated using standard rules \cite{mybook}  based on the  pairs $\{s_j, y_j\}$, where
\begin{align}	\label{eq:curv_pairs}
	s_{k} = x_{k+1} - x_k,\qquad y_{k} = \nabla_{h} f(x_{k+1}) - \nabla_{h} f(x_{k}).
\end{align}
When the line search is able to satisfy the Armijo condition but not the Wolfe condition, there is no guarantee that the product $s_k^Ty_k$  is positive. In this case the pair $(s_k, y_k)$ is discarded if the curvature condition $s_k^Ty_k  \geq \zeta \| s_k \| \| y_k\|$ is not satisfied for some $\zeta \in (0,1)$.

\subsubsection{ Stopping Tests}  A variety of stopping tests have been proposed in the derivative-free optimization literature; see e.g., \cite{larson2013non,conn2009introduction,powell2006newuoa,KoldLewiTorc03,hooke1961direct,kelley2011implicit}. 
Here we discuss two approaches that can be employed in isolation or in combination,  as no single test is best suited for all situations.  

\smallskip
 { \em  Gradient Based Stopping Test. } One could terminate the algorithm as soon as    
\begin{equation} \label{normtest}
      \| \nabla_{h} f(x_{k}) \|_\infty \leq {\tt tol},
\end{equation}
where {\tt tol} is a user-specified parameter. When using forward differences, the best one can hope is for the norm of the gradient approximation to be $\tau (\epsilon_f)^{1/2}$, where $\tau$ depends on the norm of the second derivative. For central differences, the best accuracy is $\bar \tau (\epsilon_f)^{2/3}$, where $\bar \tau$ depends on the norm of the third derivative. 


\smallskip
 {\em  Function Value Based Stopping Test. }One could also terminate the algorithm when
\begin{equation}  \label{ftest}
       |f(x_k) - f(x_{k-1})| \leq \hat \tau \epsilon_f,
\end{equation}
for $\hat \tau >1$. This test is reasonable because the ECnoise procedure that estimates $\epsilon_f$ is scale invariant. 
However, there is a risk that \eqref{ftest} will trigger termination too early, and to address this  one  could employ a moving average. The test can have the form
\begin{align*}
		\frac{ |f_{\rm  MA}(x_k) - f(x_k)|}{|f_{\rm MA}(x_k)|}  \leq {\tt tol}    \quad \text{or} \quad \frac{ |f_{\rm MA}(x_k) - f(x_k)| }{ \max \{1, |f_{\rm MA}(x_k)| \}} \leq {\tt tol},
\end{align*}
where $f_{\rm MA}(x_k)$ is a moving average of  function values, of length $M$,  calculated as follows. Let $f_k = f(x_k)$ and let $F^k = [f_{k-j+1},...,f_{k-1},f_k]$ be the vector formed by the most recent function values, where $j = \min \{k+1,M \}$. We define
	\begin{align}
		f_{\rm MA}(x_k) = \frac{1}{j} \sum_{i=1}^{j} F_i^k.   \label{mav}
	\end{align}
An alternative, is to formulate the stop test as
\begin{align*}
		|f_{\rm  MA}(x_k) - f(x_k)| \leq \tau (\epsilon_f)^{1/2} \ \ \mbox{for FD}, \quad \text{or} \quad |f_{\rm MA}(x_k) - f(x_k)| \leq \tau (\epsilon_f)^{2/3} \ \ \mbox{for CD}.
\end{align*}

\subsection{Implementation of the Noise Estimation Procedure}   \label{discussion}

{\tt ECnoise}  has three  parameters that if not appropriately chosen can cause  {\tt ECnoise}  to fail to return a reliable estimate of the noise level: (i) the order of differencing $j$; see \eqref{ef}; (ii) the number of points $q$ used in the noise estimation, and (iii) the spacing $\delta$ between the points.

We have found the strategy proposed by Mor\'e and Wild \cite{more2011estimating} to be effective in deciding the order of differencing $j$. It determines that $j$ is appropriate  if  the values of $\sigma_i$ surrounding $\sigma_j$ are close to each other \emph{and} if there are changes in sign among elements of the $j$-th column of the difference table $T_{i,j}$; the latter 
is a clear indication that the entries of the $j$-th column are due to noise.

Concerning the number $q$ of sample points, we found that the values employed by {\tt ECnoise} \cite{more2011estimating} are reliable, namely  $4-8$ function evaluations for stochastic noise and  $6-10$  for deterministic noise. However, our experience suggests that those settings may be conservative, and in our implementation use only $4$ function evaluations for stochastic noise and $6$ function evaluations for deterministic noise.

 The finite difference approximations \eqref{fd-noise} and \eqref{cd-noise}  require a coarse estimate of the second or third derivatives ($\nu_2$, $\nu_3$), and are often fairly insensitive to these estimates. However, in some difficult cases poor estimates may prevent the algorithm from reaching the desired accuracy. In our implementation we estimate $\nu_2$  using the  heuristic proposed in \cite[Section 5, Algorithm 5.1]{more2012estimating} at the cost of $ 2-4$ function evaluations. If this heuristic fails to provide a reasonable estimate of $\nu_2$,  we employ the finite difference table \eqref{htable} as a back-up to construct a rough approximation of the norm of the second derivative.  In our experiments this back-up mechanism has proved to be adequate. When computing central difference approximations to the gradient, we simply set $\nu_3 \leftarrow \nu_2$, for simplicity.

\section{Convergence Analysis}
\label{analysis}
In this section, we present two sets of convergence results for the minimization of the noisy function \eqref{objn} under the assumption that $\phi$ is strongly convex. First, we analyze a method that uses a fixed steplength, and then consider a more sophisticated version that employs a line search to compute the steplength. The main contribution of our analysis is the inclusion of this line search and the fact that
we do not assume that the errors in objective function or gradient go to zero, in any deterministic or  probabilistic
sense; we only assume a bound on these errors.  To focus on these issues, we assume that $H_k=I$, because the proof for a general positive definite
matrix $H_k$ with bounded eigenvalues is essentially the same, but is longer and more
 cluttered as it involves additional constants. Convergence and complexity results for other derivative-free optimization methods can be found in \cite{garmanjani2013smoothing,nesterov2017random,bauschke2015derivative}.

\subsection{Fixed Steplength Analysis}  We consider the method 
\begin{align}	\label{update}
	x_{k+1} = x_k - \alpha g_k ,
\end{align}
where  $g_k$  stands for a finite difference approximation to the gradient, $
      g_k = \nabla_h f(x_k),
$
or some other approximation; our treatment is general.
We define $e(x)$ to be the error in the gradient approximation, i.e.,
\begin{align}	\label{fd_err}
	g_k = \nabla \phi (x_k) + e(x_k).
\end{align}
 We should note the distinction between $e(x)$ and $\epsilon(x)$; the latter denotes the noise in the objective function  \eqref{objn}.

We introduce the following assumptions to establish the first convergence result.
\medskip

\begin{enumerate}
	\item[A.1] \textbf{(Strong Convexity of $\phi$)} The function $\phi$ (see \eqref{objn}) is twice continuously differentiable and there  exist positive constants $\mu$ and $L$ such that $\mu I \preceq \nabla^2 \phi(x) \preceq LI$ for all $x \in \mathbb{R}^n$. (We write $\phi^\star = \phi(x^\star)$, where $x^\star$ is the minimizer of $\phi$.)
	\item[A.2] \textbf{(Boundedness of Noise in the Gradient)} There is a constant $\bar{\epsilon}_g >0$ such that 
	\begin{align}  \label{bound-noise}
	 \| e(x) \| \leq \bar{\epsilon}_g \qquad\mbox{ for all} \ x \in \mathbb{R}^n.
	 \end{align}	
\end{enumerate}

Assumption A.2 is satisfied if the approximate gradient is given by forward  or central differences with the  value of $h$ given in \eqref{fd-noise} or \eqref{cd-noise}, provided that the error $\epsilon(x)$ in the evaluation of the  objection function value is bounded;  see \eqref{easy}.

We now establish linear convergence to a neighborhood of the solution.   Afterwards, we comment on the extension of this result to the case when a quasi-Newton iteration of the form $x_{k+1} = x_k - \alpha H_kg_k$ is used. 

\begin{theorem} \label{thm:constant_oc} Suppose that Assumptions A.1-A.2 hold.  Let $\{ x_k\}$ be the iterates generated by  iteration \eqref{update}, where $g_k$ is given by \eqref{fd_err}, and  the steplength satisfies
\begin{align}   \label{al-1_oc}
\alpha \leq 1/L .
\end{align}
Then for all $k$, 
\begin{align}    \label{1state}
\phi(x_{k+1})  - \left[\phi^\star + \frac{ \bar{\epsilon}_g^2}{2 \mu} \right] &\leq  (1 -  \alpha \mu ) \left(\phi(x_k) - \left[ \phi^\star  + \frac{\bar{\epsilon}_g^2 }{2 \mu} \right] \right) ,
\end{align}
where  $  \bar{\epsilon}_g$ is defined in \eqref{bound-noise}.
 \end{theorem}
 
\begin{proof} 
Since $\phi$ satisfies Assumption~A.1, we have by \eqref{fd_err},
\begin{align}
	\phi(x_{k+1}) &\leq \phi(x_k) - \alpha \nabla \phi(x_k)^T  g_k + \frac{\alpha^2  L}{2}\| g_k \|^2 \nonumber\\
		&= \phi(x_k) - \alpha  \nabla \phi(x_k)^T  \left(\nabla \phi(x_k) + e(x_k) \right) + \frac{\alpha^2  L}{2}\| \nabla \phi(x_k) + e(x_k) \|^2 \nonumber\\
		& = \phi(x_k) - \alpha \left(1 - \frac{\alpha  L}{2} \right) \| \nabla \phi(x_k)\|^2 - \alpha(1 - {\alpha  L})\nabla \phi(x_k)^Te(x_k) + \frac{\alpha^2  L}{2}\|e(x_k) \|^2\nonumber\\
		& \leq \phi(x_k) - \alpha \left(1 - \frac{\alpha  L}{2} \right) \| \nabla \phi(x_k)\|^2 + \alpha(1 - {\alpha  L}) \| \nabla \phi(x_k)\| \|e(x_k)\| + \frac{\alpha^2  L}{2}\|e(x_k) \|^2 .\nonumber  \\
                  &\leq \phi(x_k) - \alpha \left(1 - \frac{\alpha  L}{2} \right) \| \nabla \phi(x_k)\|^2 + \alpha(1 - {\alpha  L}) \left[\frac{1}{2}\| \nabla \phi(x_k)\|^2  +  \frac{1}{2} \|e(x_k) \|^2\right]  + \frac{\alpha^2  L}{2}\|e(x_k) \|^2 , \nonumber
\end{align}
where the last inequality follows from  the fact that $(\frac{1}{\sqrt{2}}  \| \nabla \phi(x_k)\| - \frac{1}{\sqrt{2}} \|e(x_k) \| )^2 \geq 0$ and the assumption  $\alpha L <1$.  Simplifying this expression, we have, for all $k$,  
\begin{align}
\phi(x_{k+1}) &\leq  \phi(x_k) -  \frac{\alpha  }{2}  \| \nabla \phi(x_k)\|^2     + \frac{\alpha }{2}  \|e(x_k) \|^2 . \label{gradrec2} 
\end{align}

Since $\phi$ is $\mu$-strongly convex, we can use the following relationship between the norm of the gradient squared, and the distance of the $k$-th iterate from the optimal solution,
\begin{align} \label{strongly}
     \| \nabla \phi(x_k) \|^2 \geq 2 \mu (\phi(x_k) - \phi^\star),
\end{align}
which together with \eqref{gradrec2} yields
\begin{align*}		
	\phi(x_{k+1})  
		&\leq  \phi(x_k) -  \alpha \mu  (\phi(x_k) - \phi^\star)   +  \frac{\alpha }{2}  \|e(x_k) \|^2 ,
\end{align*}
and by (\ref{bound-noise}),
\begin{align*}		
	\phi(x_{k+1})  - \phi^\star
		&\leq  (1 -  \alpha \mu ) (\phi(x_k) - \phi^\star)   +  \frac{\alpha }{2}  \bar{\epsilon}_g^2 .
\end{align*}
Hence,
\begin{align*}		
	\phi(x_{k+1})  - \phi^\star -\frac{ \bar{\epsilon}_g^2}{2 \mu} 
		&\leq  (1 -  \alpha \mu ) (\phi(x_k) - \phi^\star)   +  \frac{\alpha }{2}  \bar{\epsilon}_g^2  -\frac{ \bar{\epsilon}_g^2}{2 \mu}   \\  
		&=  (1 -  \alpha \mu ) (\phi(x_k) - \phi^\star)   + (\alpha \mu -1)   \frac{ \bar{\epsilon}_g^2}{2 \mu}    \\
                & =  (1 -  \alpha \mu ) \left(\phi(x_k) - \phi^\star  -\frac{\bar{\epsilon}_g^2 }{2 \mu} \right).
\end{align*}  
\end{proof}

We interpret the term $\left[ \phi^\star + \frac{ \bar{\epsilon}_g^2}{2 \mu} \right] $ in \eqref{1state} as the best value of the objective that can be achieved in the presence of noise. Theorem~\ref{thm:constant_oc} therefore establishes a $Q$-linear rate of convergence of $\{ \phi(x_k)\}$ to that value.

Another way of stating the convergence result embodied in Theorem~\ref{thm:constant_oc} is by applying the recursion to \eqref{1state}, i.e.,
\begin{align*}    
\phi(x_{k})  - \left[\phi^\star + \frac{ \bar{\epsilon}_g^2}{2 \mu} \right] &\leq  (1 -  \alpha \mu )^k \left(\phi(x_0) - \left[ \phi^\star  + \frac{\bar{\epsilon}_g^2 }{2 \mu} \right] \right), 
\end{align*}
so that 
\begin{align}    \label{rlinear}
\phi(x_{k})  - \phi^\star  &\leq  (1 -  \alpha \mu )^k \left(\phi(x_0) - \left[ \phi^\star  +\frac{\bar{\epsilon}_g^2 }{2 \mu} \right] \right)   + \frac{ \bar{\epsilon}_g^2}{2 \mu}.
\end{align}
This convergence result has a  similar flavor to that presented in \cite{nedic2001convergence} for the incremental gradient method using a fixed steplength; see also \cite[Section 4]{bottou2017optimization}. 
Note that \eqref{rlinear} is an $R$-linear convergence result and as such is weaker than \eqref{1state}.

It is possible to prove a similar result to Theorem~\ref{thm:constant_oc} for a general positive definite $H_k$, assuming bounds on $\|H_k\|$ and $\|H_k^{-1}\|$, as would be the case with a limited memory BFGS update. However, the provable convergence rate in this case would be closer to $1$, and the provable asymptotic objective value would be no smaller than the value in \eqref{1state}.  This is similar to the situation in optimization without noise, where the provable convergence rate for L-BFGS is no better than for steepest descent, even though L-BFGS is superior in practice. Since this more general analysis does not provide additional insights on the main topic of this paper, we have not included them here.

\subsection{Line Search Analysis}  
In  the literature of optimization of noisy functions, a number of convergence results have been established for algorithms that employ a fixed steplength strategy \cite{bertsekas2015convex,nedic2001convergence,bottou2017optimization}, but there has been little analysis of methods  that use a line search.

In this section, we present a convergence result for the iteration
\begin{align}  \label{algk}
    x_{k+1} = x_k - \alpha_k  g_k,
\end{align}
where the steplength $\alpha_k$ is computed by a backtracking  line search governed by the relaxed Armijo condition
\begin{align}	\label{armijo}
	f(x_k - \alpha_k  g_k) \leq f(x_k) - c_1 \alpha_k  g_k^T  g_k + 2 \epsilon_{\rm A}.
\end{align}
Here $c_1 \in (0,1)$  and $\epsilon_{\rm A}>0$ is a user specified  parameter whose choice is discussed later on.
 If a trial value $\alpha_k$ does not satisfy \eqref{armijo}, the new value is set to a (fixed) fraction $\tau <1$ of the previous value, i.e., $\alpha_k \leftarrow \tau \alpha_k$.

 Our analysis relies on the following additional assumption on the error in the objective function.
 
 \smallskip
\begin{enumerate}
	\item[A.3] \textbf{(Boundedness of Noise in the Function)} There is a constant $\bar{\epsilon}_f>0$ such that 
	\begin{align}		\label{bound-noise-f}
	| f(x) - \phi(x) |= |\epsilon(x)|  \leq \bar{\epsilon}_f \qquad\mbox{ for all} \ x \in \mathbb{R}^n.
	\end{align}
\end{enumerate}
The following result establishes linear convergence  to a neighborhood of the solution.

\begin{theorem} \label{thm:armijo} Suppose that Assumptions A.1-A.3 hold.  Let $\{ x_k\}$ be the iterates generated by iteration \eqref{algk}, where $g_k$ is given by \eqref{fd_err} and the step length $\alpha_k$  is the maximum value in $\{\tau^{-j}; j=0,1, \dots\}$ satisfying the relaxed Armijo condition \eqref{armijo} with $\epsilon_{ \rm A} > \bar{\epsilon}_f$ and $0<c_1< 1/2$.
Then, for any  $\beta \in \Big(0,\frac{1-2c_1}{1+2c_1}\Big]$, we have that 
\begin{align}    \label{east}
\phi(x_{k+1}) - \left[\phi^\star + \bar{\eta}  \right] \leq \rho \left(\phi(x_k) - \left[\phi^\star +  \bar{ \eta} \right]\right) ,\quad k=0,1, \ldots,
\end{align}
where 
\begin{align}   \label{rho-eta}
   \rho = 1 - \frac{2 \mu c_1 \tau  (1-\beta)^2}{L}  , \quad \bar{\eta} = \frac{1}{2 \mu \beta^2} \bar{\epsilon}_g^2  +\frac{L}{ \mu c_1 \tau  (1-\beta)^2}\left( \epsilon_{ \rm A} + \bar{\epsilon}_f \right),
\end{align}
 and $\,  \bar{\epsilon}_g$ and $\bar{\epsilon}_f$ are defined in \eqref{bound-noise} and \eqref{bound-noise-f}. Additionally, if $c_1 < 1/4$, we can choose $\beta = 1- 4c_1$, in which case,
\begin{align}   \label{rho-eta-simple}
   \rho = 1 - \frac{32 \mu \tau  c_1^3}{L}  , \quad \bar{\eta} = \frac{1}{2 \mu (1-4c_1)^2} \bar{\epsilon}_g^2  +\frac{L}{16 \mu  \tau  c_1^3}\left( \epsilon_{ \rm A} + \bar{\epsilon}_f \right) .
\end{align}
 \end{theorem}
 
 \begin{proof}
By equation \eqref{gradrec2} in the proof of Theorem \ref{thm:constant_oc}, if $\alpha \leq 1/L$, we have
\begin{equation}
\phi(x_{k}- \alpha g_k) \leq  \phi(x_k) -  \frac{\alpha  }{2}  \| \nabla \phi(x_k)\|^2     + \frac{\alpha }{2}  \|e(x_k) \|^2  ,
\end{equation}
which given assumption A.3 implies
\begin{equation}
f(x_{k} -\alpha g_k) \leq  f(x_k) -  \frac{\alpha  }{2}  \left( \| \nabla \phi(x_k)\|^2     -  \|e(x_k) \|^2 \right)  + 2 \bar{\epsilon}_f    \label{gradrec3} .
\end{equation}
Since we assume $\epsilon_{ \rm A} > \bar{\epsilon}_f$, it is clear from comparing \eqref{armijo} and \eqref{gradrec3}  that \eqref{armijo} will be satisfied for sufficiently small $\alpha$.
Thus, we have shown that the line search always finds a value of $\alpha_k$ such that \eqref{armijo} is satisfied.

In addition, we need to ensure that $\alpha_k$ is not too small when the iterates are far from $x^\star$.
To this end we define
\begin{equation}  \label{betadef}
\beta_k = \frac{\|e(x_k)\|}{\| \nabla \phi(x_k)\|},
\end{equation}
which together with \eqref{fd_err} gives
\begin{equation}
(1-\beta_k) \| \nabla \phi(x_k)\| \leq    \|g_k\| \leq (1+\beta_k) \| \nabla \phi(x_k)\| .  \label{gradbounds}
\end{equation}

Now, using  \eqref{betadef} and \eqref{gradbounds} in \eqref{gradrec3} we have
\begin{align*}
f(x_{k} -\alpha g_k) & \leq  f(x_k) -  \frac{\alpha  }{2}  (1-\beta_k^2)  \| \nabla \phi(x_k)\|^2    + 2 \bar{\epsilon}_f  \\
                              &\leq f(x_k) -  \frac{\alpha (1-\beta_k^2) }{2 (1+\beta_k)^2}   \| g_k \|^2    + 2 \bar{\epsilon}_f \\
                              &= f(x_k) -  \frac{\alpha (1-\beta_k) }{2 (1+\beta_k)}   \| g_k \|^2    + 2 \bar{\epsilon}_f.
\end{align*}
Since we assume that $\epsilon_{\rm A} >  \bar{\epsilon}_f$ and  $c_1< 1/2$, it is then clear that the Armijo condition \eqref{armijo} is satisfied for any $\alpha \leq 1/L$ if 
$ (1-\beta_k) /(1+\beta_k) \geq 2c_1$.  This is equivalent to requiring that
\begin{equation} \label{beta}
\beta_k \leq \beta \leq  \frac{1-2c_1}{1+2c_1} < 1 ,
\end{equation}
where $\beta$ is an arbitrary positive constant satisfying (\ref{beta}). 

\medskip\noindent
\textbf{Case 1}.
We now select such a value $\beta$ 
and refer to iterates such that $\beta_k \leq \beta$ as Case 1 iterates. 
Thus, for these iterates any $\alpha \leq 1/L$ satisfies the relaxed Armijo condition \eqref{armijo}, and since we find $\alpha_k$ using a constant  backtracking factor of $\tau<1$, we have that 
$\alpha_k > \tau/L$. Therefore, using Assumption~A.3 and (\ref{gradbounds}) we have
%
\begin{align}
	\phi(x_k - \alpha_k g_k) &\leq \phi(x_k) - c_1 \alpha_k \| g_k\|^2 + 2\epsilon_{\rm A}+ 2\bar{\epsilon}_f \nonumber\\
	& \leq \phi(x_k) - \frac{c_1 \tau (1-\beta)^2}{L} \| \nabla \phi(x_k)\|^2 + 2\epsilon_{\rm A}+ 2\bar{\epsilon}_f.    \label{phi_prog_c1} 
\end{align}
Expression \eqref{phi_prog_c1} measures the reduction in $\phi$ at iterates belonging to Case 1.


\medskip\noindent
\textbf{Case 2}. We now consider iterates that do not satisfy the conditions of Case 1, namely,  iterates for which $\beta_k >\beta$, or equivalently,
\begin{align}	\label{case2}
	\|e(x_k) \| > \beta \|\nabla \phi(x_k)\|.
\end{align}

We have shown that the relaxed  Armijo condition \eqref{armijo} is satisfied at every iteration of the algorithm.
 Using as before Assumption~A.3, we deduce from \eqref{armijo} that
\begin{align}
	\phi(x_k - \alpha_k g_k) & \leq \phi(x_k) - c_1 \alpha_k \| g_k\|^2 +  2\epsilon_{\rm A} + 2 \bar{\epsilon}_f \nonumber\\
	 & \leq \phi(x_k)  + 2 \epsilon_{ \rm A} + 2 \bar{\epsilon}_f. \label{rec_case2}
\end{align}
We now add and subtract  $c_1( \tau/L) (1-\beta)^2   \| \nabla \phi(x_k)\|^2$ from the right hand side of this relation and recall \eqref{case2}, to obtain
\begin{align}
	\phi(x_k - \alpha_k g_k) & \leq \phi(x_k) -  \frac{c_1 \tau (1-\beta)^2 }{L} \|\nabla \phi(x_k)\|^2   +\frac{c_1 \tau (1-\beta)^2}{L} \| \nabla \phi(x_k)\|^2+ 2\epsilon_{\rm A} +2 \bar{\epsilon}_f \nonumber\\
	& \leq \phi(x_k) -  \frac{c_1 \tau (1-\beta)^2}{L} \| \nabla \phi(x_k)\|^2  +\frac{c_1 \tau (1-\beta)^2 }{L\beta^2} \|e(x_k)\|^2  + 2\epsilon_{ \rm A} +2 \bar{\epsilon}_f  \nonumber\\
	& \leq \phi(x_k) -  \frac{c_1 \tau (1-\beta)^2}{L} \| \nabla \phi(x_k)\|^2  +\frac{c_1 \tau (1-\beta)^2 }{L\beta^2} \bar{\epsilon}_g^2  + 2\epsilon_{ \rm A} +2 \bar{\epsilon}_f  \nonumber\\
	& = \phi(x_k) - \frac{c_1 \tau (1-\beta)^2}{L}  \| \nabla \phi(x_k)\|^2 + \eta \label{phi_prog_c2} ,
\end{align}
where 
\begin{align}	\label{eta_neigh}
\eta =\frac{c_1 \tau (1-\beta)^2 }{L\beta^2} \bar{\epsilon}_g^2  + 2\epsilon_{\rm A} +2 \bar{\epsilon}_f. 
\end{align}
Equation \eqref{phi_prog_c2} establishes a recursion in $\phi$ for the iterates in Case 2. 

\smallskip
Now we combine the results from the two cases: \eqref{phi_prog_c1} and  \eqref{phi_prog_c2}. Since the term multiplying $\| \nabla \phi(x_k)\|^2$ is the same in the two cases, and since $\eta \geq 2\epsilon_{ \rm A} + 2 \bar{\epsilon}_f$, we have that for all $k$
\begin{align}	\label{rec}
\phi(x_{k+1}) \leq \phi(x_k) - \frac{c_1 \tau  (1-\beta)^2}{L} \| \nabla \phi(x_k)\|^2 + \eta.
\end{align}
 Subtracting $\phi^\star$ from both sides of \eqref{rec}, and using the strong convexity condition \eqref{strongly}, gives
\begin{align}	\label{rhodef}
\phi(x_{k+1}) - \phi^\star \leq \underbrace{\left(1 - \frac{2 \mu c_1 \tau  (1-\beta)^2}{L}  \right)}_{\text{$\rho$}}(\phi(x_k) - \phi^\star)  + \eta. 
\end{align}
Clearly $0<\rho<1$, since the quantities $2c_1,\mu/L, \tau$ and $1-\beta$ are all less than one. We have thus shown that for all $k$
\begin{align}
\phi(x_{k+1}) - \phi^\star \leq \rho(\phi(x_k) - \phi^\star)  + \eta. 
\end{align}
Subtracting $\eta/(1- \rho)$ from both sides, it follows that
\begin{align*}
	\phi(x_{k+1}) - \phi^\star - \frac{\eta}{1- \rho} &\leq \rho(\phi(x_k) - \phi^\star)  + \eta  - \frac{\eta}{1- \rho}\nonumber \\
	& = \rho(\phi(x_k) - \phi^\star)  -  \frac{\rho \eta}{1- \rho} \nonumber \\
	& = \rho \left(\phi(x_k) - \phi^\star -  \frac{ \eta}{1- \rho} \right)   ,
	\end{align*}
	and thus
	\begin{equation} \label{nubardecrease} 
\phi(x_{k+1}) - \phi^\star - \bar{\eta} \leq \rho \left(\phi(x_k) - \phi^\star -  \bar{ \eta} \right) , \quad k=0,1,\ldots, 
\end{equation}
where $\bar{\eta}=\eta /(1-\rho)$. From \eqref{eta_neigh}, \eqref{rhodef}, we have that
\begin{align}
\bar{\eta}  
	&=  \frac{L}{2 \mu c_1 \tau  (1-\beta)^2} \left( \frac{c_1 \tau (1-\beta)^2 }{L\beta^2} \bar{\epsilon}_g^2  + 2\epsilon_{ \rm A} +2 \bar{\epsilon}_f \right)       \nonumber \\
              & = \frac{1}{2 \mu \beta^2} \bar{\epsilon}_g^2  +\frac{L}{ \mu c_1 \tau  (1-\beta)^2}\left( \epsilon_{ \rm A} + \bar{\epsilon}_f \right)  .      \label{nubarvalue}
\end{align}
This establishes \eqref{east} and \eqref{rho-eta}.

We can obtain simpler expressions for $\rho$ and $\bar \eta$ by making a particular selection of $\beta$. Recall that this parameter is required to satisfy \eqref{beta} so that \eqref{nubardecrease} holds for all $k$. In the case when $c_1 \in (0, 1/4]$, the choice $\beta = 1-4c_1$ will satisfy \eqref{beta} since
\begin{align*}
1-4c_1-\frac{1-2c_1}{1+2c_1} = \frac{1-4c_1 +2c_1 -8c_1^2 -(1-2c_1)}{1+2c_1} <0.
\end{align*}
 Substituting this value of $\beta$ in the definition of $\rho$ (see \eqref{rhodef})  and in  \eqref{nubarvalue} gives
\begin{align*}
\rho=1-\frac{2 \mu c_1 \tau  (1-\beta)^2}{L} 
   =1-\frac{32 \mu  \tau  c_1^3}{L},
\end{align*}
and 
\begin{align*}
\bar{\eta}  &=\frac{1}{2 \mu \beta^2} \bar{\epsilon}_g^2  +\frac{L}{\mu c_1 \tau  (1-\beta)^2}\left( \epsilon_{ \rm A} + \bar{\epsilon}_f \right) 
   =\frac{1}{2 \mu (1-4c_1)^2} \bar{\epsilon}_g^2  +\frac{L}{ \mu  \tau  16c_1^3}\left( \epsilon_{ \rm A} + \bar{\epsilon}_f \right),
\end{align*}
which gives \eqref{rho-eta-simple}.
\end{proof}


As with \eqref{rlinear} there is a different way of stating the convergence result. Applying \eqref{nubardecrease} recursively and then moving the constant term to the right-hand-side of the expression yields
\begin{align*}
\phi(x_{k}) - \phi^\star \leq \rho^k \left(\phi(x_0) - \left[\phi^\star +  \bar{\eta} \right]\right) + \bar{\eta}.
\end{align*}

Application of Theorem~\ref{thm:armijo} requires a choice of the parameter $\beta$ that affects the convergence constant $\rho$ and the level of accuracy $\bar \eta$.  One of the most intuitive  expressions we could find was \eqref{rho-eta-simple}, which was obtained by assuming that $c_1<1/4$ and choosing $\beta= 1-4c_1$. This value is reasonable provided $c_1$ is not too small. There are other ways of choosing $\beta$ when $c_1$ is very small; in general the theorem is stronger if $\beta$ is not chosen close to zero.

In the case where  $\nabla \phi(x)$ is estimated by a forward difference approximation we can further distill our error estimate, since $\bar{\epsilon}_f$ is the source of all noise. Using A.3, it is easy to show that
\begin{align}  
	 \| e(x_k) \| =  \|g_k - \nabla \phi(x_k)\| & \leq  \, Lh/2+ 2\bar{\epsilon}_f /h   \nonumber \\
              & =  \, 2 \sqrt{L\bar{\epsilon}_f},  \label{easy}
\end{align}
where the equality follows by substituting $h=2\sqrt{\bar{\epsilon}_f/L}$, which is the value that minimizes  the right hand side of the  inequality. Therefore, by \eqref{bound-noise} we can assume that $\bar{\epsilon}_g = 2\sqrt{L\bar{\epsilon}_f}$, and the asymptotic accuracy  can be estimated as
\begin{align}
\bar{\eta} 
               &= \frac{1}{2 \mu \beta^2} \bar{\epsilon}_g^2  +\frac{L}{ \mu c_1 \tau  (1-\beta)^2}\left( \epsilon_{ \rm A} + \bar{\epsilon}_f \right)    \label{likefixed}       \\
              & = \frac{2L}{ \mu \beta^2} \bar{\epsilon}_f  +\frac{L}{ \mu c_1 \tau  (1-\beta)^2}\left( \epsilon_{ \rm A} + \bar{\epsilon}_f \right) \nonumber \\
               &= \frac{L\bar{\epsilon}_f}{\mu} \left( \frac{2}{  \beta^2}   +\frac{1+\theta}{ c_1 \tau  (1-\beta)^2}  \right) ,   \label{epsfonly}
\end{align}
where in the last line we have written $\epsilon_{ \rm A} =\theta \bar{\epsilon}_f$, for some $\theta >1$. If we choose $\beta=1-4c_1$ when $c_1<1/4$, then
\begin{align}
\bar{\eta} = \frac{L\bar{\epsilon}_f}{\mu} \left( \frac{2}{  (1-4c_1)^2}   +\frac{1+\theta}{ 16  \tau c_1^3}  \right). \label{epsfonly2}
\end{align}

Equations \eqref{epsfonly} and \eqref{epsfonly2} show that the asymptotic accuracy level $\bar{\eta}$ is proportional to the bound in the error in the objective function \eqref{bound-noise-f} times the condition number. 
Note also that \eqref{likefixed} is comparable (but larger) than the accuracy level $ \bar{\epsilon}_g^2 /2 \mu$ in Theorem~\ref{thm:constant_oc}. This is not surprising as it reflects the fact that, although the line search method can take steps that are much larger than $1/L$,  the theory only uses the fact that $\alpha$ cannot be much smaller than $1/L$.

In the analysis presented in this section, we have chosen to analyze a backtracking line search  rather than one satisfying the Armijo-Wolfe conditions \eqref{eq:armijo}-\eqref{eq:wolfe}  because in the strongly convex case, the curvature condition \eqref{eq:wolfe} is less critical, and because this simplifies the analysis. It would, however, be possible to do a similar analysis for the Armijo-Wolfe conditions, but the algorithm and analysis would be more complex. (We should also mention that we could remove the parameter $\epsilon_{ \rm A}$ from the relaxed Armijo condition \eqref{armijo} at the cost of making the analysis more complex.)

An extension of Theorem~\ref{thm:armijo} to a quasi-Newton iteration with positive definite $H_k$ could be proved, but for the reasons discussed above, we have limited our analysis to the case $H_k=I$.

\section{Numerical Experiments on Noisy Functions}
\label{numerical}

In this section we present numerical results comparing the performance of the finite difference L-BFGS method ({\tt FDLM}, Algorithm \ref{alg1}), the  function interpolating trust region method (FI) described in \cite{conn2009introduction}, which we denote by  {\tt DFOtr}, and the direct search method \texttt{NOMAD} \cite{abramson2011nomad,audet2006mesh}. We selected 49 nonlinear optimization test problems from the Hock and Schittkowski collection \cite{schttfkowski1987more}. The distribution of problems, in terms of their dimension $n$, is given in Table \ref{hs_probs}. 

\begin{table}[H]
\centering
\caption{\small Dimensions of Hock-Schittkowski problems tested}  
\label{hs_probs} 
\small
\begin{tabular}{|c||ccccccccccc|}  \hline 
$n$                                                          & 2  & 3 & 4 & 5 & 6 & 9 & 10 & 20 & 30 & 50 & 100 \\ \hline \hline
\begin{tabular}[c]{@{}c@{}}Number of  Problems\end{tabular} & 15 & 6 & 5 & 2 & 5 & 1 & 4   & 3  & 3  & 3  & 2 \\
\hline
\end{tabular}
\end{table}

A limit of $100 \times n$ function evaluations and 30 minutes of {\sc cpu} time is given to each method. The FDLM method  also terminates if one of the following two conditions holds: i) $| f_{\tiny \rm{MA}} (x_k)- f(x_k)| \leq 10^{-8} \max\{1,|f_{\tiny \rm{MA}}(x_k)|\}$, where $f_{\tiny \rm{MA}}(x_k)$ is defined in \eqref{mav};  or ii) $\| \nabla_{h}f(x_k) \|  \leq 10^{-8}$. The FI method terminates if the trust region radius satisfies $ \Delta_k  \leq 10^{-8}$. The very small tolerance $10^{-8}$ was chosen to display the complete evolution of the runs. We ran   \texttt{NOMAD}  with default parameters \cite{abramson2011nomad}.

We experimented with 4 different types of noise: (i) stochastic additive, (ii) stochastic multiplicative, (iii) deterministic additive, and (iv) deterministic multiplicative, and for each type of noise we consider 4 different noise levels. 
Below, we show a small sample of results for the first two types of noise. 
\smallskip
\paragraph{Stochastic Additive Noise}
The objective function has the form $f(x) = \phi(x)+  \epsilon(x)$,
where $\phi$ is a smooth function and $ \epsilon(x)$ is a uniform random variable, i.e.,
\begin{equation}		\label{stoch_noise}
	  \epsilon(x) \sim U(-\xi,\xi).
\end{equation}
We investigate the behavior of the methods for  noise levels corresponding to $\xi \in \{10^{-8}, 10^{-6}, 10^{-4}, 10^{-2} \}$. In Figure~\ref{stoch_add_exp} we report results for the 4 problems studied in Section~\ref{sec:det_case}, namely {\tt s271}, {\tt s334}, {\tt s293} and  {\tt s289}, for 2 different noise levels ($10^{-8}$ and $10^{-2}$). The figure plots the optimality gap ($f(x_k) - \phi^\star$) versus the number of function evaluations. (For all test problems, $\phi^\star$ is known.) We  note that $\phi^\star =0$ for problems {\tt s271}, {\tt s293}, {\tt s289}. 
The first problem, {\tt s271} is quadratic, which is benign for {\tt DFOtr}, which  terminates as soon as a fully quadratic model is constructed.

\begin{figure}[htp]
\begin{centering}

\includegraphics[width=0.24\textwidth]{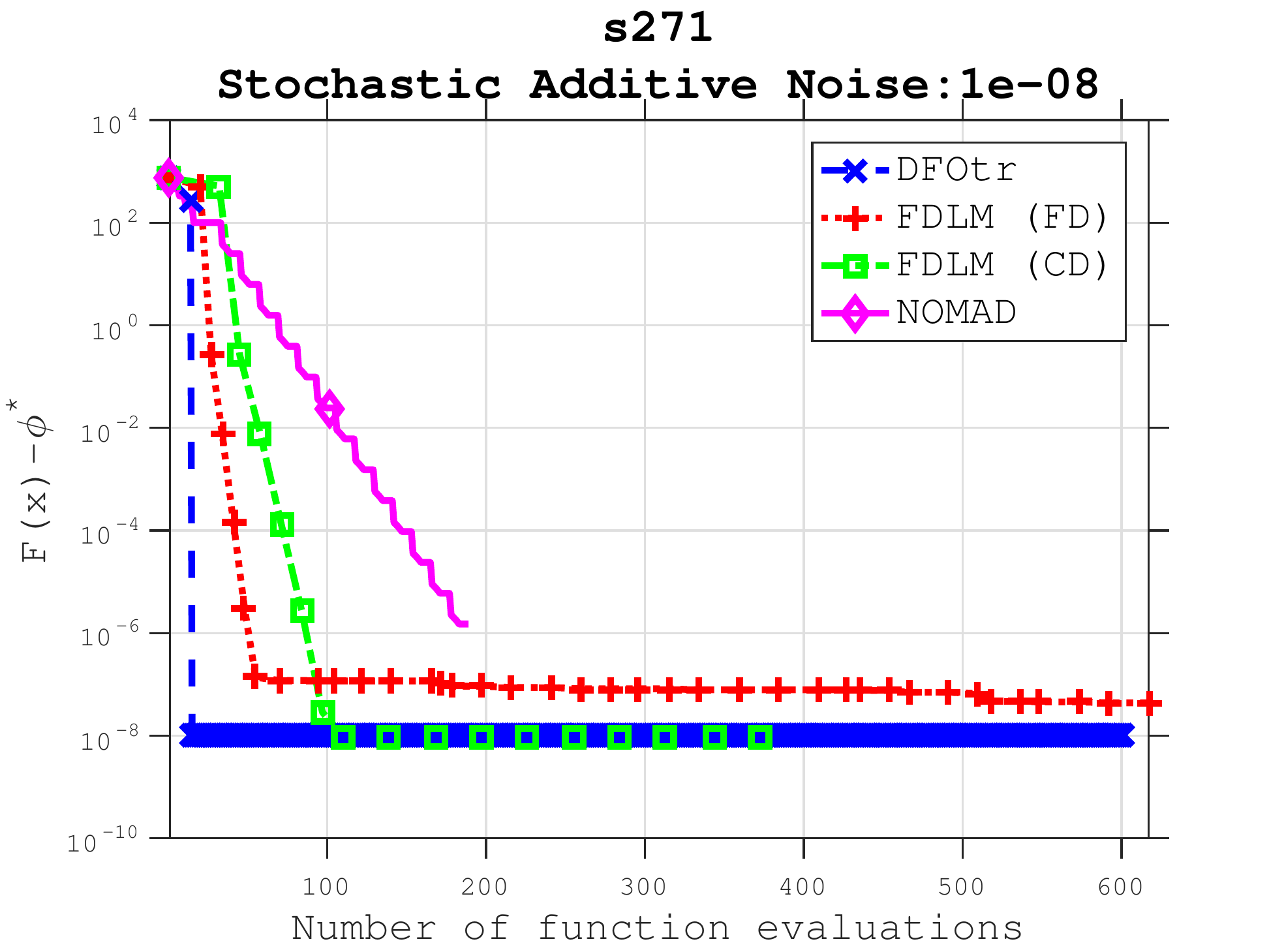}
\includegraphics[width=0.24\textwidth]{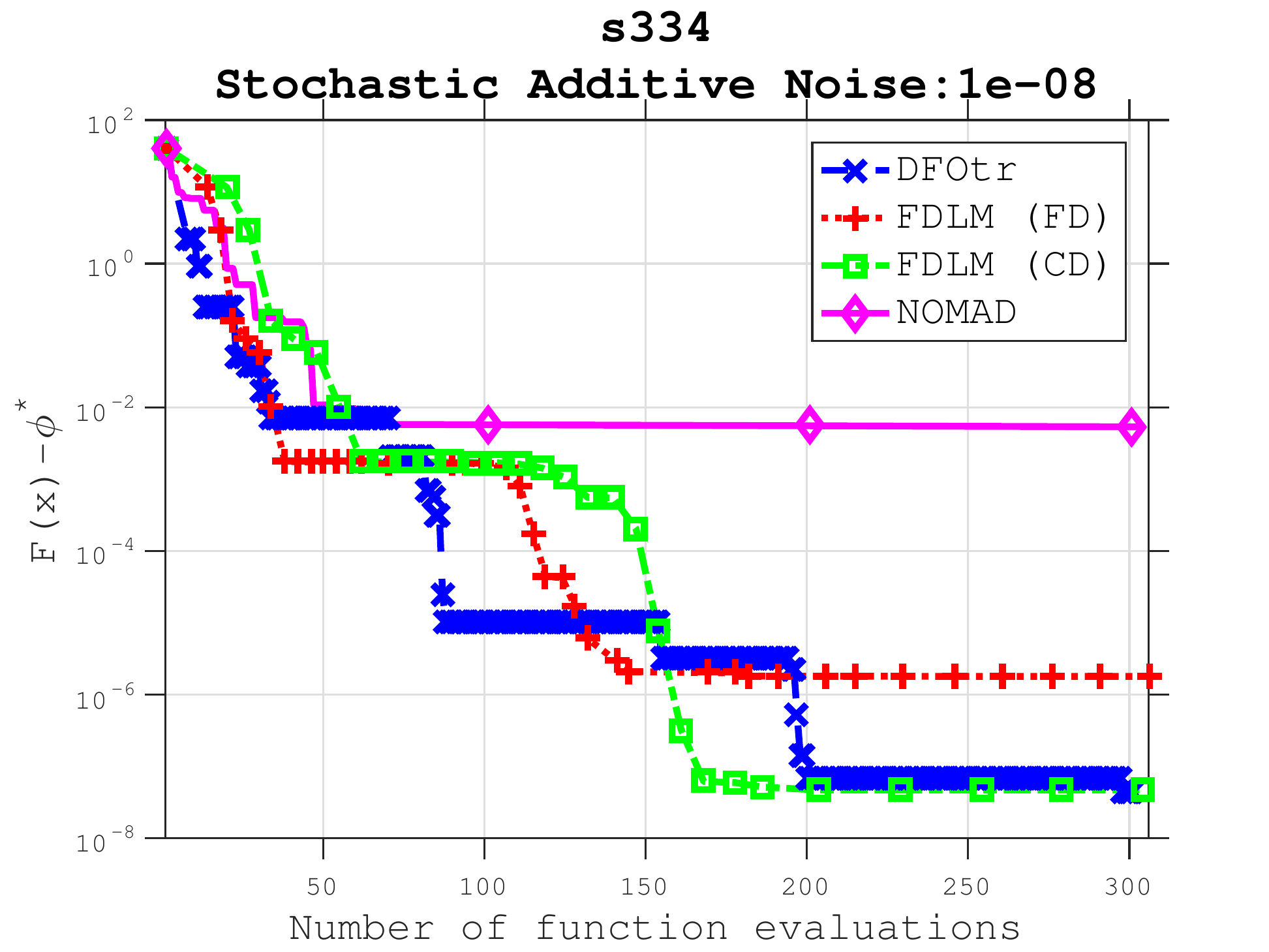}
\includegraphics[width=0.24\textwidth]{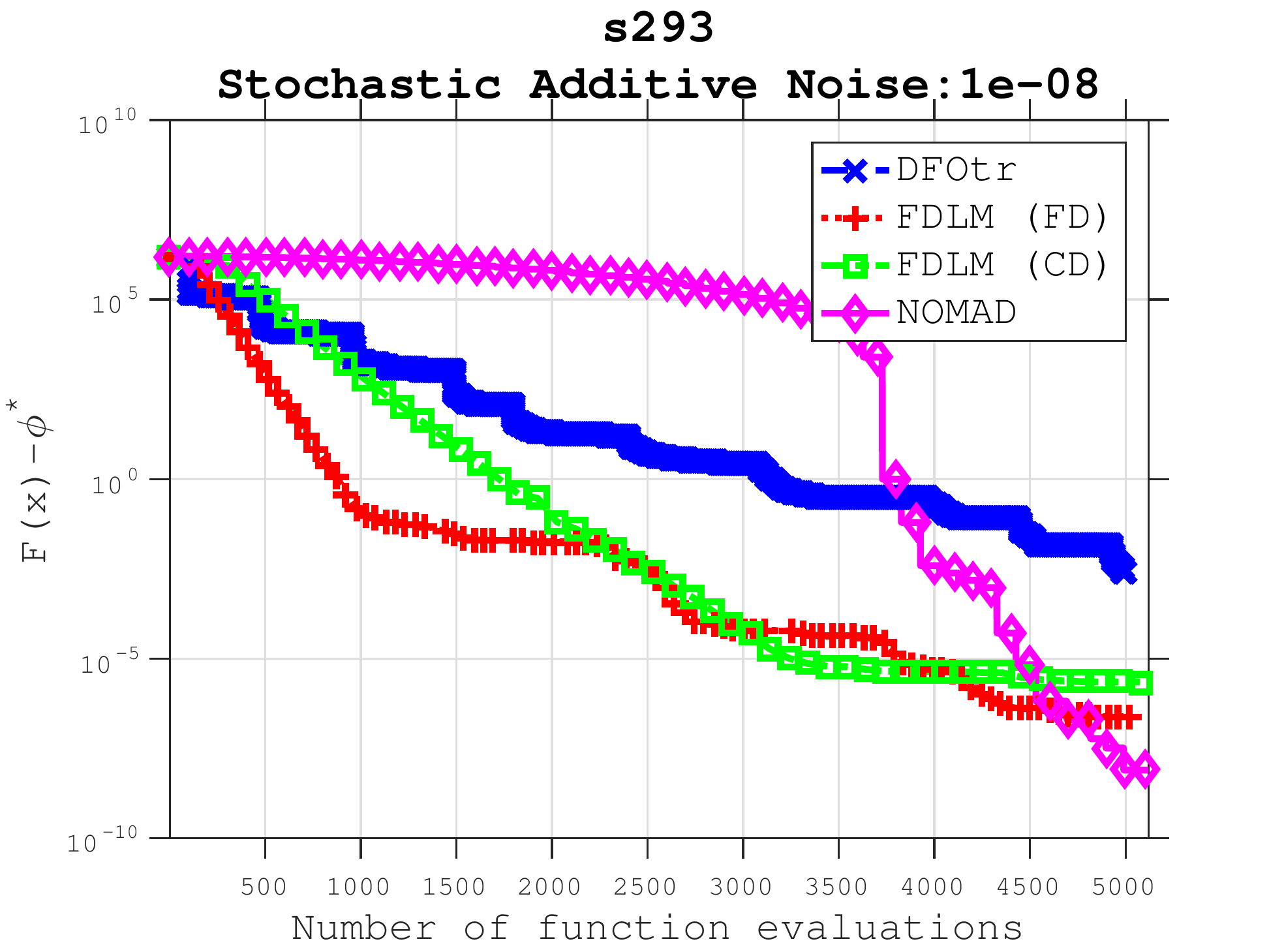}
\includegraphics[width=0.24\textwidth]{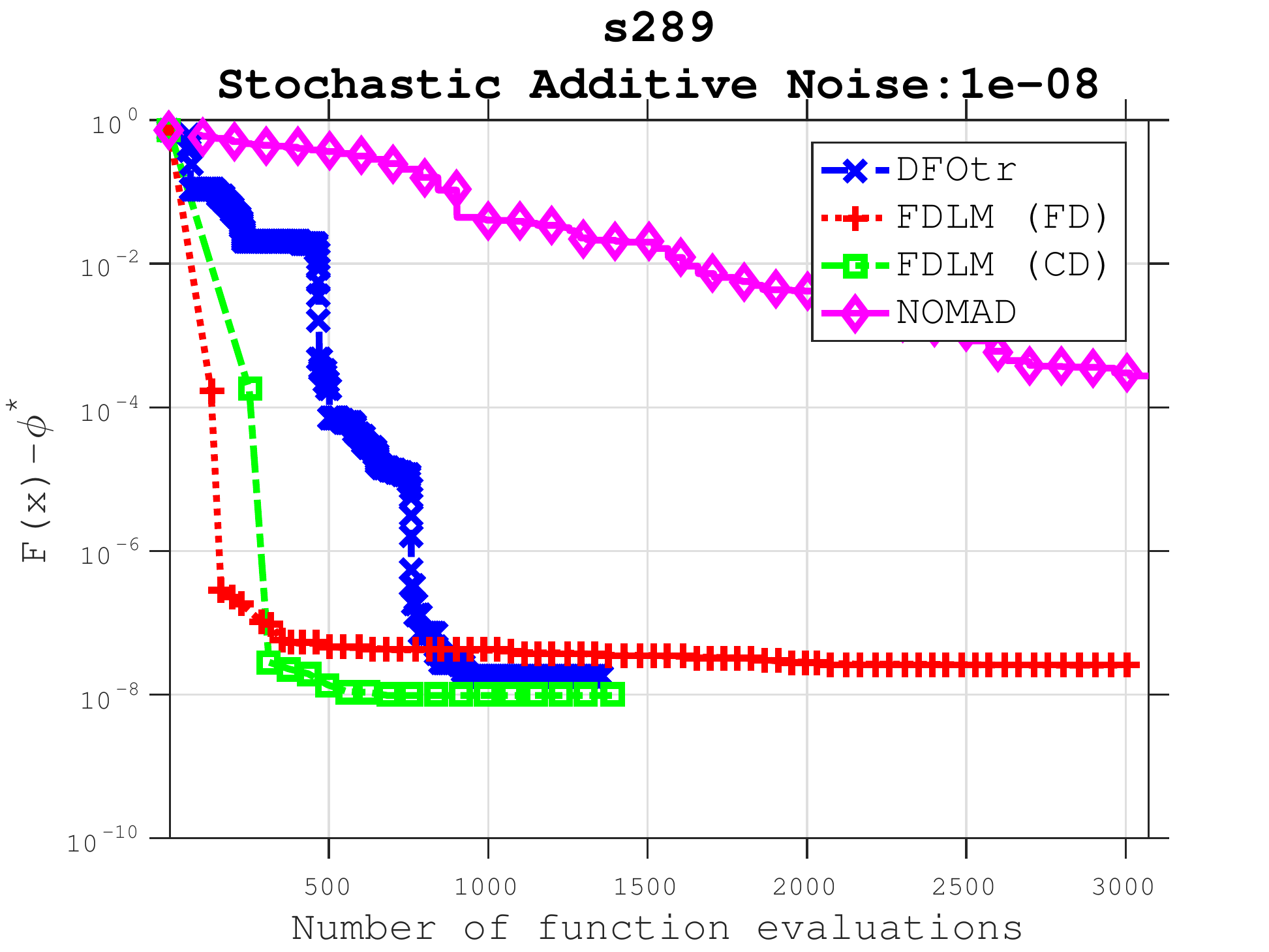}

\includegraphics[width=0.24\textwidth]{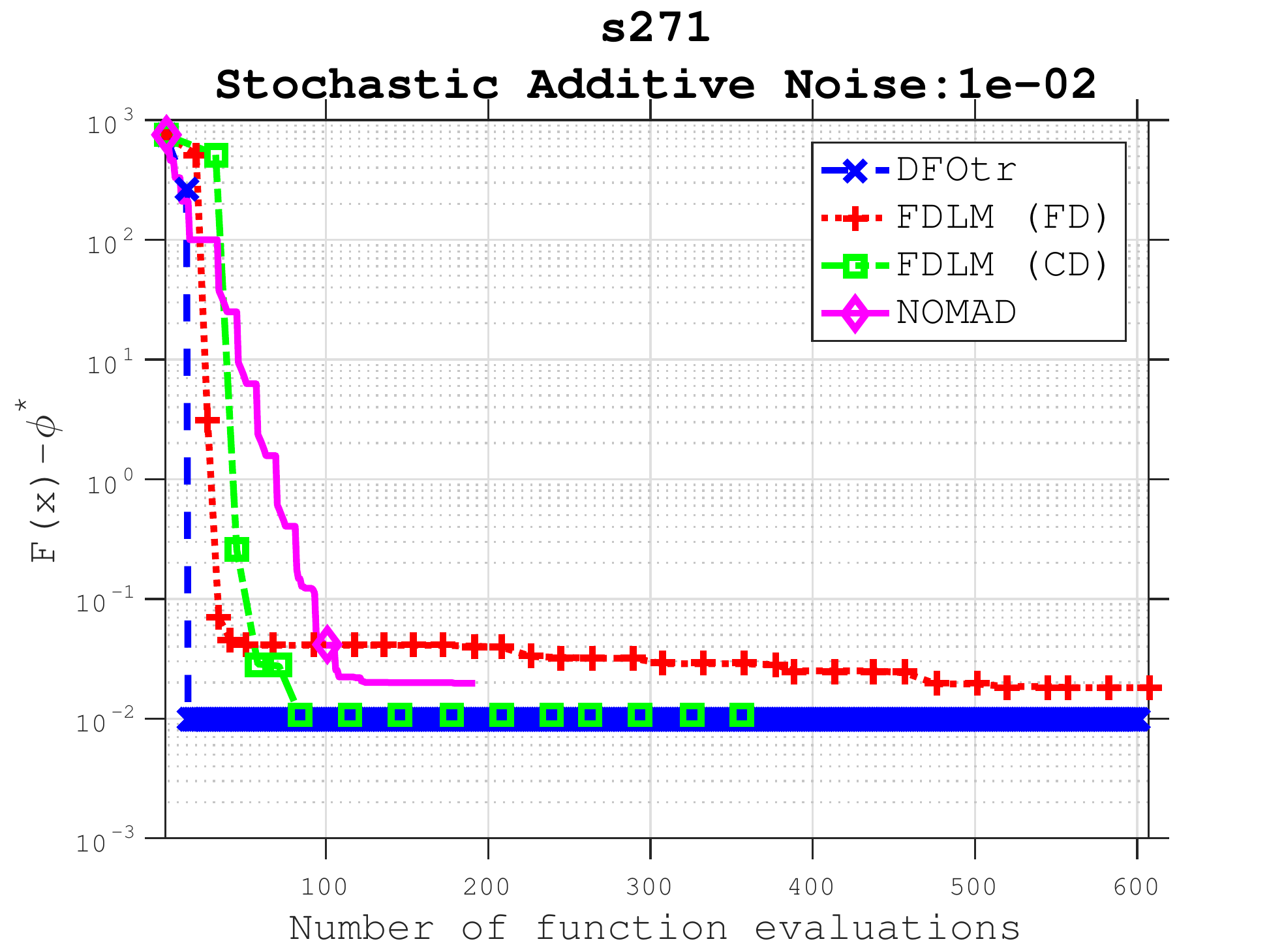}
\includegraphics[width=0.24\textwidth]{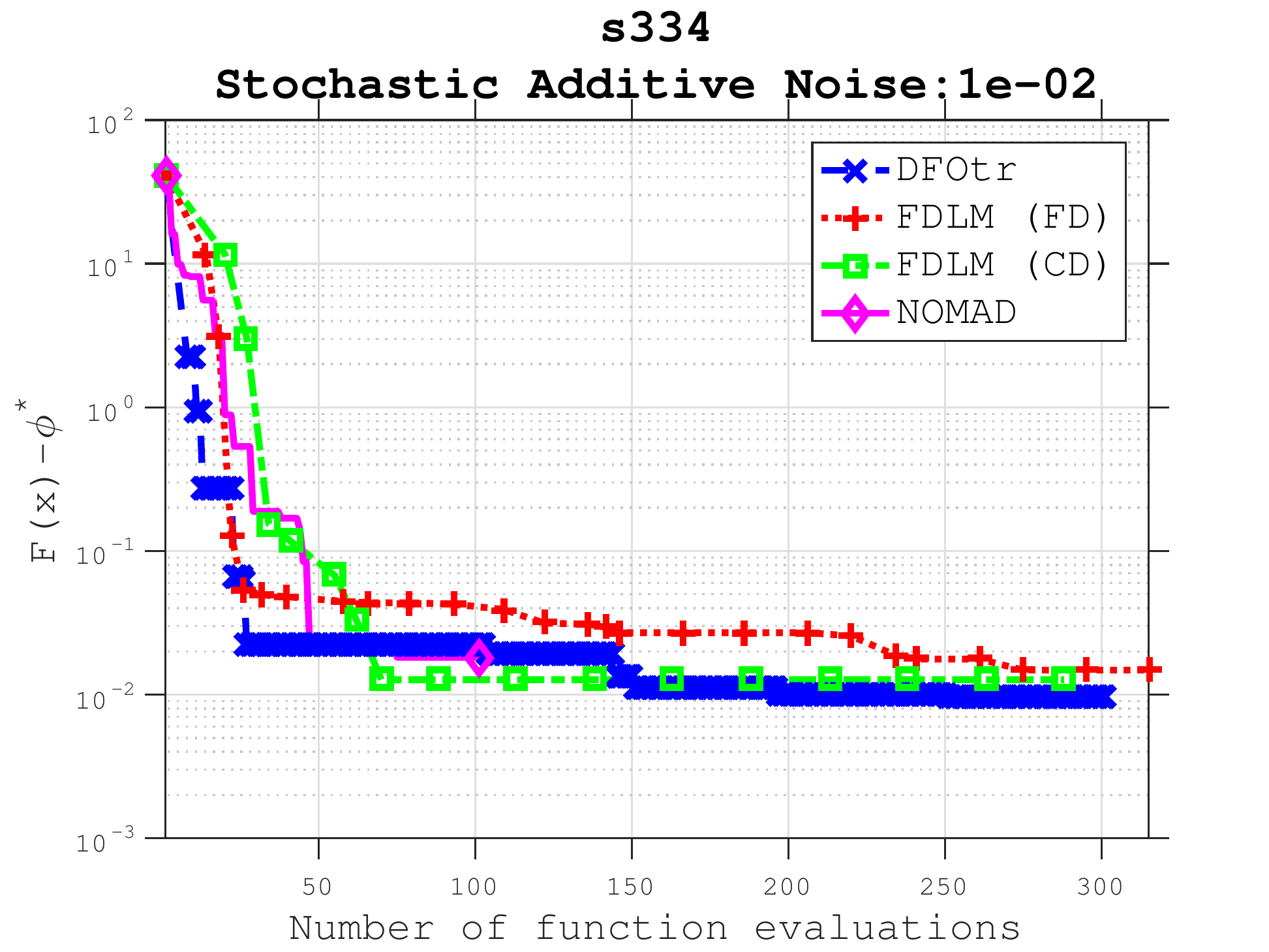}
\includegraphics[width=0.24\textwidth]{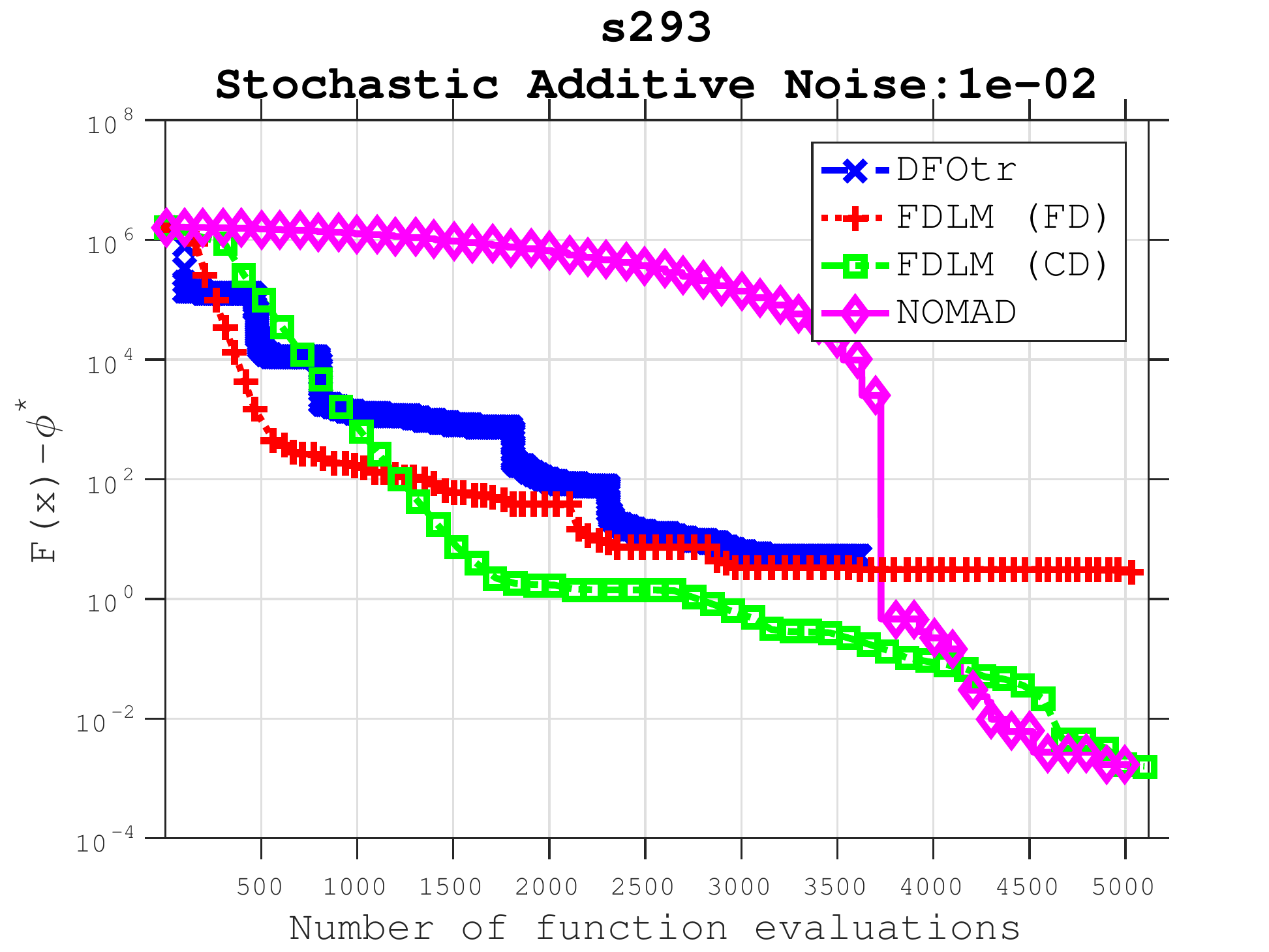}
\includegraphics[width=0.24\textwidth]{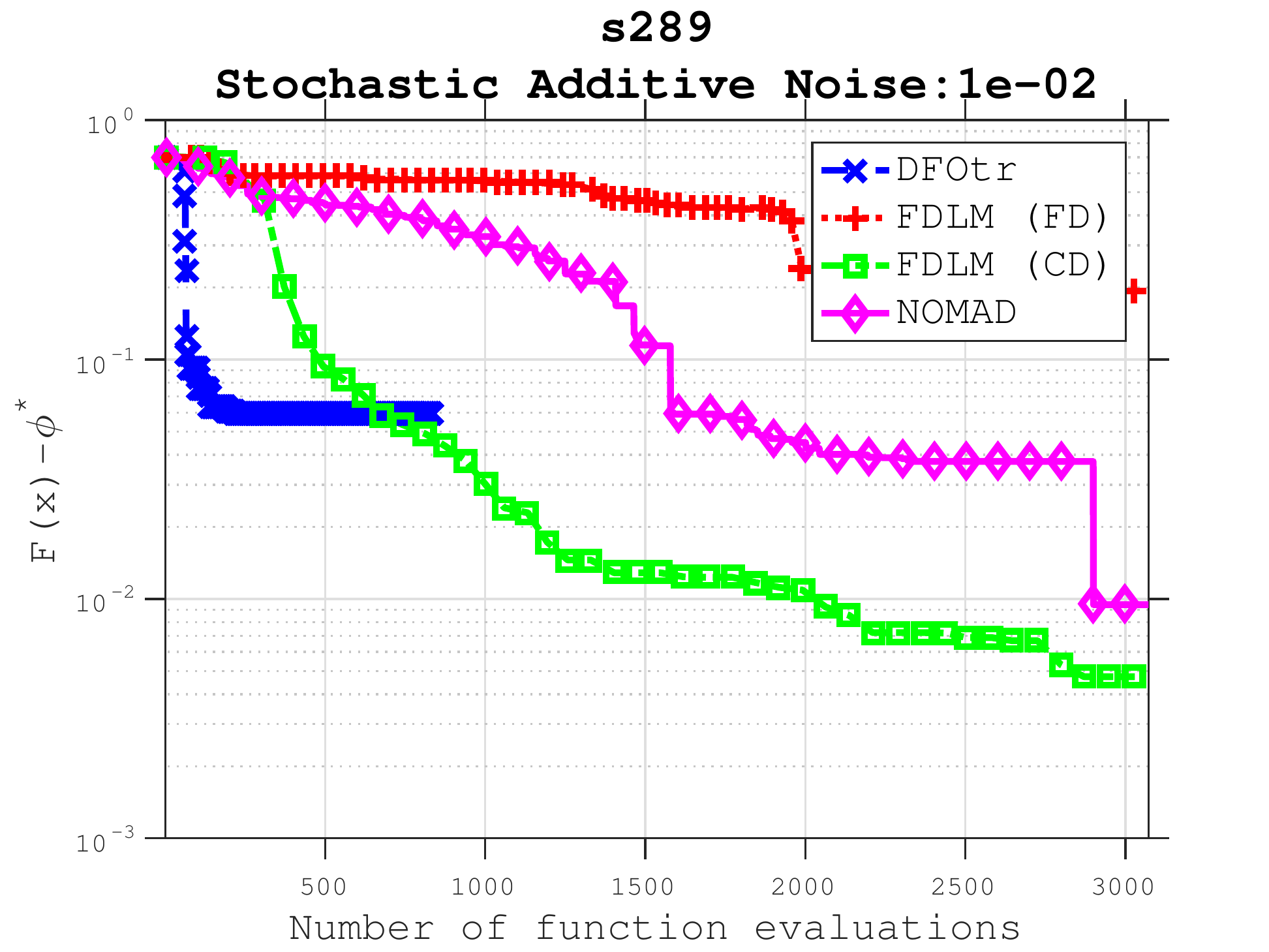}

\par\end{centering}
\caption{\small \texttt{Stochastic Additive Noise.} Performance of the function interpolating trust region method ({\tt DFOtr}) described in \cite{conn2009introduction}, {\tt NOMAD}\cite{audet2006mesh,abramson2011nomad} and the finite difference L-BFGS method ({ \tt FDLM}) using forward or central differences. The figure plots results for 4 problems from the Hock-Schittkowki collection \cite{schttfkowski1987more}, for 2 different noise levels. Each column represents a different problem and each row a different noise level. 
}
\label{stoch_add_exp}
\end{figure}

\smallskip
\paragraph{Stochastic Multiplicative Noise} 
The objective  has the form $ f(x)= \phi(x)(1+ \hat \epsilon(x))$,
where $\phi$ is smooth and $\hat \epsilon(x)$ is the stochastic noise defined as in \eqref{stoch_noise}. 
We can write the objective in the additive form $f(x) = \phi(x) + \epsilon(x)$, where  $\epsilon(x)= \phi(x) \hat \epsilon(x)$ varies with $x$, and since $\phi^\star =0$ for problems {\tt s271}, {\tt s293} and {\tt s289}, the noise term $\epsilon(x) $ decays to zero as $x_k$ approaches the solution; this is not the case for problem {\tt s334}.
The results are given in Figure~\ref{stoch_mult_exp}.

\begin{figure}[]
\begin{centering}

\includegraphics[width=0.24\textwidth]{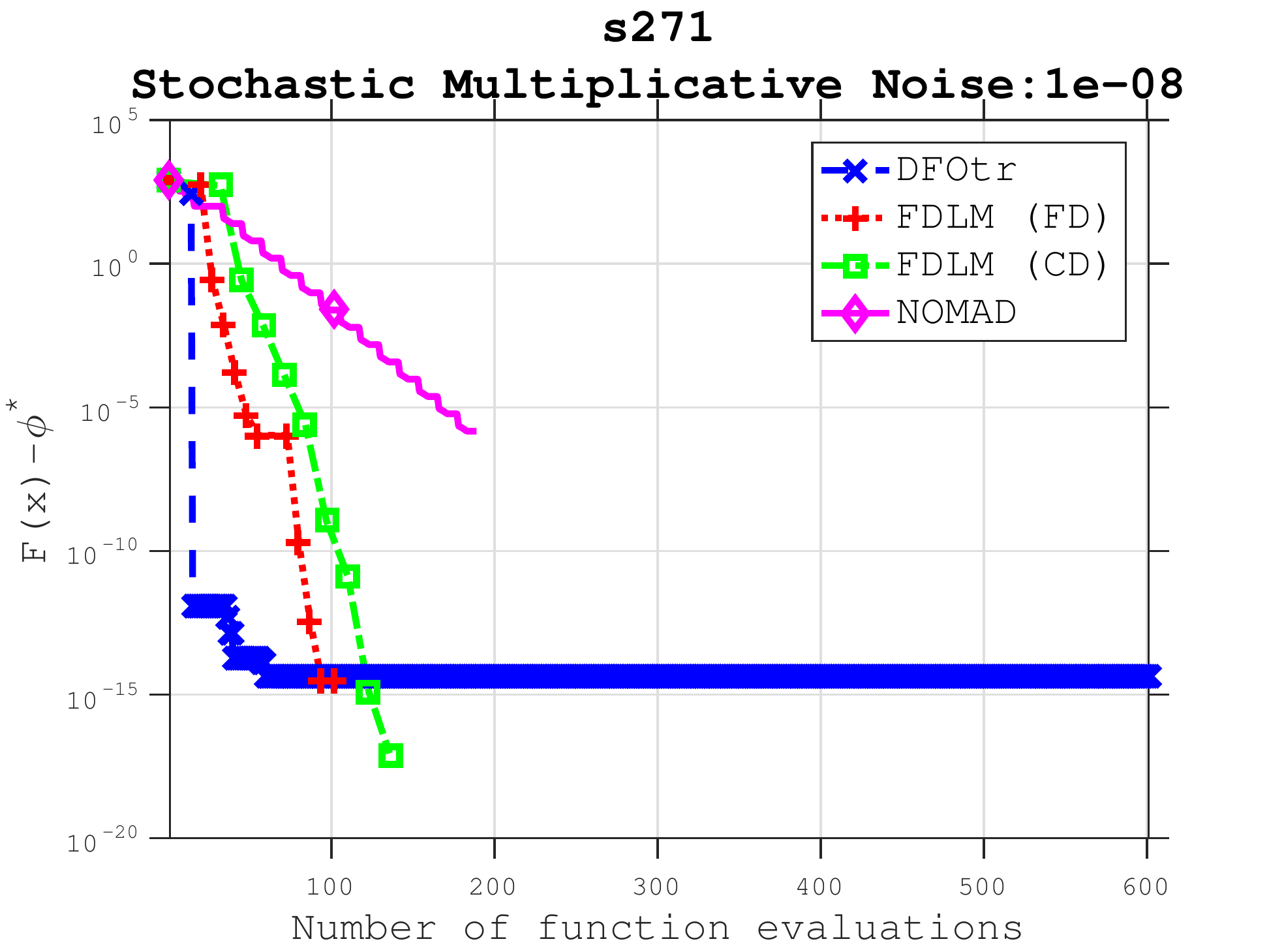}
\includegraphics[width=0.24\textwidth]{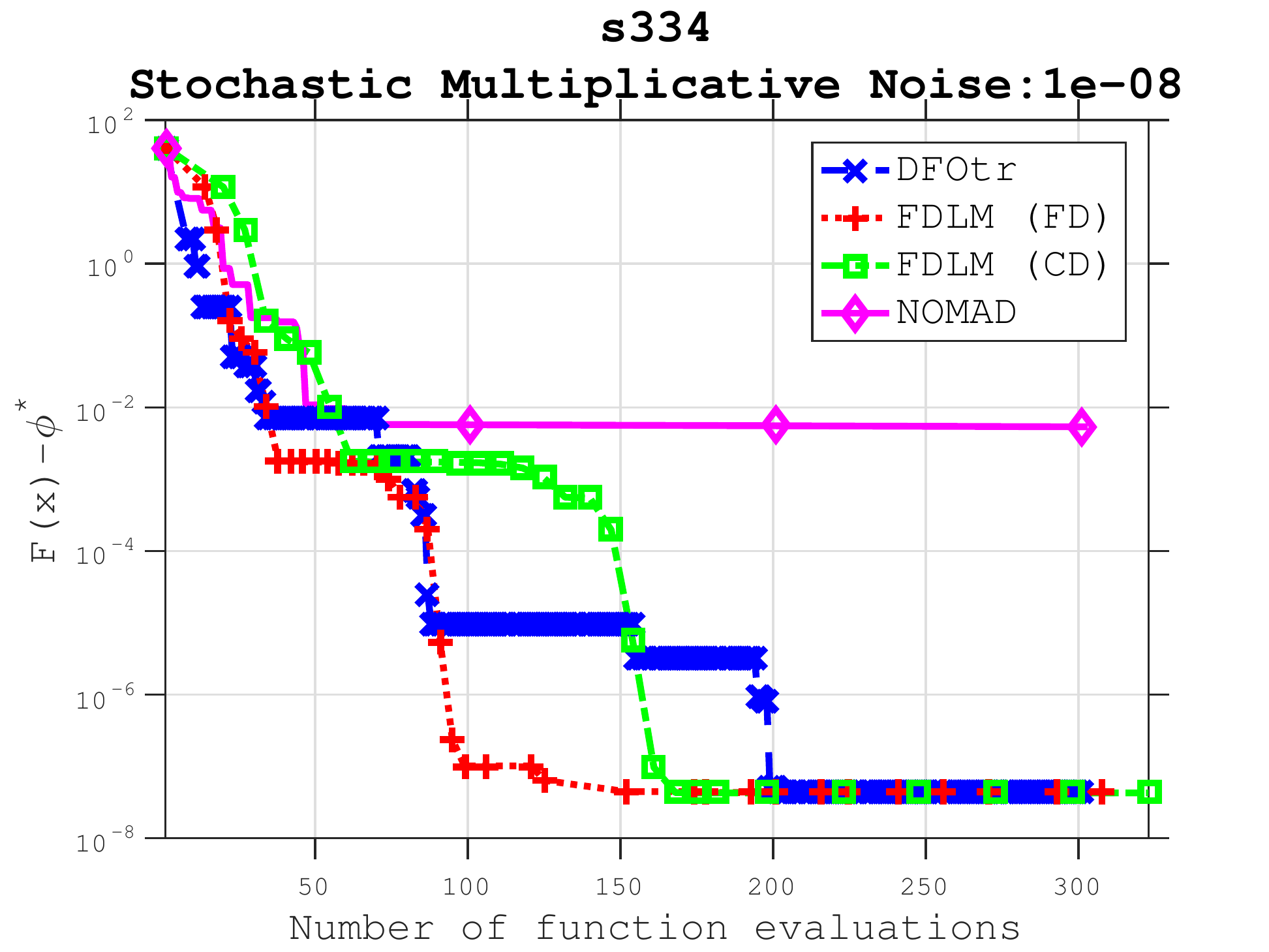}
\includegraphics[width=0.24\textwidth]{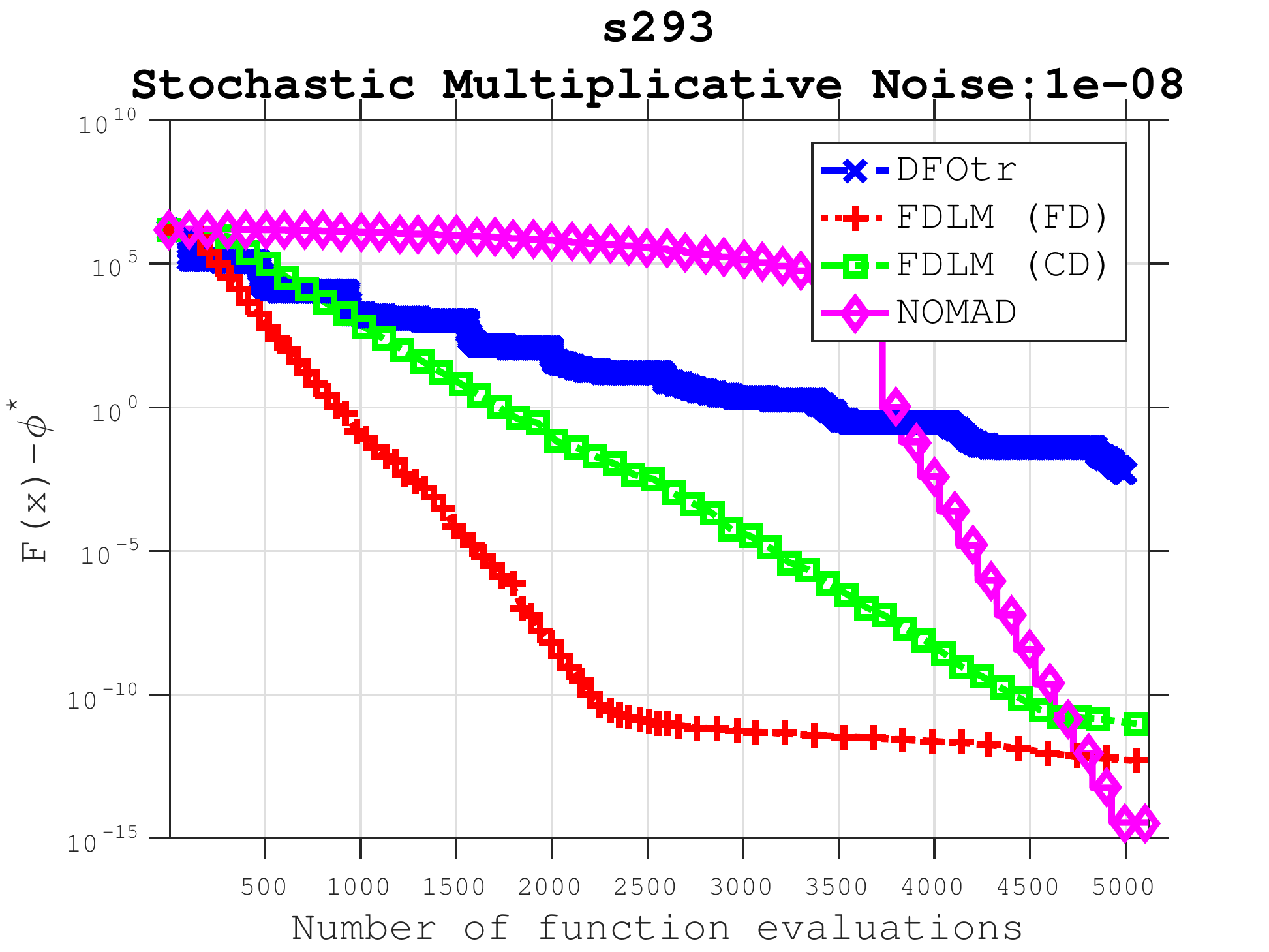}
\includegraphics[width=0.24\textwidth]{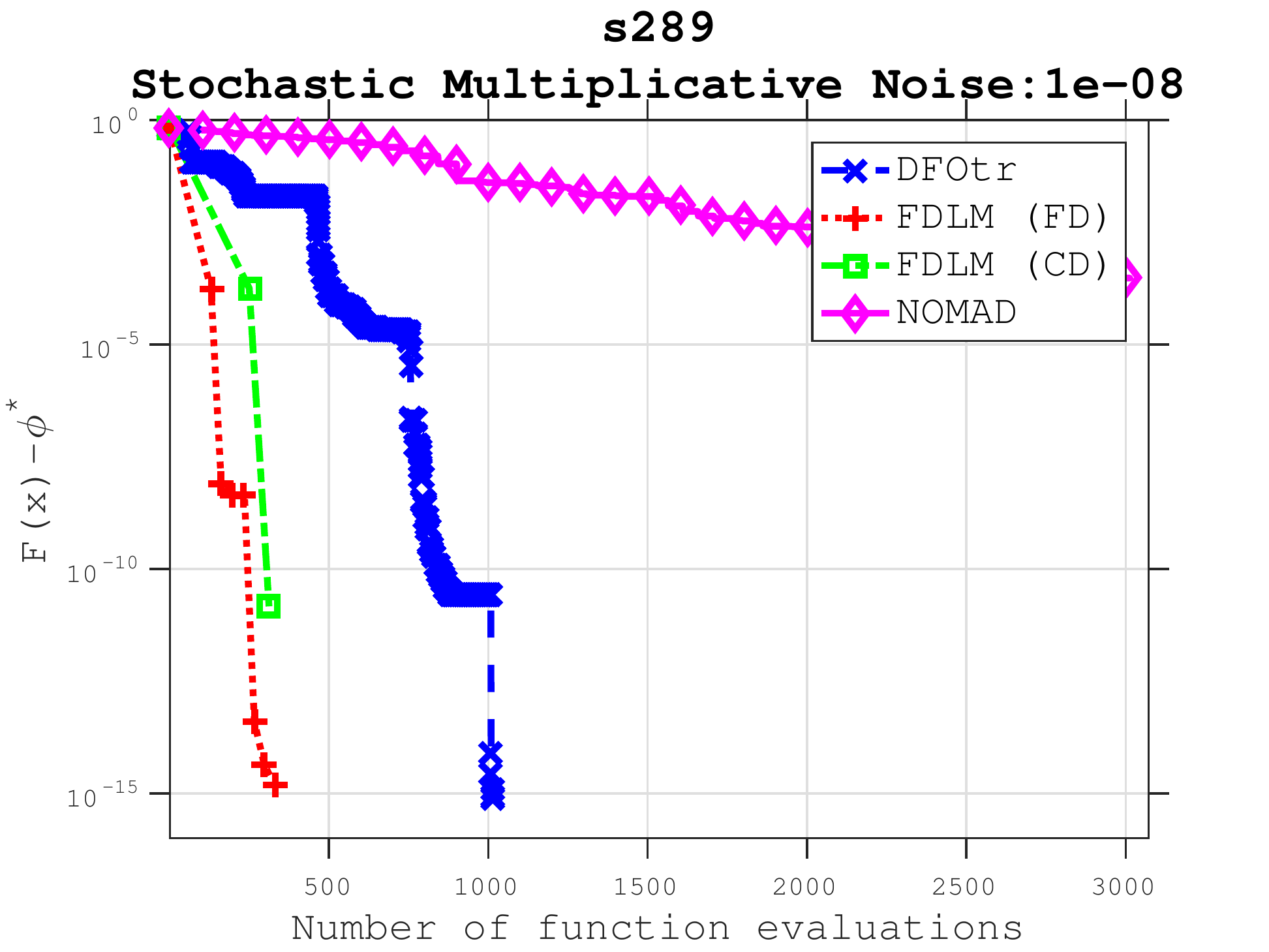}

\includegraphics[width=0.24\textwidth]{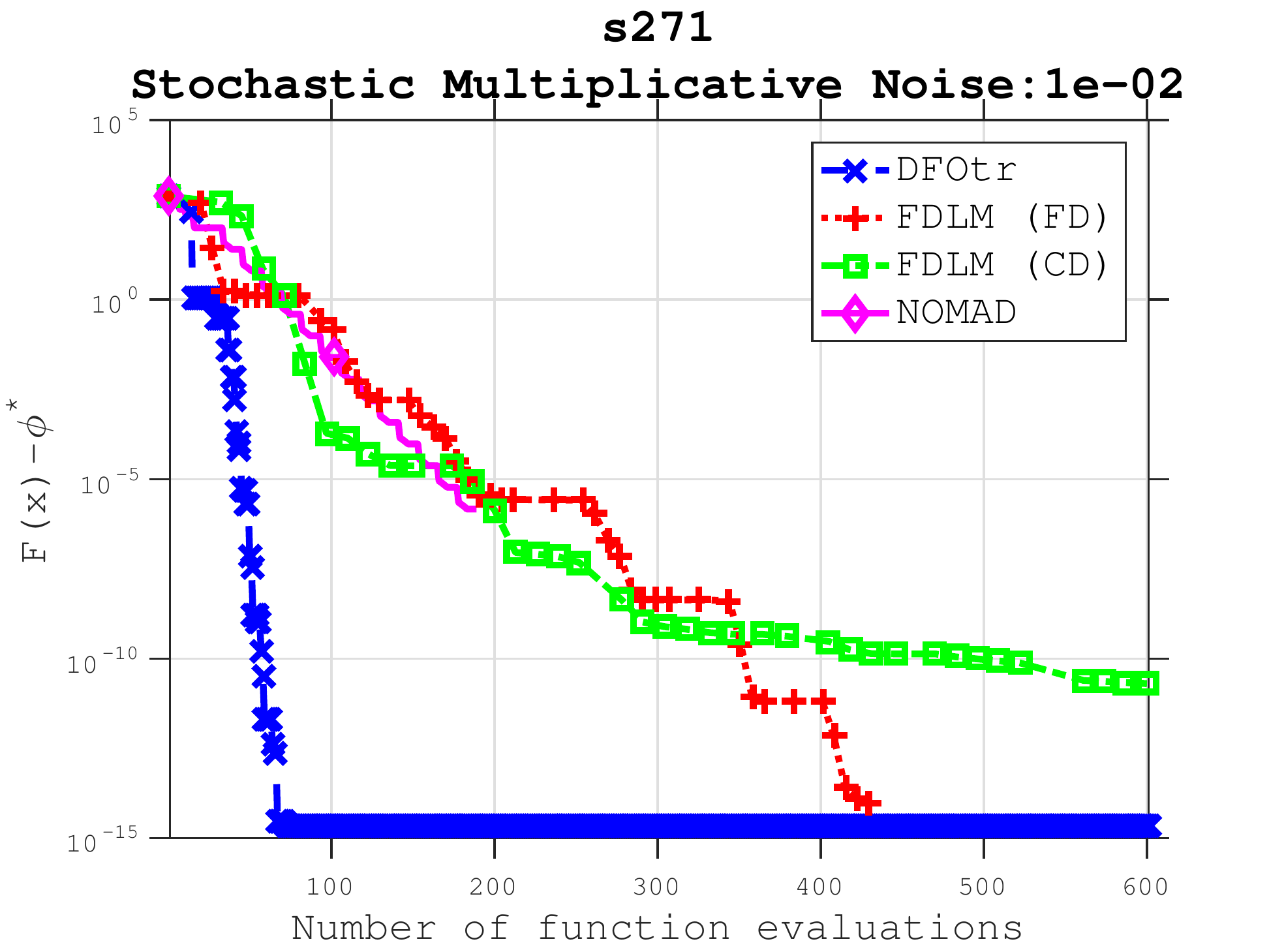}
\includegraphics[width=0.24\textwidth]{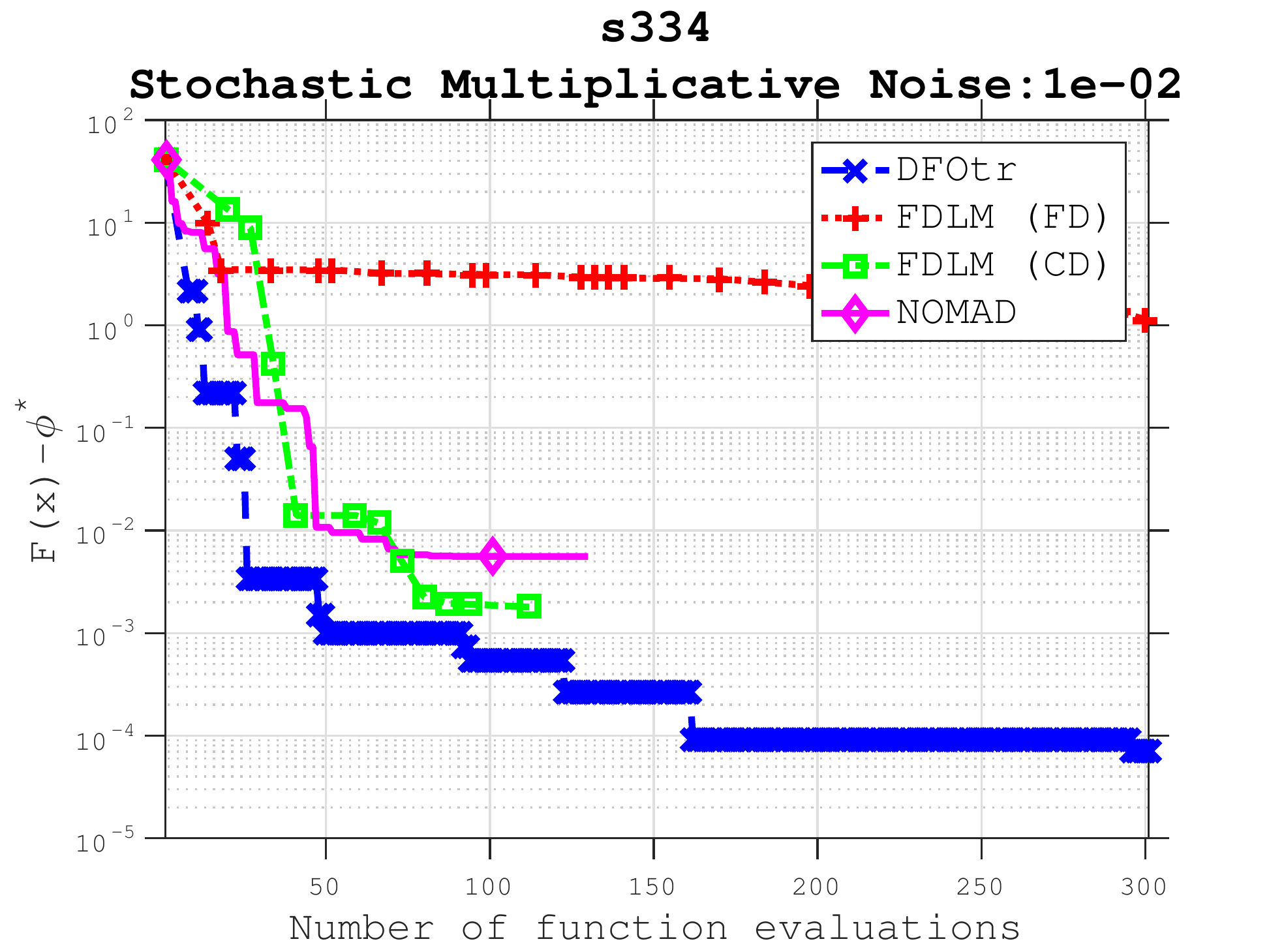}
\includegraphics[width=0.24\textwidth]{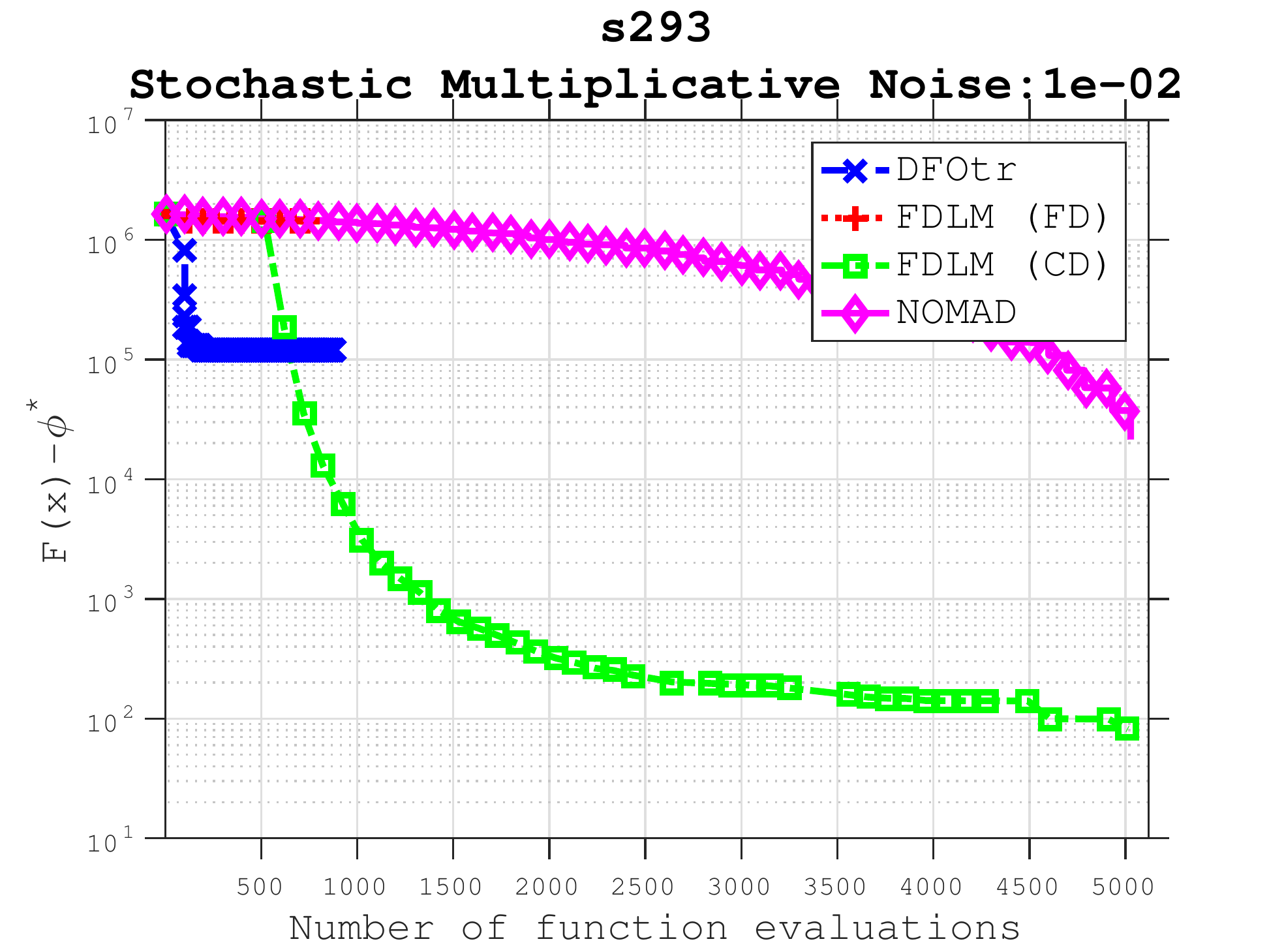}
\includegraphics[width=0.24\textwidth]{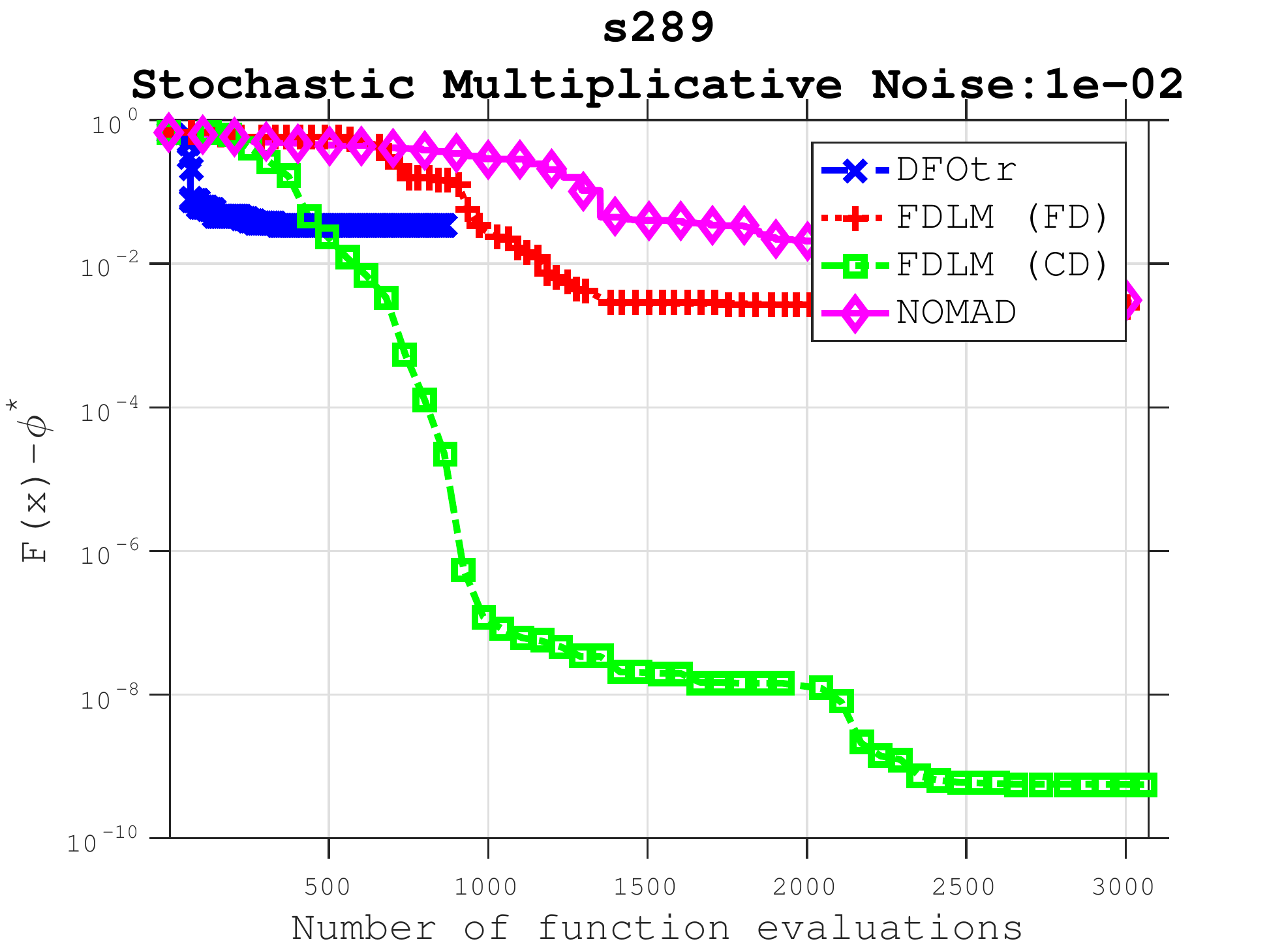}

\par\end{centering}
\caption{\small \texttt{Stochastic Multiplicative Noise.} 
Performance of the function interpolating trust region method ({\tt DFOtr}) described in \cite{conn2009introduction}, {\tt NOMAD} \cite{audet2006mesh,abramson2011nomad} and the finite difference L-BFGS method ({ \tt FDLM}) using forward or central differences. The figure plots results for 4 problems from the Hock-Schittkowki collection \cite{schttfkowski1987more}, for 2 different noise levels. Each column represents a different problem and each row a different noise level. }
\label{stoch_mult_exp}
\end{figure}

\paragraph{Performance Profiles}

Figure~\ref{perf_profs} shows performance profiles \cite{DolMor01} for all 49 problems, for  the $4$ different types of noise mentioned above, and for $2$ different noise levels ($10^{-8}$ and $10^{-2}$), giving a total of $392$ problems; see Appendix \ref{data_profs} for data profiles. Since we observe from the previous results that {\tt NOMAD} is the slowest of the methods, we do not report for it in the sequel. Deterministic noise was generated using the procedure proposed by Mor\'e and Wild \cite{more2009benchmarking}, which is described in Appendix~\ref{det_noise}. For these tests, we follow \cite{more2009benchmarking} and use the following convergence test that measures the decrease in function value
\begin{align}	\label{conv_perf_profs}
	  f(x_0) - f(x_k)  \leq  (1-\tau)\left( f(x_0) - f_L \right), 
\end{align}
where $\tau = 10^{-5}$, $x_0$ is the starting point, and $ f_L$ is  the smallest value of $f$ obtained by any solver within a given budget of  $100\times n$ function evaluations. This convergence test is well suited for derivative-free optimization because it is invariant to affine transformations and measures the function value reduction achieved relative to the best possible reduction \cite{more2009benchmarking}. 

\begin{figure}[]
\begin{centering}

\includegraphics[width=0.24\textwidth]{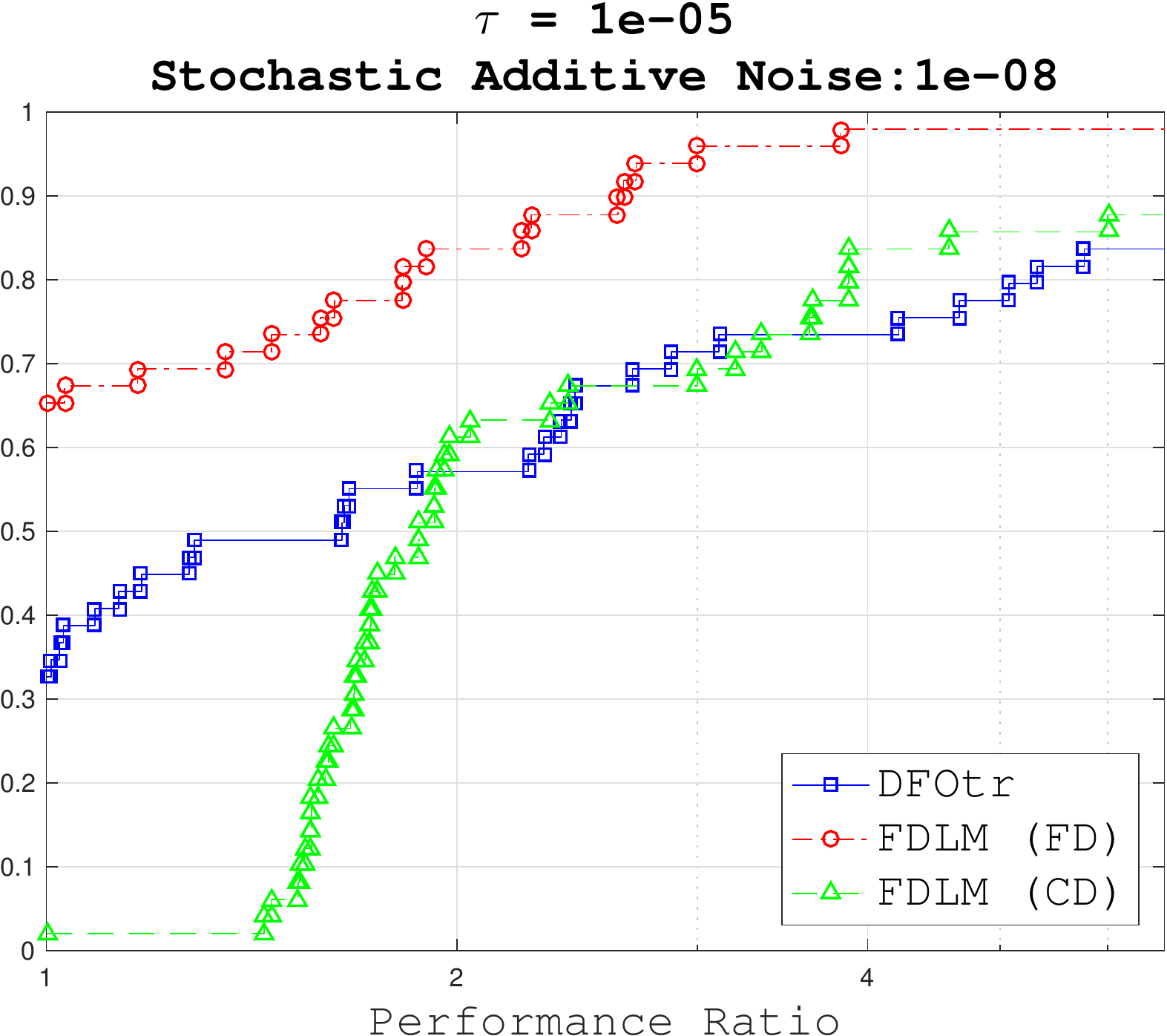}
\includegraphics[width=0.24\textwidth]{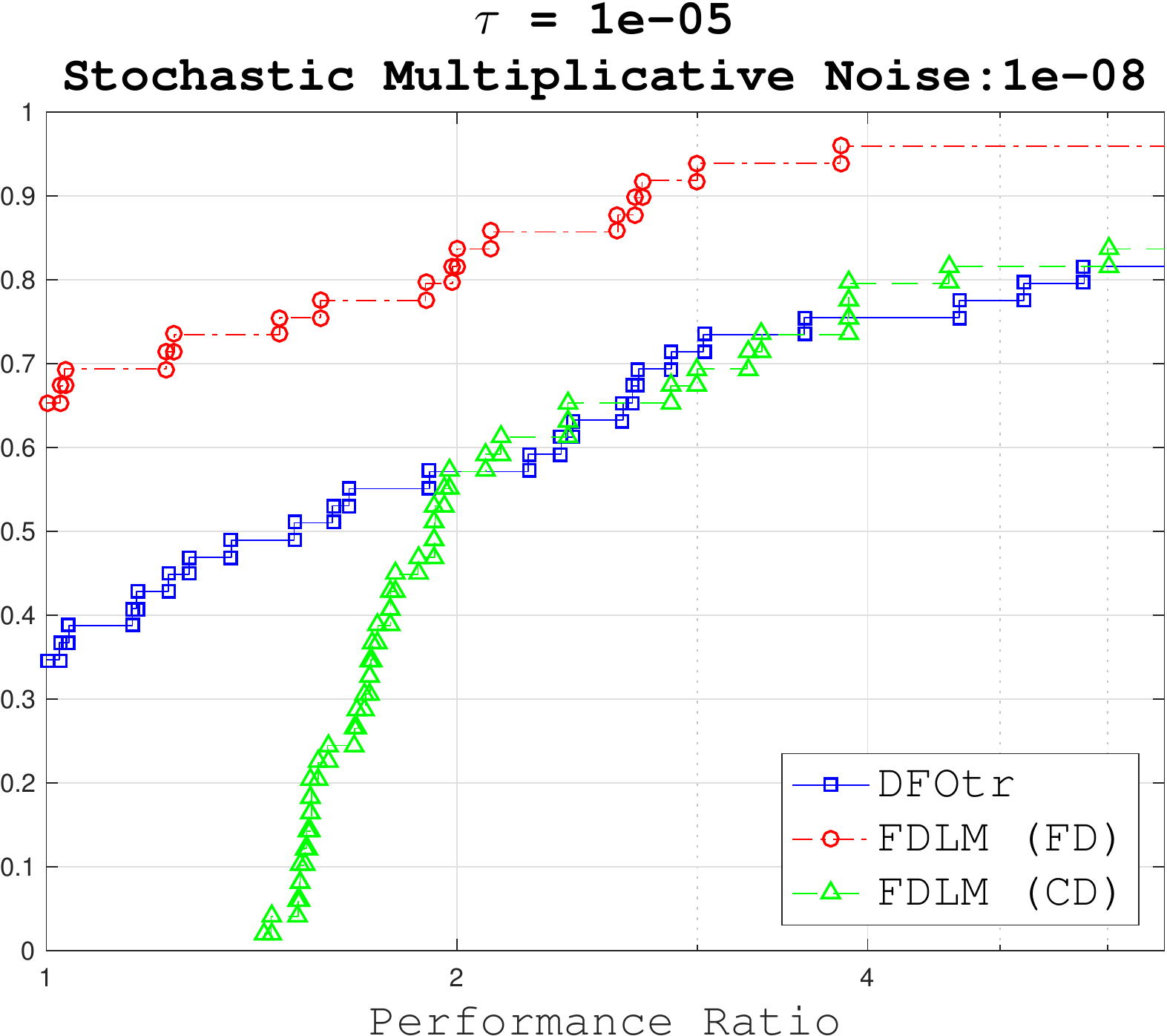}
\includegraphics[width=0.24\textwidth]{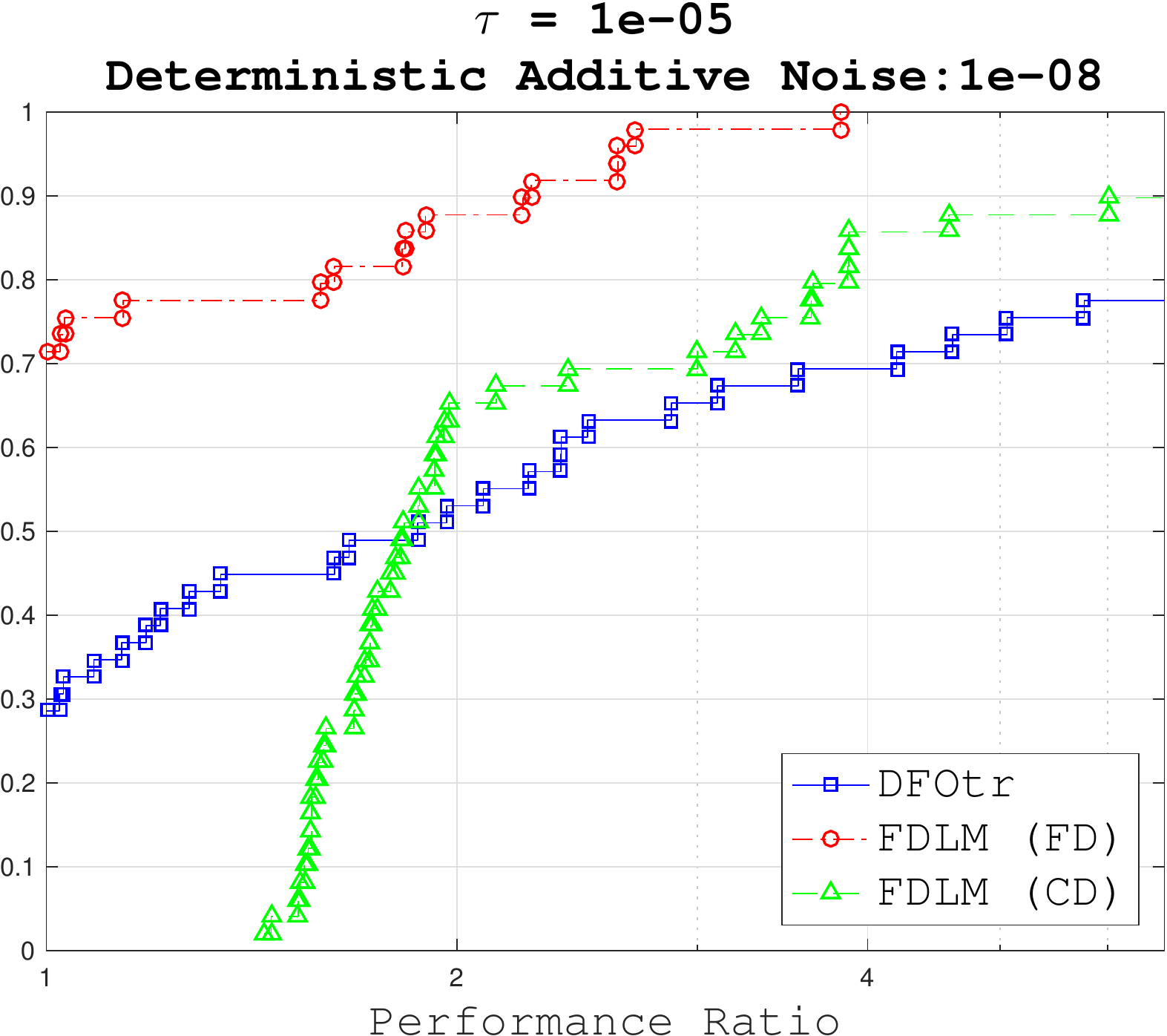}
\includegraphics[width=0.24\textwidth]{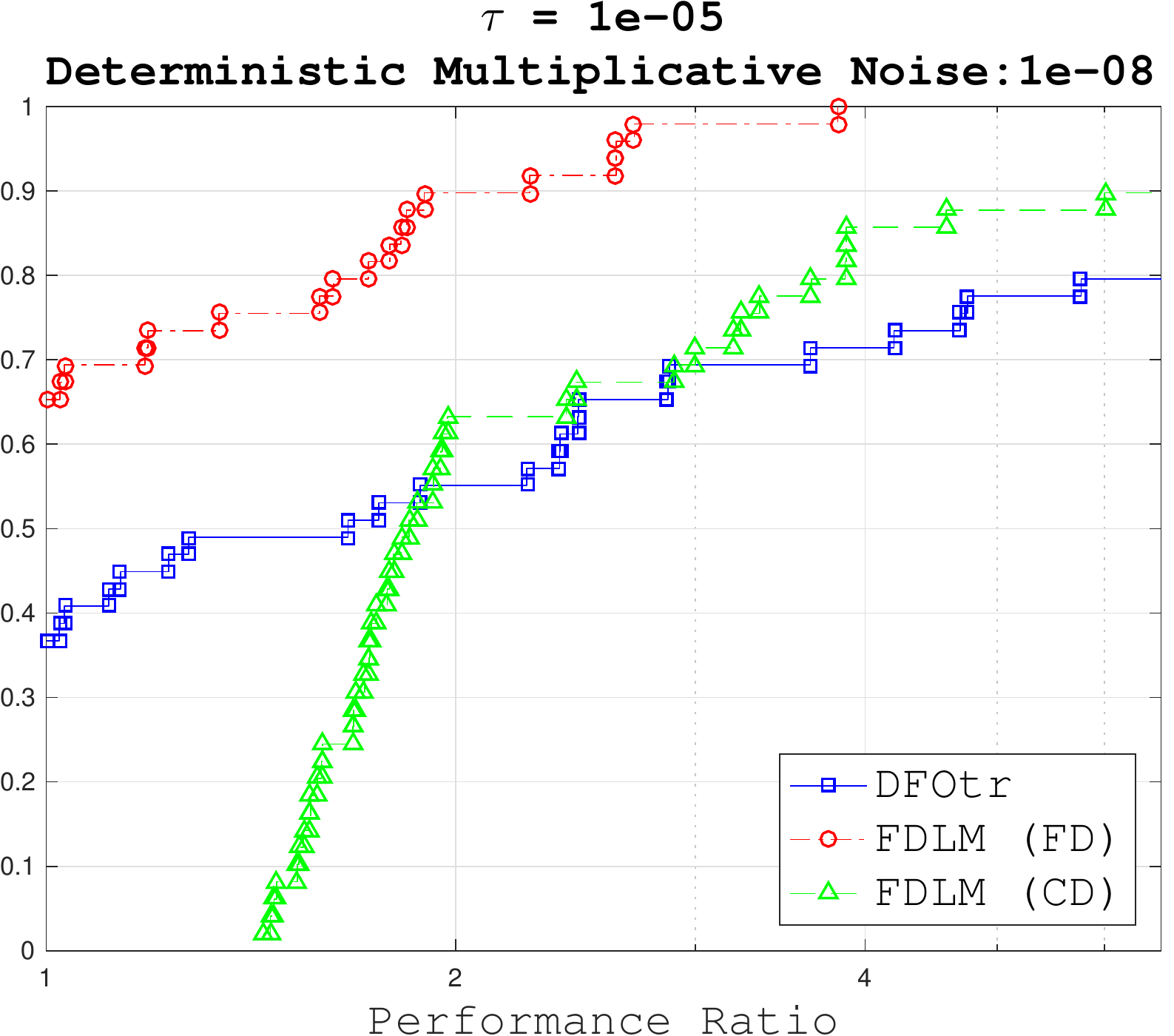}

\vspace{0.2cm}

\includegraphics[width=0.24\textwidth]{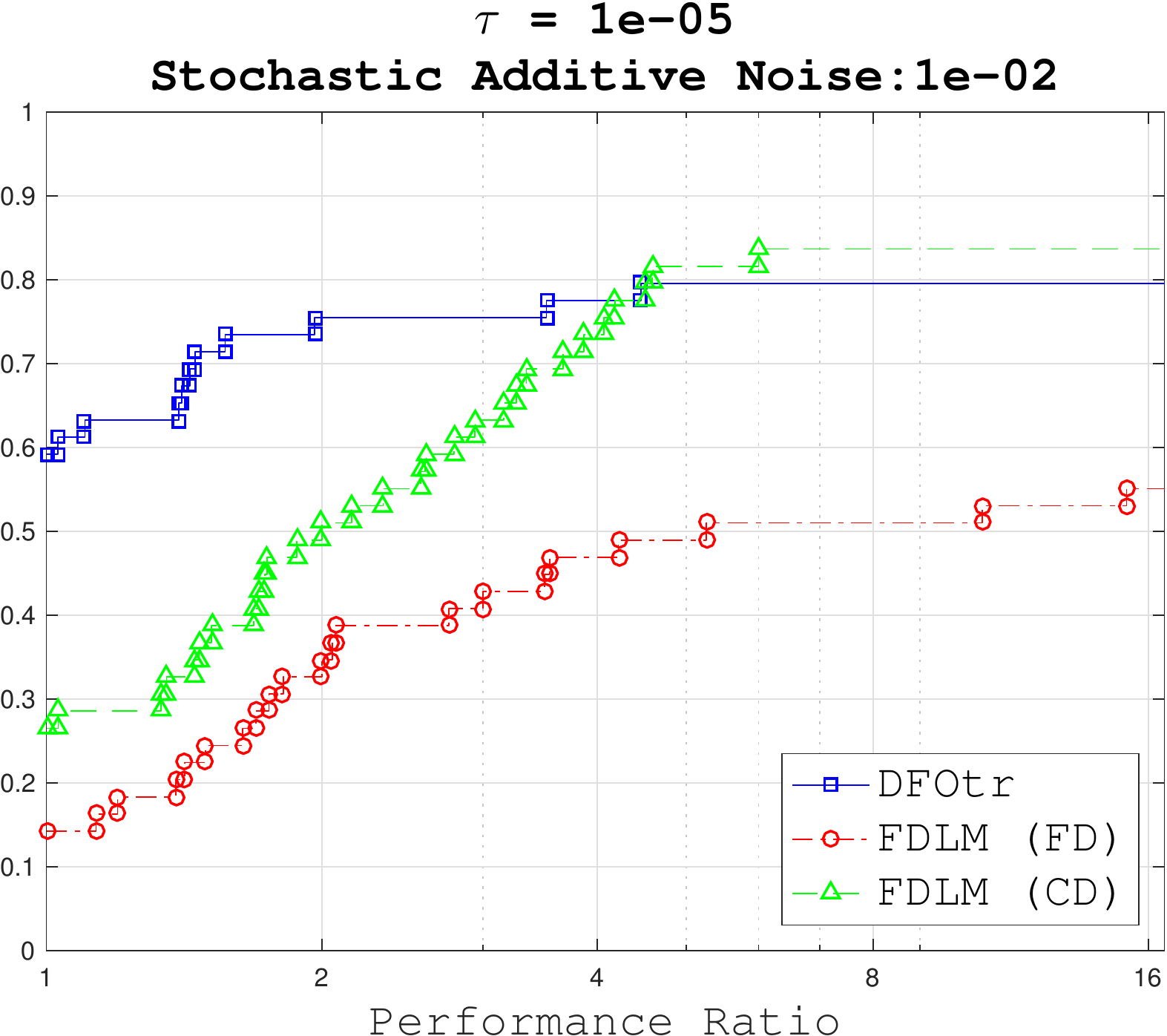}
\includegraphics[width=0.24\textwidth]{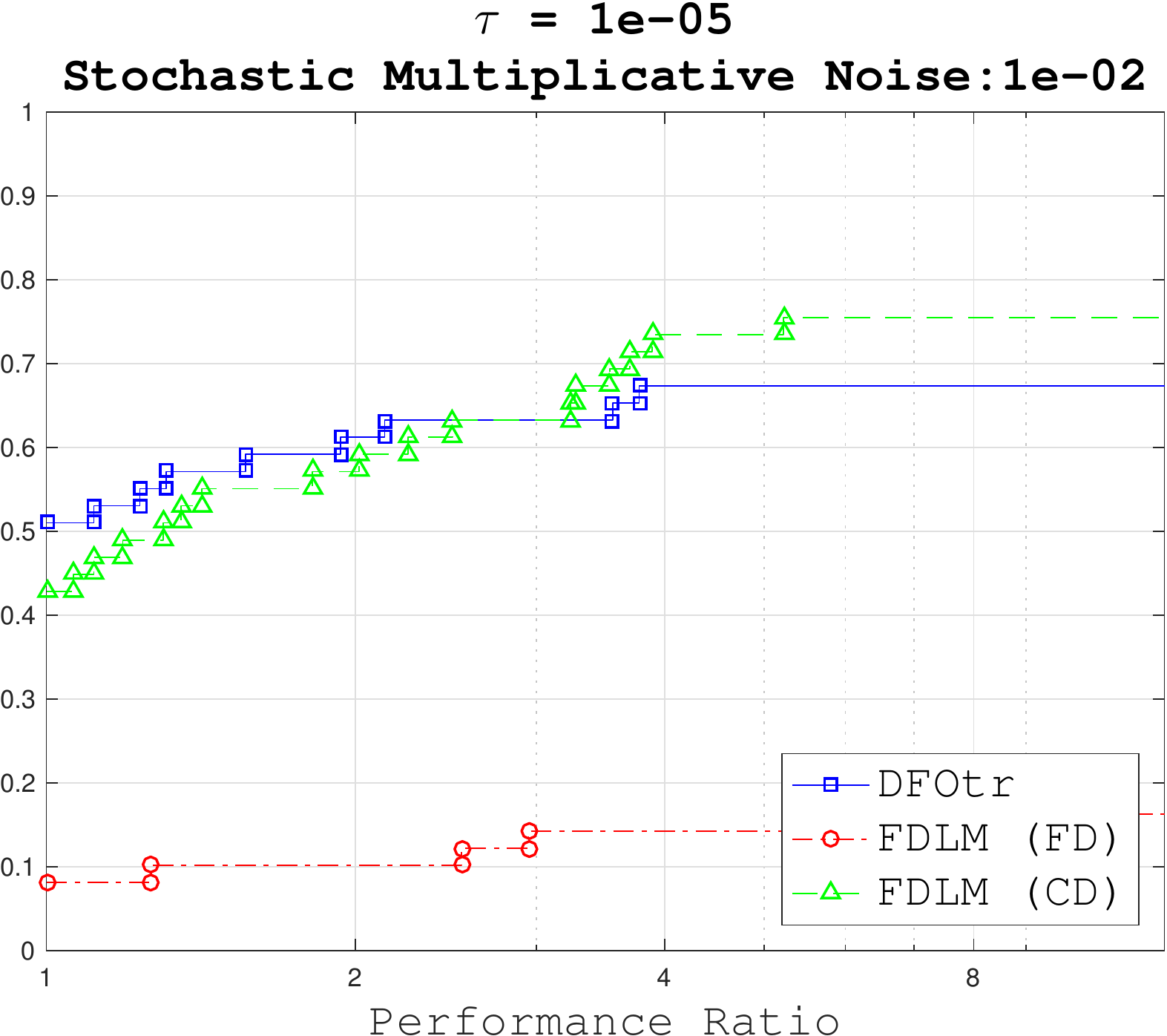}
\includegraphics[width=0.24\textwidth]{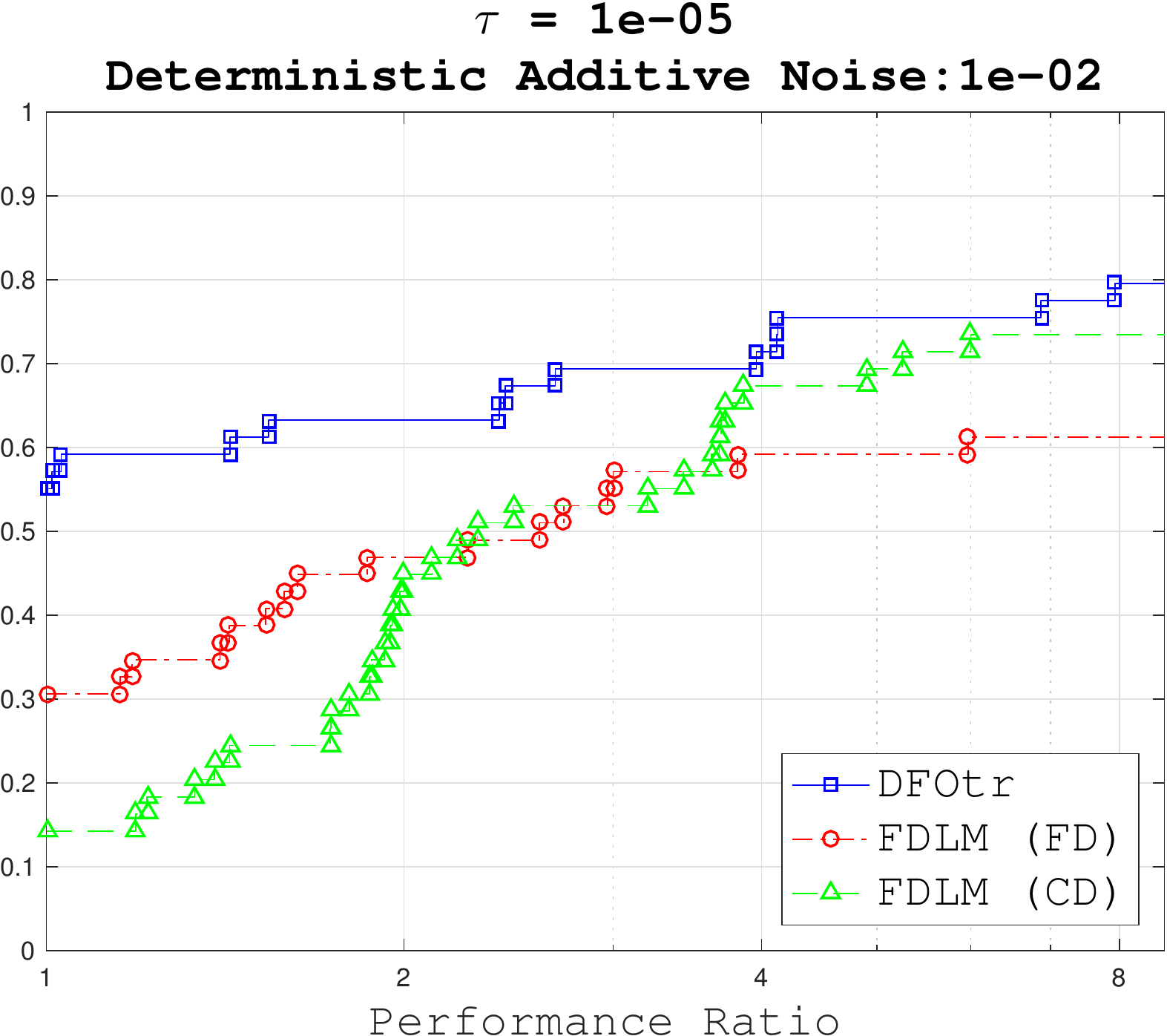}
\includegraphics[width=0.24\textwidth]{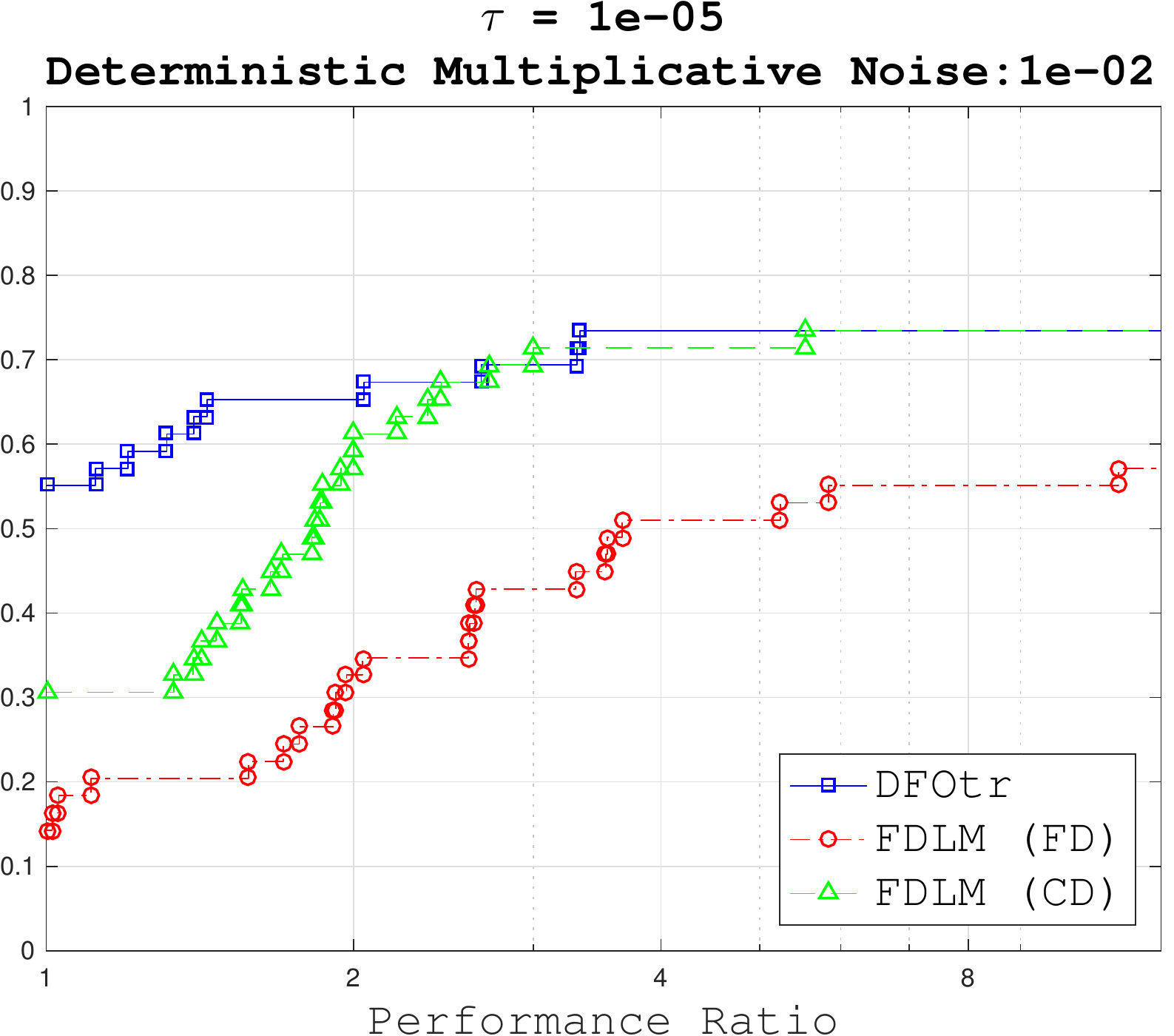}

\par\end{centering}
\caption{\small Performance Profiles ($\tau = 10^{-5}$) \cite{DolMor01}. Each column represents a different noise type and each row  different noise level. Row 1: Noise level $10^{-8}$; Row 2: Noise level $10^{-2}$. Column 1: Stochastic Additive Noise; Column 2: Stochastic Multiplicative Noise; Column 3: Deterministic Additive Noise; Column 4: Deterministic Multiplicative Noise.}
\label{perf_profs}
\end{figure}

\paragraph{Observations} Even though there is some variability, our tests indicate that overall the performance of the FI  and  FDLM methods is roughly comparable in terms of function evaluations. 
Contrary to conventional wisdom, the high per-iteration cost of finite differences  is offset by the faster convergence of the FDLM method, and the potential instability of finite differences is generally not harmful in our tests.  As expected, forward differences are more efficient for low levels of noise, and central differences give rise to a more robust algorithm for high noise levels (e.g., $10^{-2}$); this  suggests that using even higher order gradient approximations would be useful for more difficult problems.

\subsection{The Recovery Mechanism}
 We now investigate if the \texttt{Recovery} mechanism improves the robustness of our approach. We ran the FDLM method with and without the \texttt{Recovery} procedure; the latter amounts to terminating as soon as the \texttt{LineSearch} procedure fails (Line 11, Algorithm~\ref{alg1}). In Figure \ref{rec_exp} we present performance profiles for problems with stochastic additive and multiplicative noise for two noise levels ($10^{-8}$ and $10^{-2}$); see Appendix \ref{data_profs_rec} for data profiles. As shown, the performance of the method deteriorates substantially when the \texttt{Recovery} procedure is not used. 
 
\begin{figure}[htp]
\begin{centering}

\includegraphics[width=0.24\textwidth]{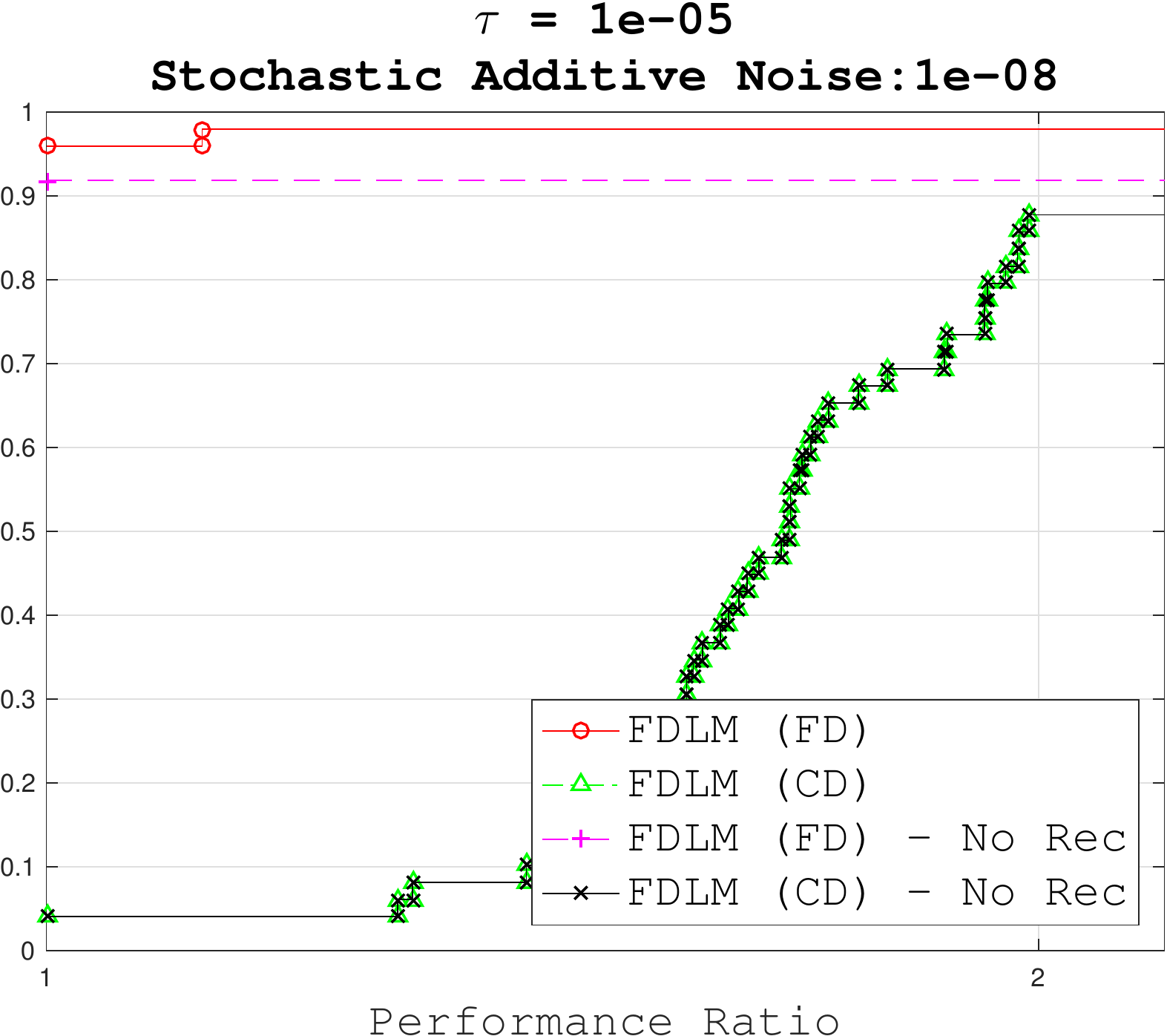}
\includegraphics[width=0.24\textwidth]{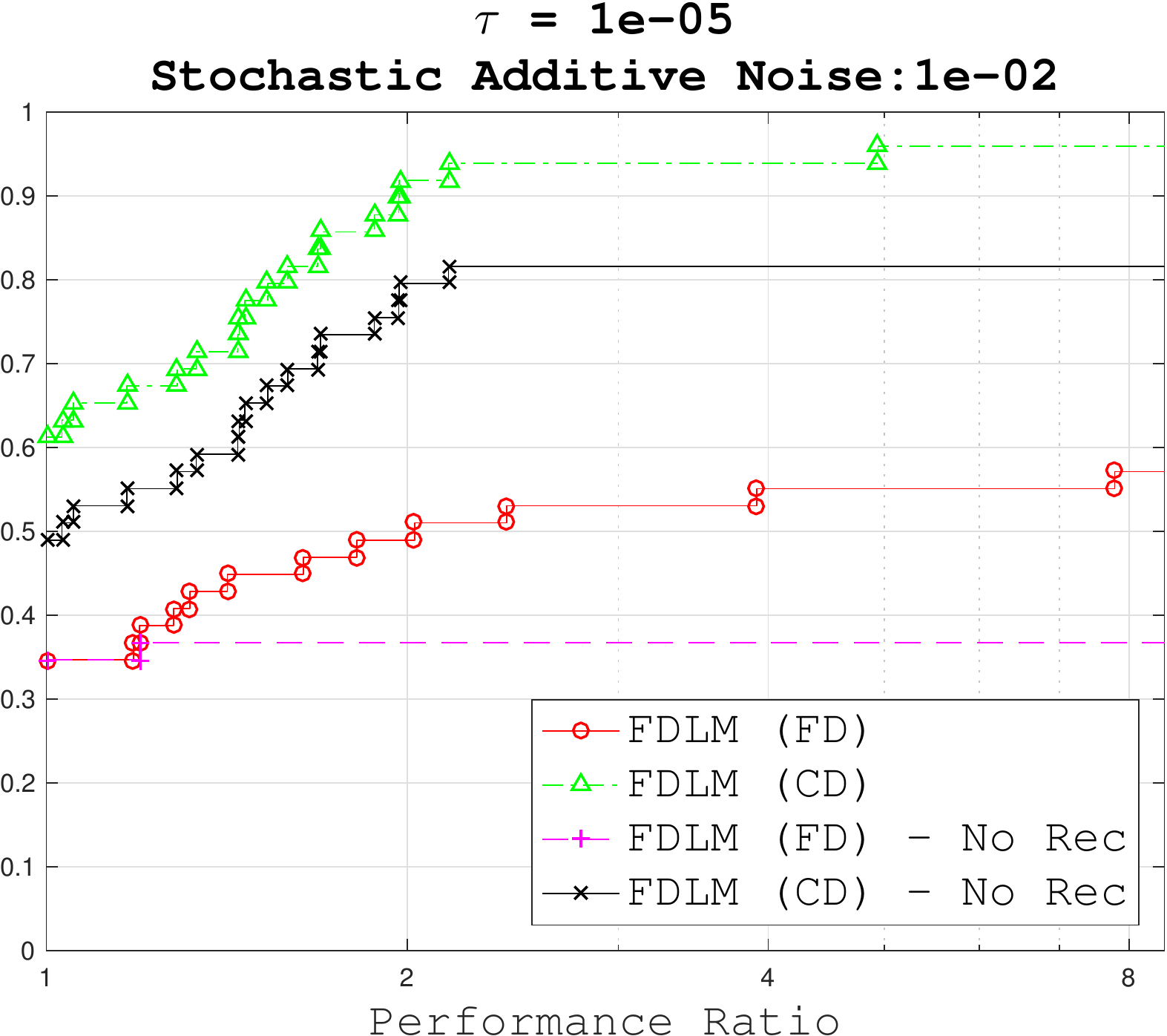}
\includegraphics[width=0.24\textwidth]{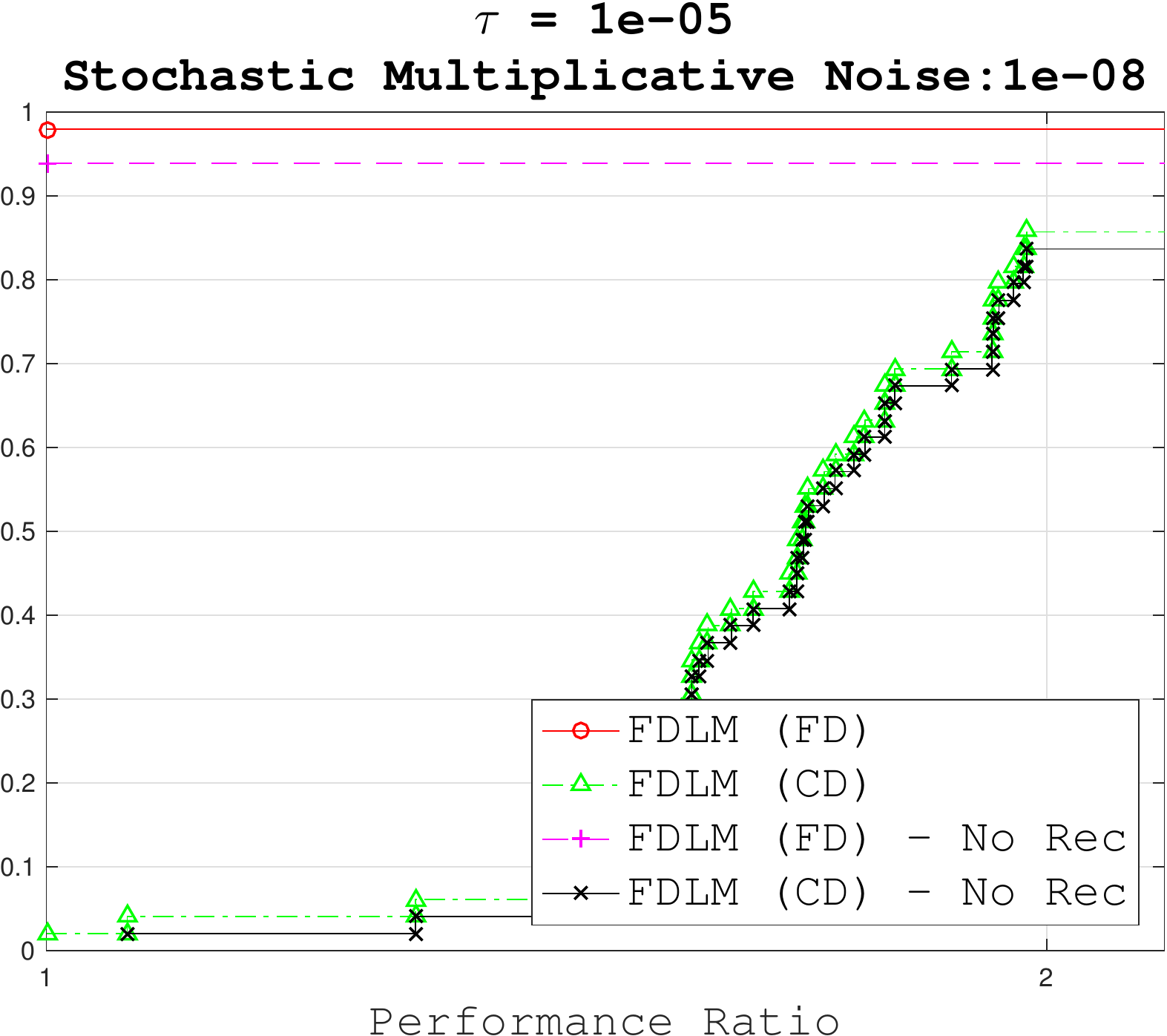}
\includegraphics[width=0.24\textwidth]{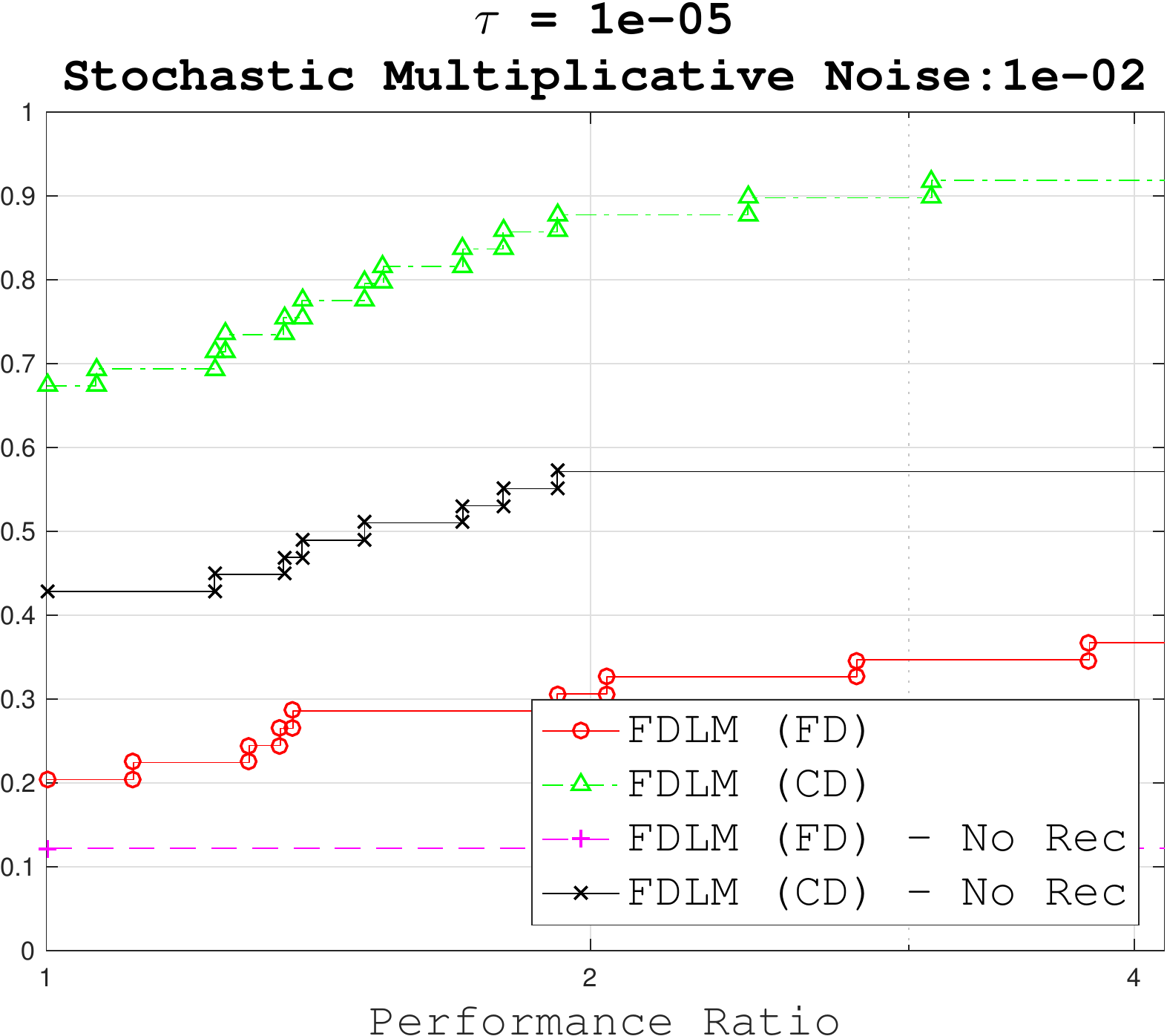}

\par\end{centering}
\caption{\small Performance of the FDLM method using forward and central differences with and without the \texttt{Recovery} procedure. The figure plots performance profiles for stochastic additive and stochastic multiplicative noise (noise levels $10^{-8}$, $10^{-2}$). The first two plots show results for stochastic additive noise, and the last two plots show results for stochastic multiplicative noise.}
\label{rec_exp}
\end{figure}

Our complete set of experiments show that the \texttt{Recovery} procedure is invoked more often when the noise level is high or when forward differences are employed, as expected. It plays an important role for problems with multiplicative noise where the noise level changes during the course of the iteration and our approach relies on the \texttt{Recovery} mechanism to adjust the finite difference interval. For some problems, the \texttt{Recovery} procedure is never invoked, but in the most difficult cases it plays a crucial role.

We also observe (see e.g., Figures~\ref{stoch_add_exp}-\ref{stoch_mult_exp})  that the FDLM method with forward differences  is more efficient than the central difference variant in the initial stages of the optimization, but the latter is able to achieve more accurate solutions within the given budget. It is therefore natural to consider a method that starts with forward differences and switches to central differences. How to do this in an algorithmically sound way remains a topic of investigation. 

%
%

%
%

\section{Final Remarks}
\label{finalr}
We presented a finite difference quasi-Newton method for the optimization of noisy black-box functions. It relies on the observation that when the level of noise (i.e., its standard deviation) can be estimated,  it is possible to choose finite difference intervals that yield reliable forward or central difference approximations to the gradient. To estimate the noise level, we employ the {\tt ECnoise}  procedure of Mor\'e and Wild \cite{more2011estimating}. Since this procedure may not always be reliable, or since the noise level may change during the course of the minimization, our algorithm includes a {\tt Recovery} procedure that adjusts the finite difference interval, as needed. This procedure operates in conjunction with a backtracking Armijo line search.  

Our numerical experiments indicate that our approach is robust and efficient. Therefore, performing $\mathcal{O}(n)$ function evaluations at every iteration to estimate the gradient is not prohibitively expensive and has the advantage that the construction of the model of the objective can be delegated to the highly scalable L-BFGS updating technique.   
We present a  convergence analysis of our method using an Armijo-type backtracking line search that does not assume that the error in the function evaluations tends to zero.

\newpage

\appendix

\section{Extended Numerical Results}
\label{app:supp_num}

In this appendix, we present additional numerical results.

\subsection{Smooth Functions}
\label{extended_schitt_smooth}
Figure \ref{perf_smooth_app} shows performance profiles for different values of $\tau$ (see equation \eqref{conv_perf_profs}) and Figure \ref{data_smooth_app} shows data profiles \cite{more2009benchmarking} for different values of $\tau$.

%
%
%
%
\begin{figure}[H]
\begin{centering}

\includegraphics[width=0.25\textwidth]{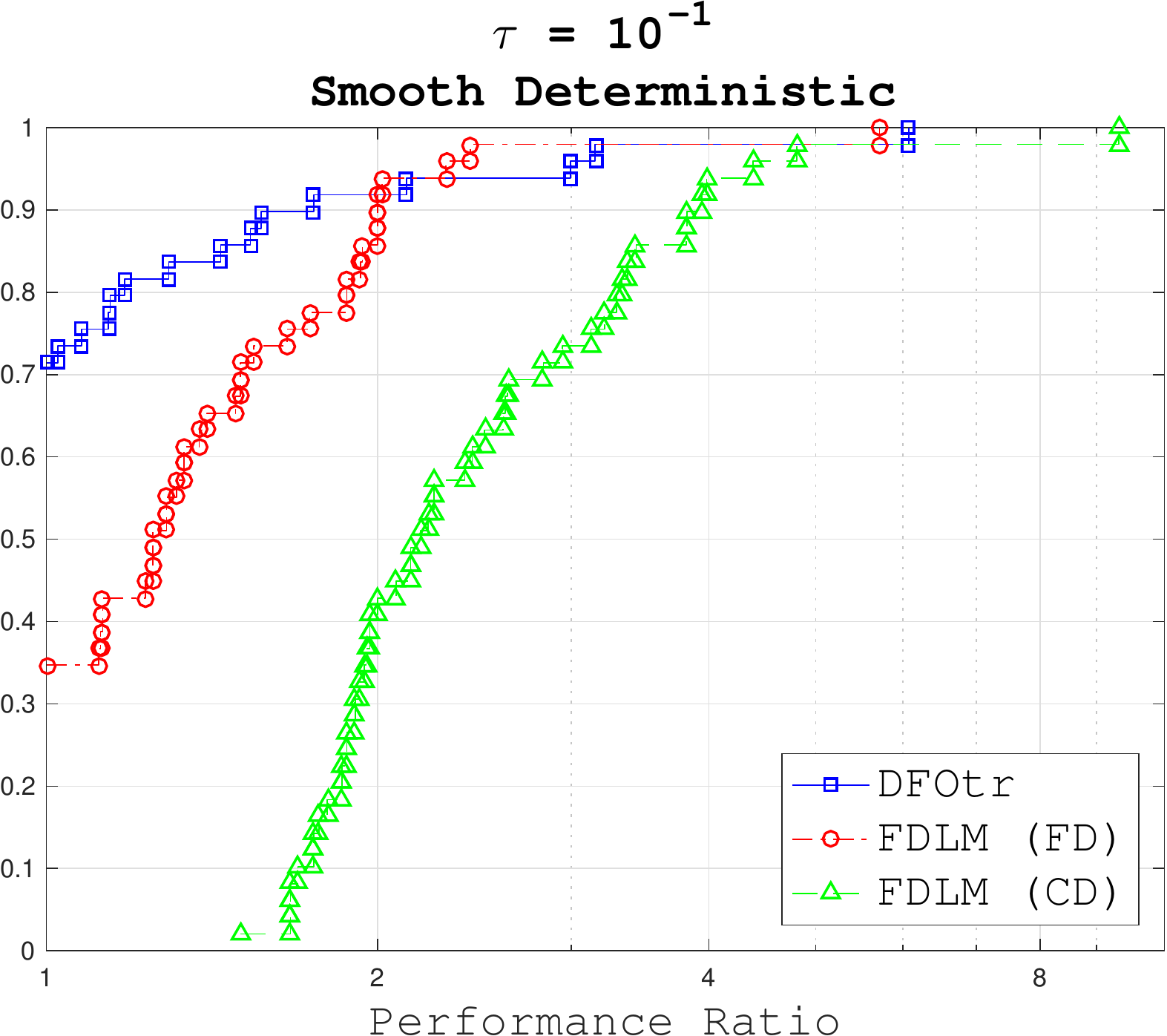}
\includegraphics[width=0.25\textwidth]{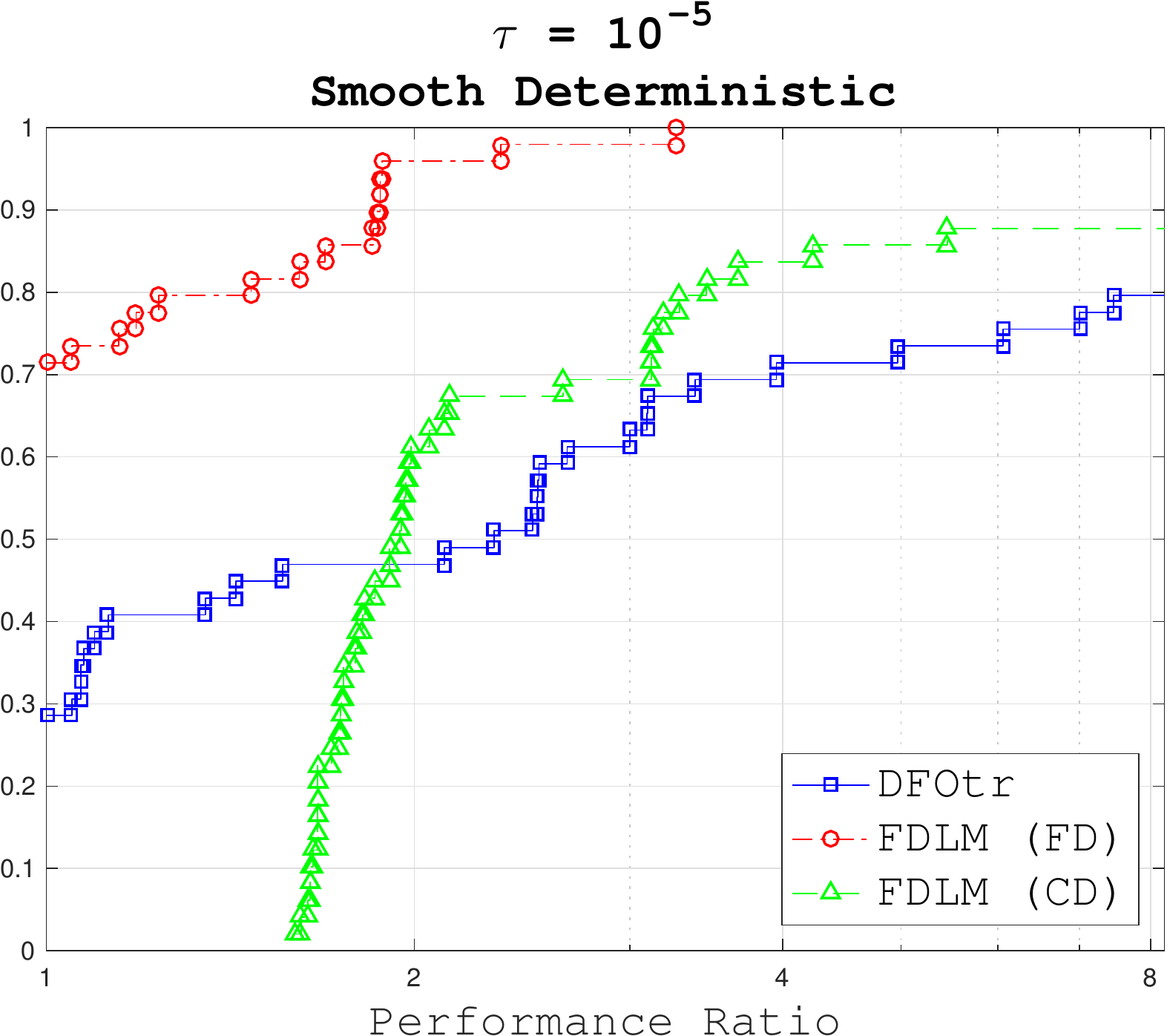}
\includegraphics[width=0.25\textwidth]{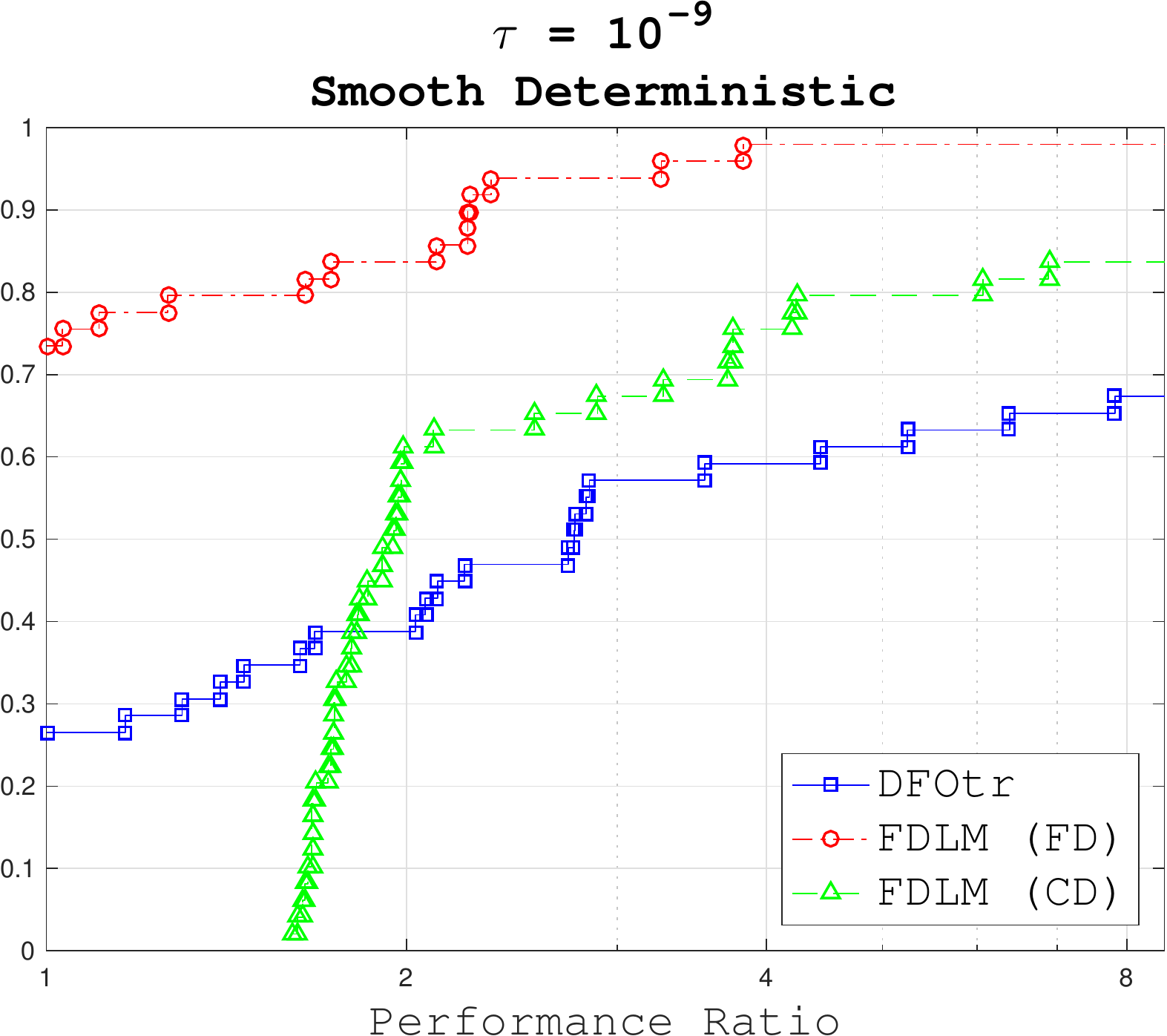}

\par\end{centering}
\caption{\small Performance profiles -- Smooth Deterministic Functions. Left: $\tau=10^{-1}$; Center: $\tau=10^{-5}$; Right: $\tau=10^{-9}$. \label{perf_smooth_app}}
\end{figure}

\begin{figure}[H]
\begin{centering}

\includegraphics[width=0.25\textwidth]{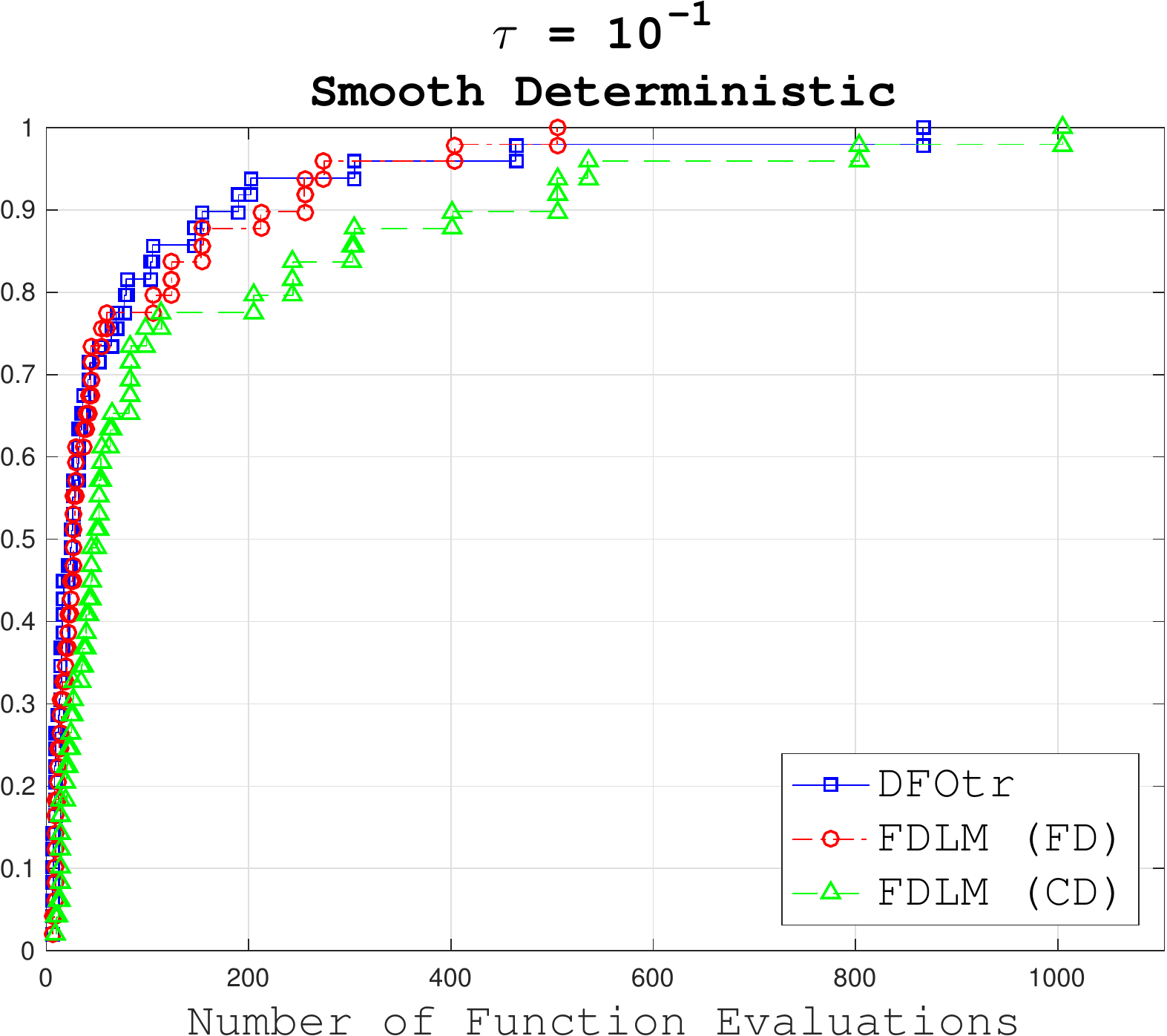}
\includegraphics[width=0.25\textwidth]{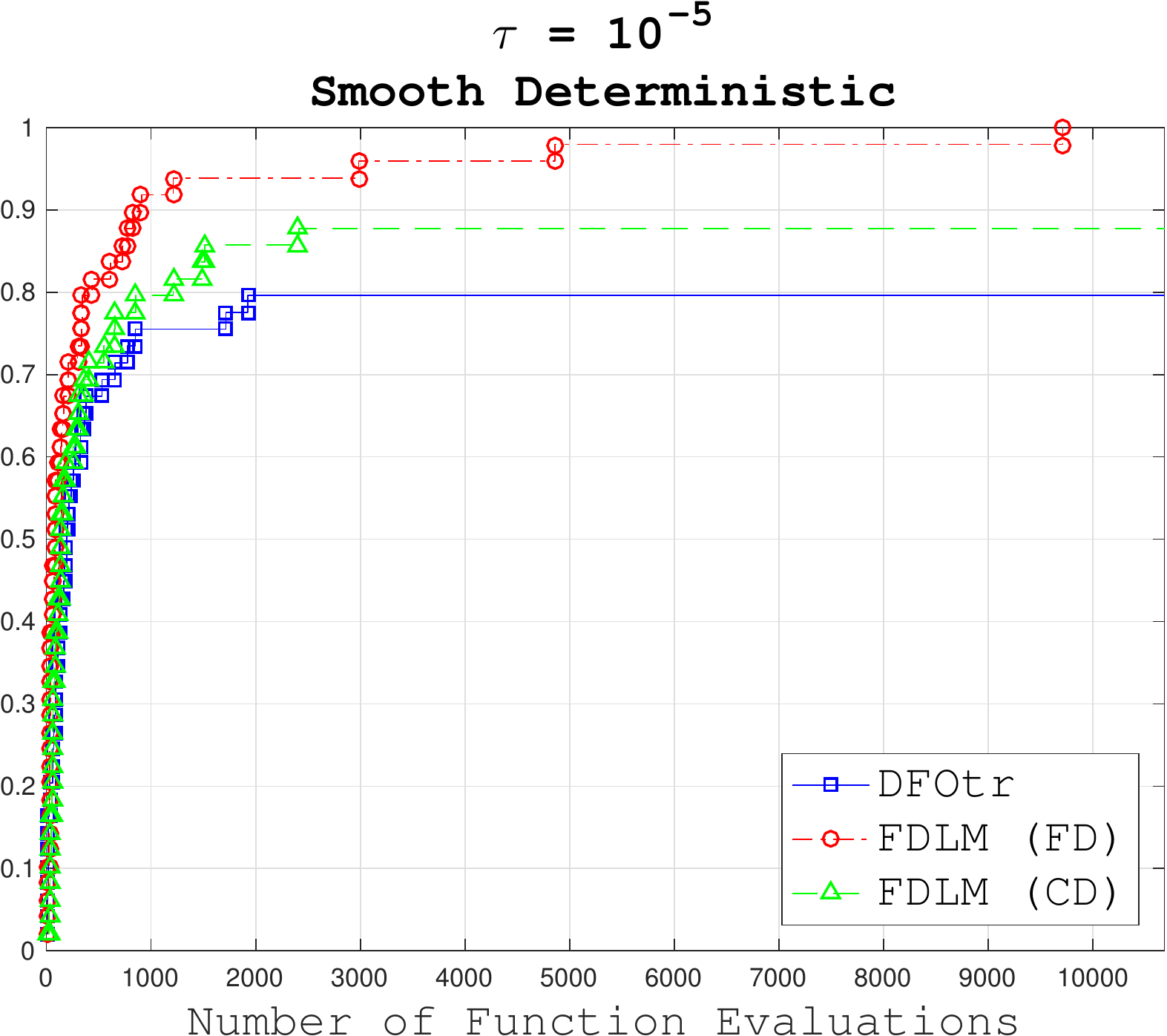}
\includegraphics[width=0.25\textwidth]{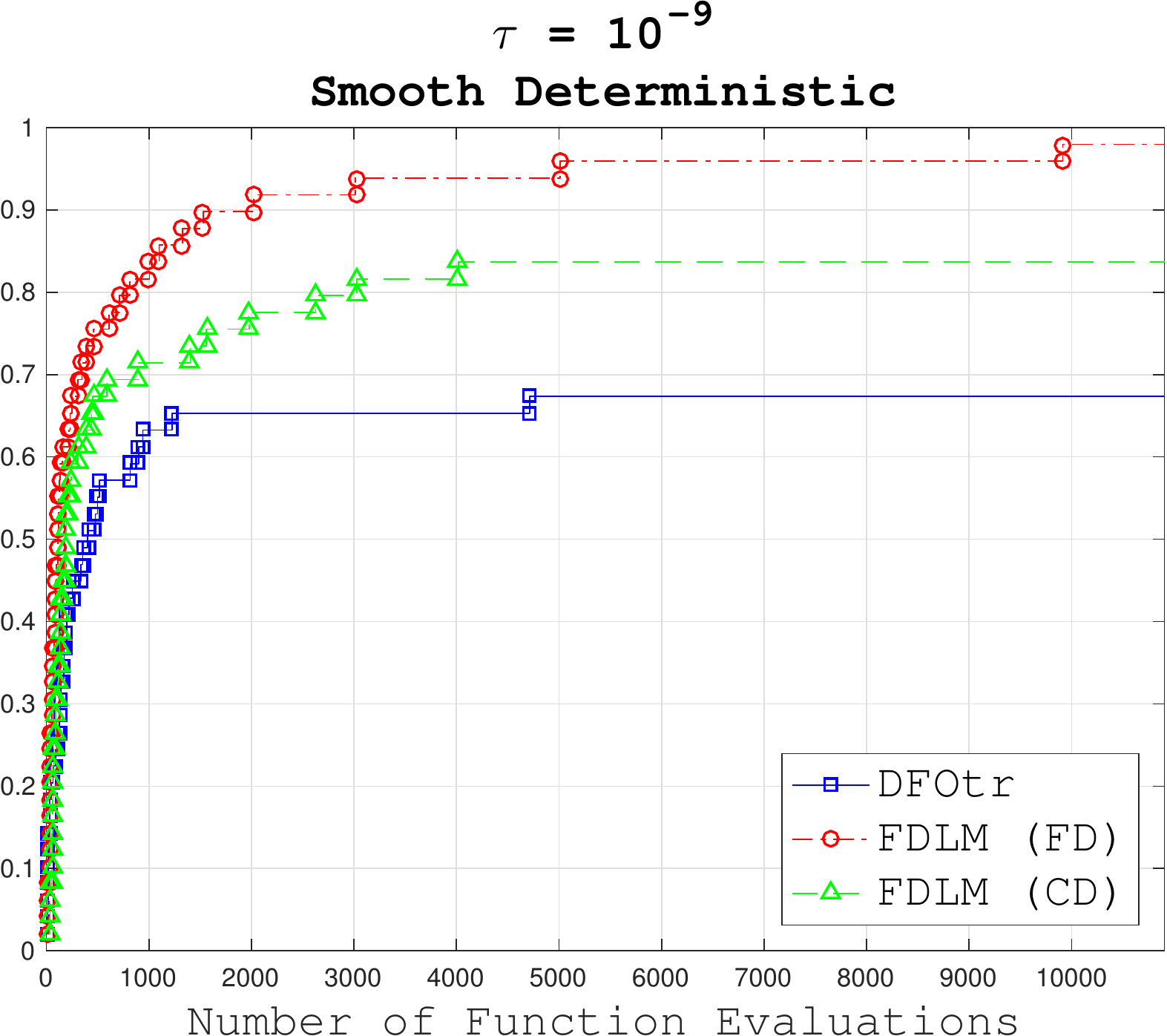}

\par\end{centering}
\caption{\small Data profiles -- Smooth Deterministic Functions. Left: $\tau=10^{-1}$; Center: $\tau=10^{-5}$; Right: $\tau=10^{-9}$. \label{data_smooth_app}}
\end{figure}

\subsection{Noisy Problems}
\label{ext_summ_noisy}

\subsubsection{Generation of Deterministic Noise}
\label{det_noise}

 Deterministic noise was generated using the procedure described by Mor\'e and Wild \cite{more2009benchmarking}. Namely,
\begin{align*}		
	\epsilon(x) = \xi \psi(x),	
\end{align*}
where $\xi \in \{10^{-8}, 10^{-6}, 10^{-4}, 10^{-2} \}$, and $\psi: \mathbb{R}^n \rightarrow [-1,1]$ is  defined in terms of the cubic Chebyshev polynomial $T_3(\alpha) = \alpha(4\alpha^2-3)$, as follows:
\begin{gather*} 
	\psi(x) = T_3(\psi_0(x)), \ \ \mbox{where} \ \ 
	\psi_0(x) = 0.9\sin(100\| x\|_1)\cos(100\| x\|_{\infty}) + 0.1\cos(\|x\|_2). \label{cheb2}
\end{gather*}

\subsubsection{Data Profiles}
\label{data_profs}

We present data profiles for the problems described in Section \ref{numerical}.

\begin{figure}[h]
\begin{centering}

\includegraphics[width=0.24\textwidth]{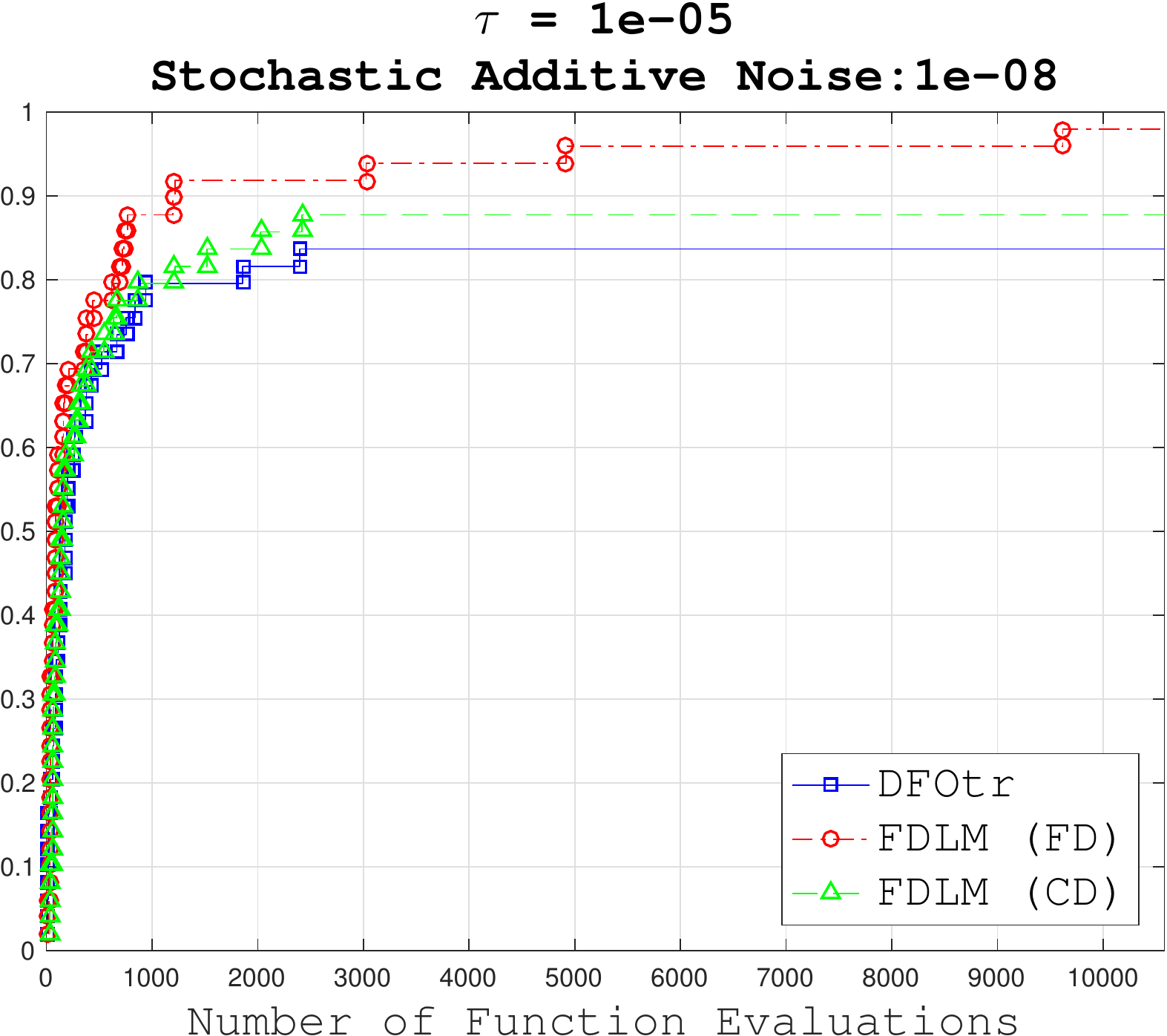}
\includegraphics[width=0.24\textwidth]{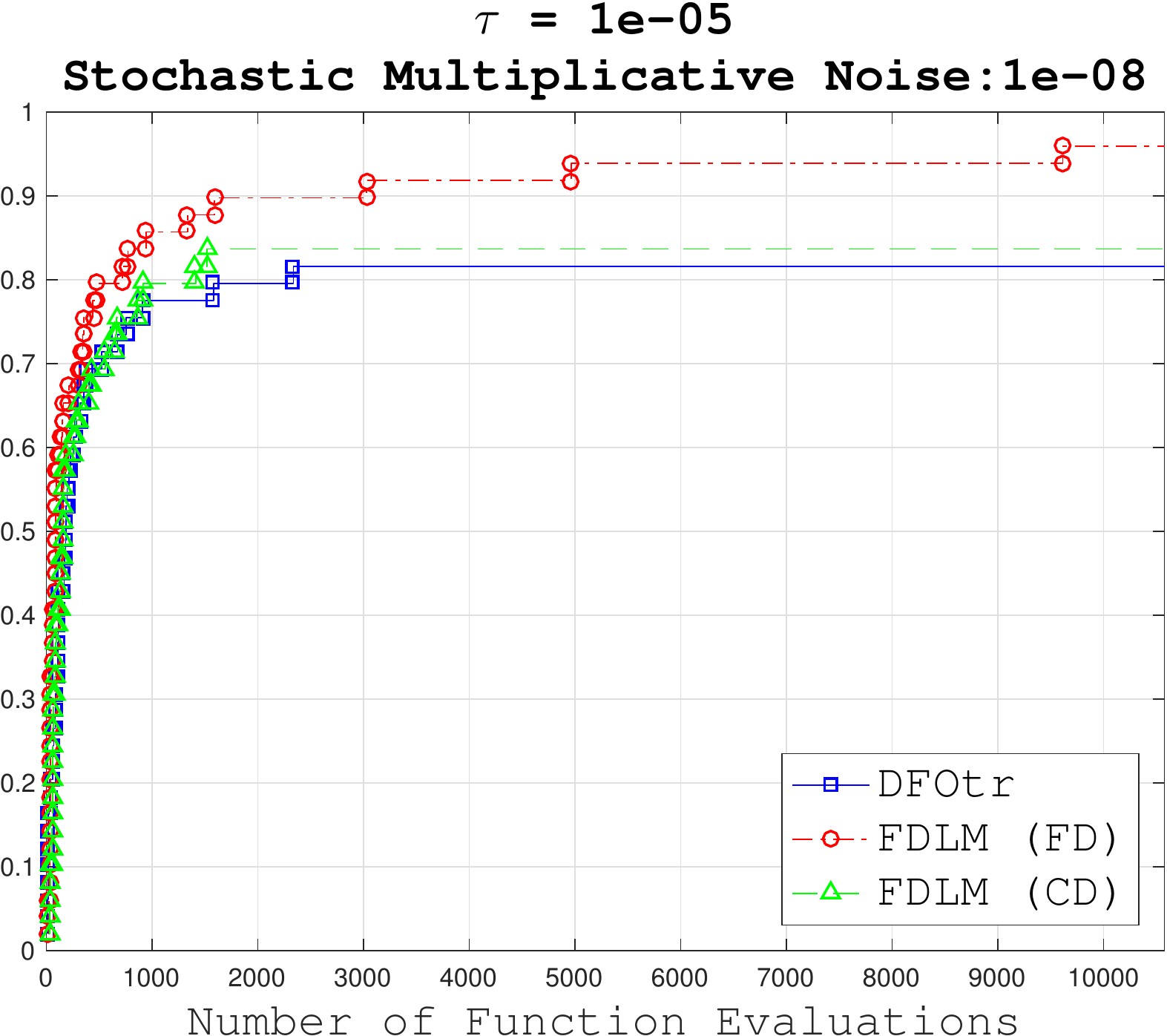}
\includegraphics[width=0.24\textwidth]{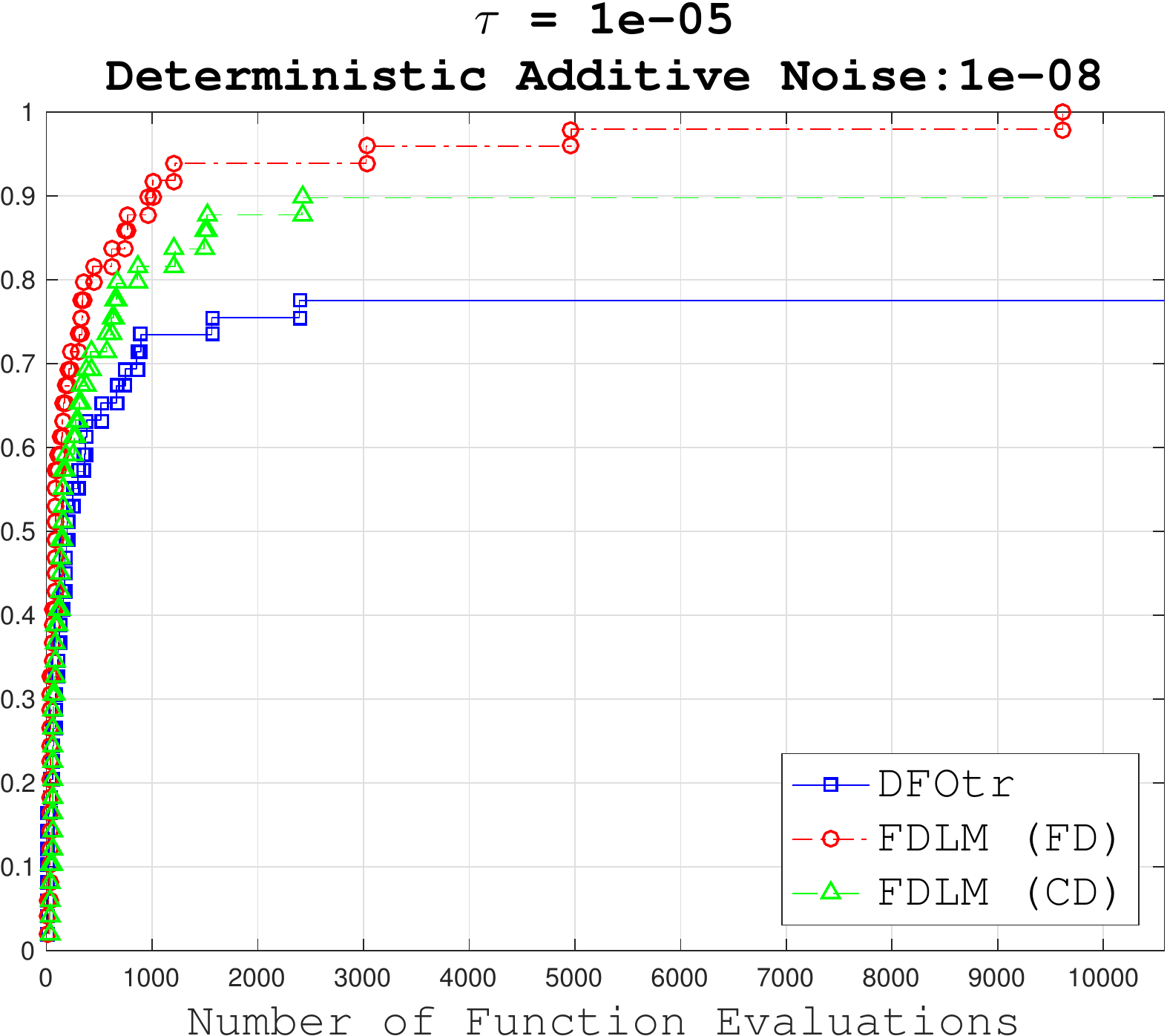}
\includegraphics[width=0.24\textwidth]{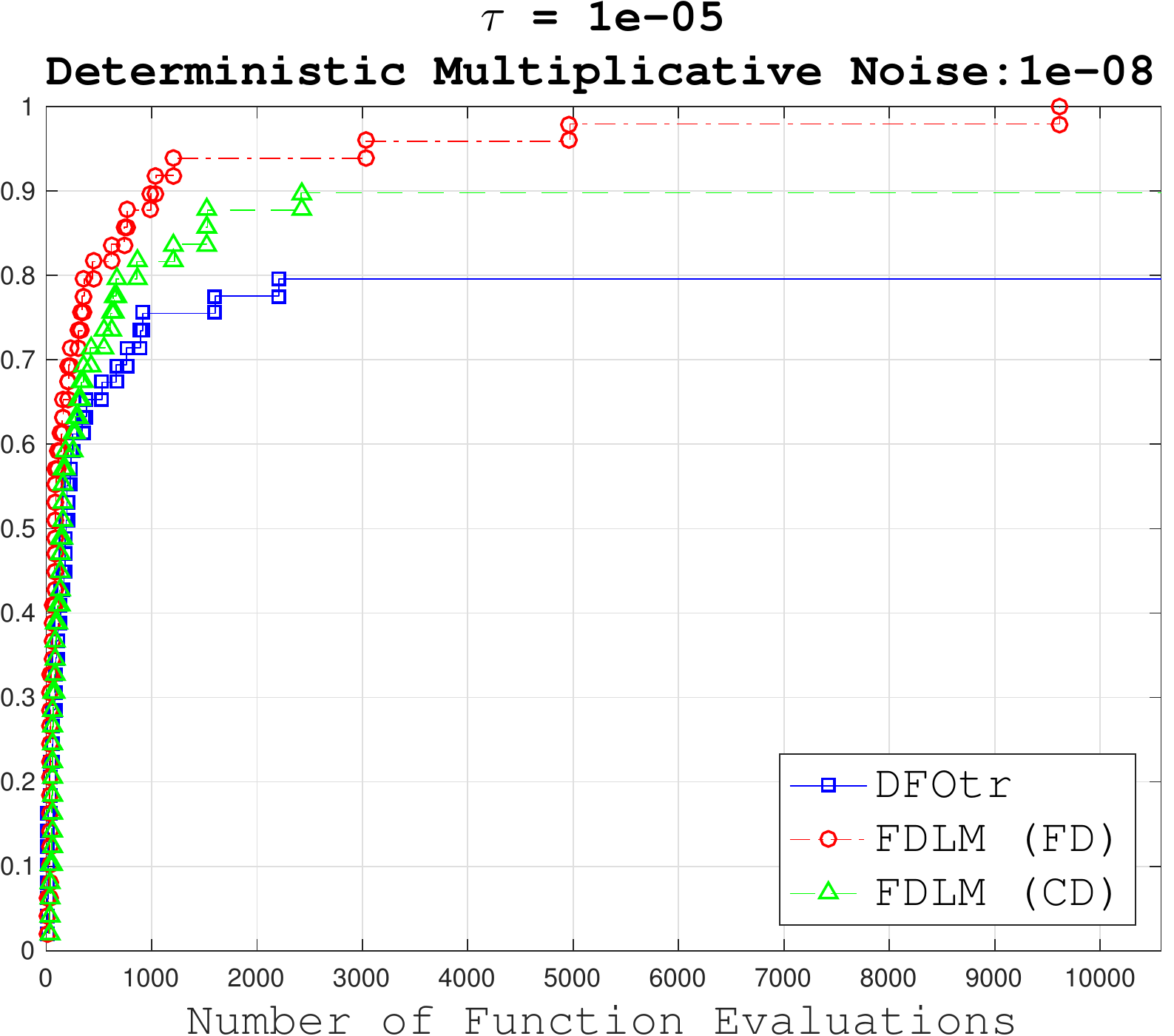}

\vspace{0.2cm}

\includegraphics[width=0.24\textwidth]{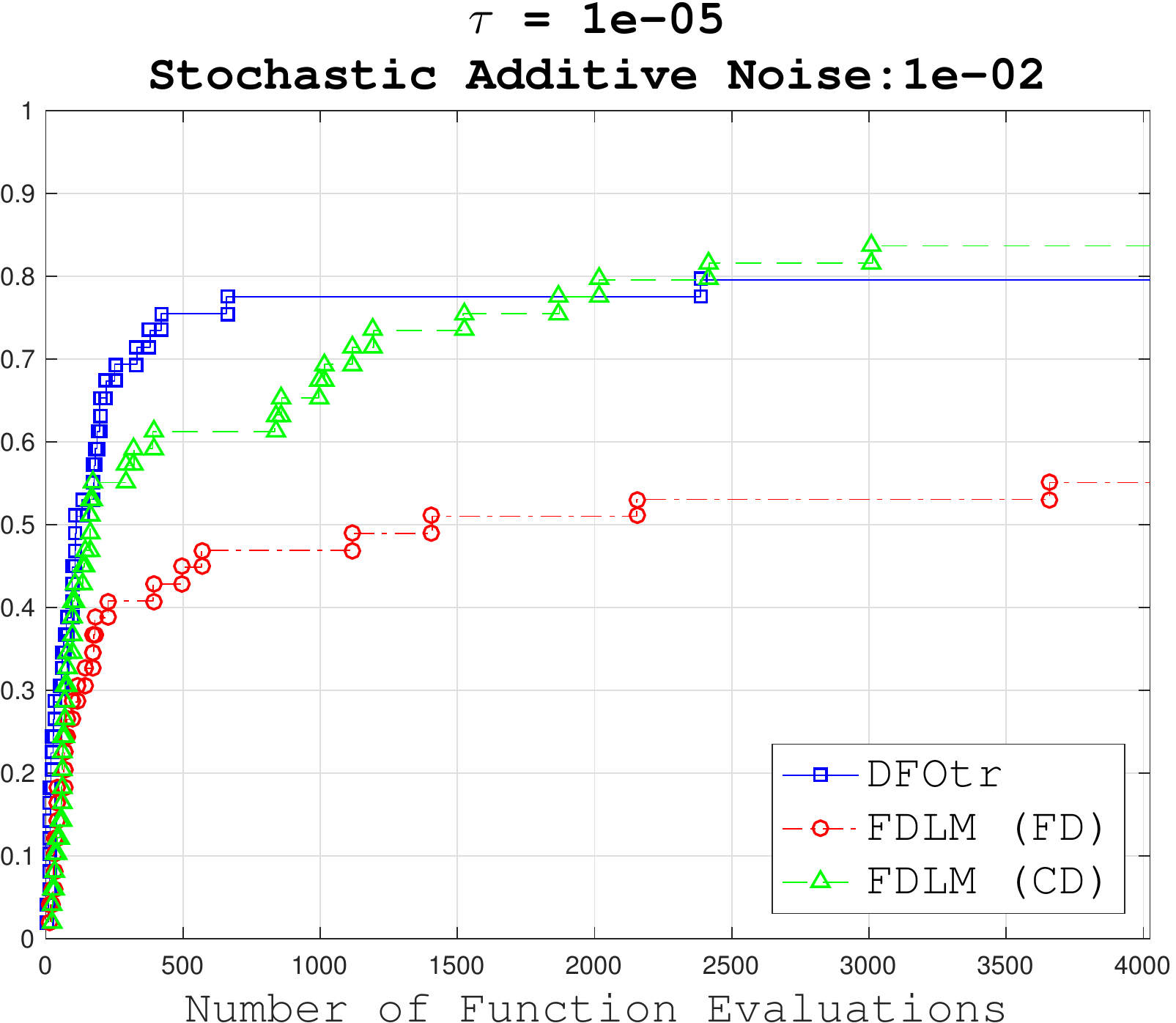}
\includegraphics[width=0.24\textwidth]{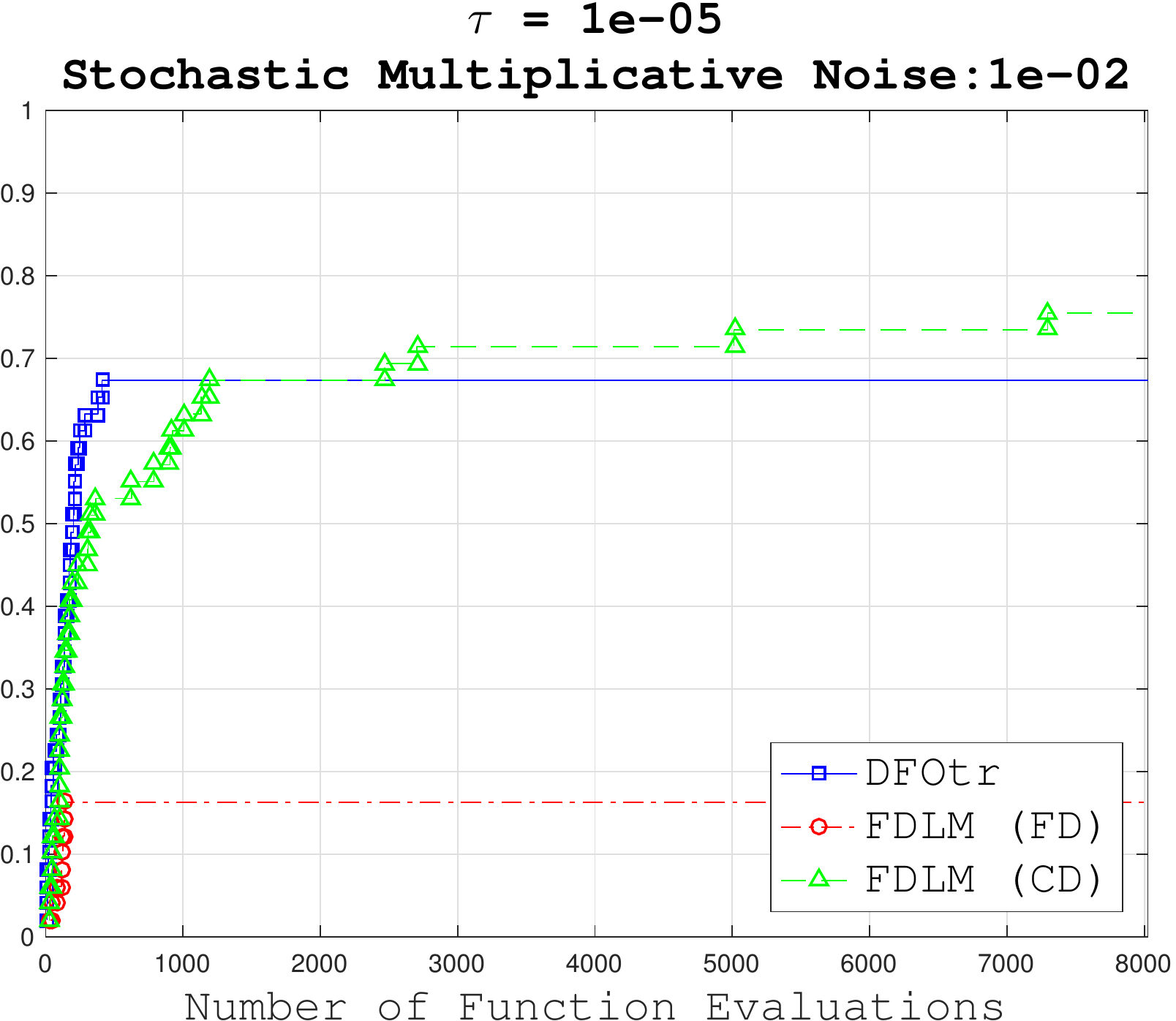}
\includegraphics[width=0.24\textwidth]{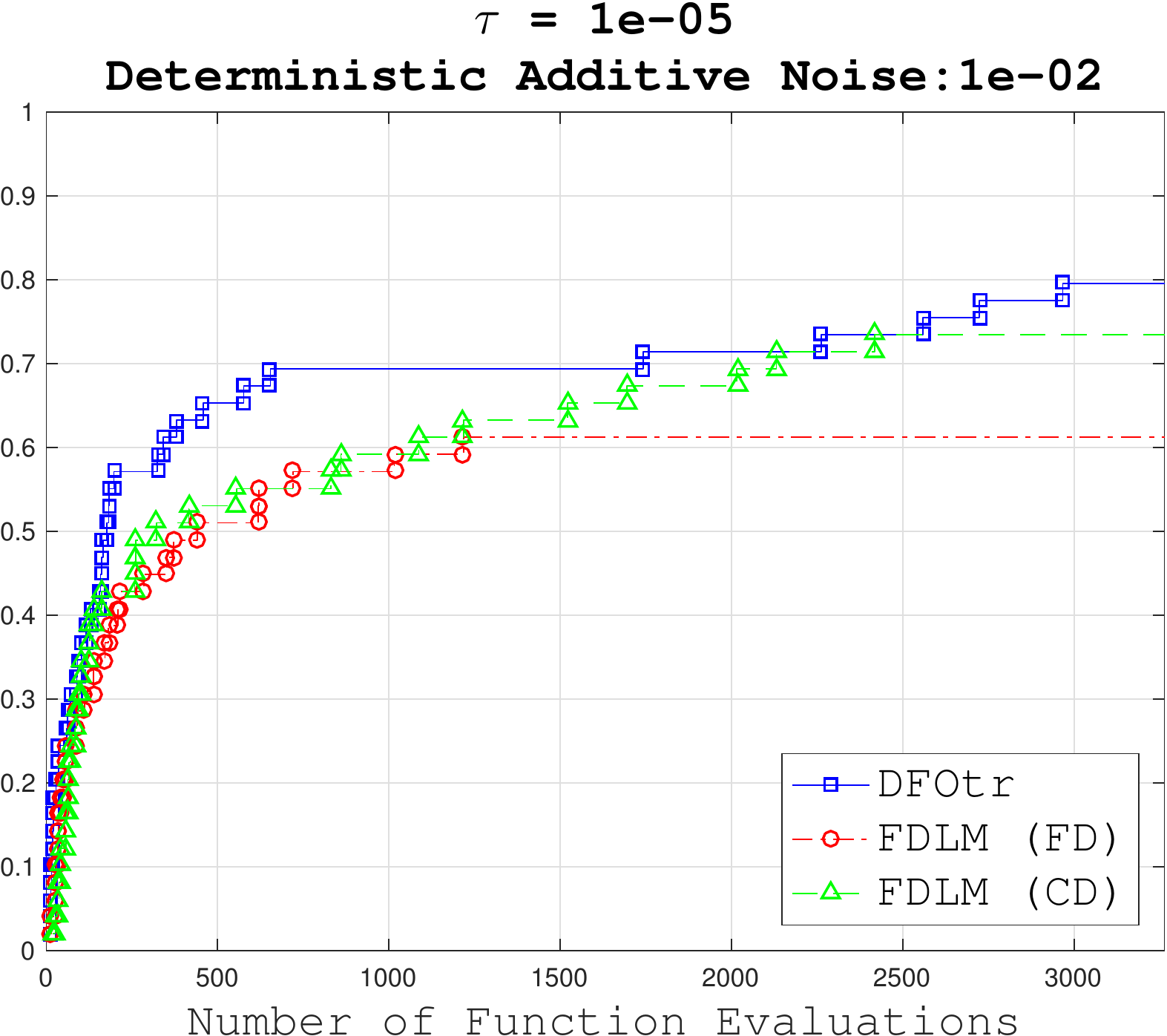}
\includegraphics[width=0.24\textwidth]{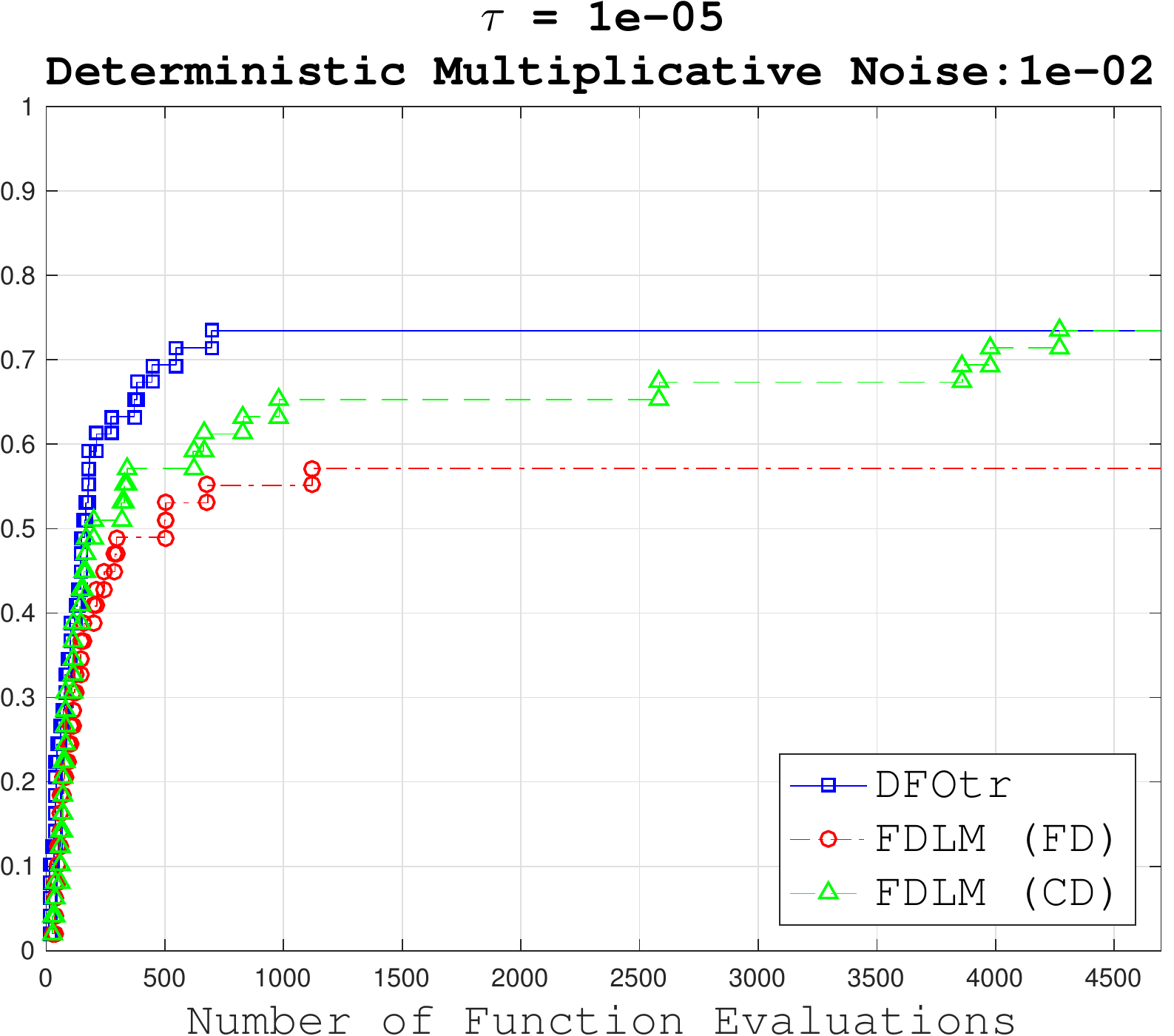}

\par\end{centering}
\caption{\small Data Profiles ($\tau = 10^{-5}$) \cite{DolMor01}. Each column represents a different noise type and each row  different noise level. Row 1: Noise level $10^{-8}$; Row 2: Noise level $10^{-2}$. Column 1: Stochastic Additive Noise; Column 2: Stochastic Multiplicative Noise; Column 3: Deterministic Additive Noise; Column 4: Deterministic Multiplicative Noise.}
\label{fig:data_profs}
\end{figure}

\subsubsection{Performance of the Recovery Mechanism}
\label{data_profs_rec}

We present data profiles to illustrate the performance of the FDLM method with and without the  \texttt{Recovery} mechanism.

\begin{figure}[h]
\begin{centering}

\includegraphics[width=0.24\textwidth]{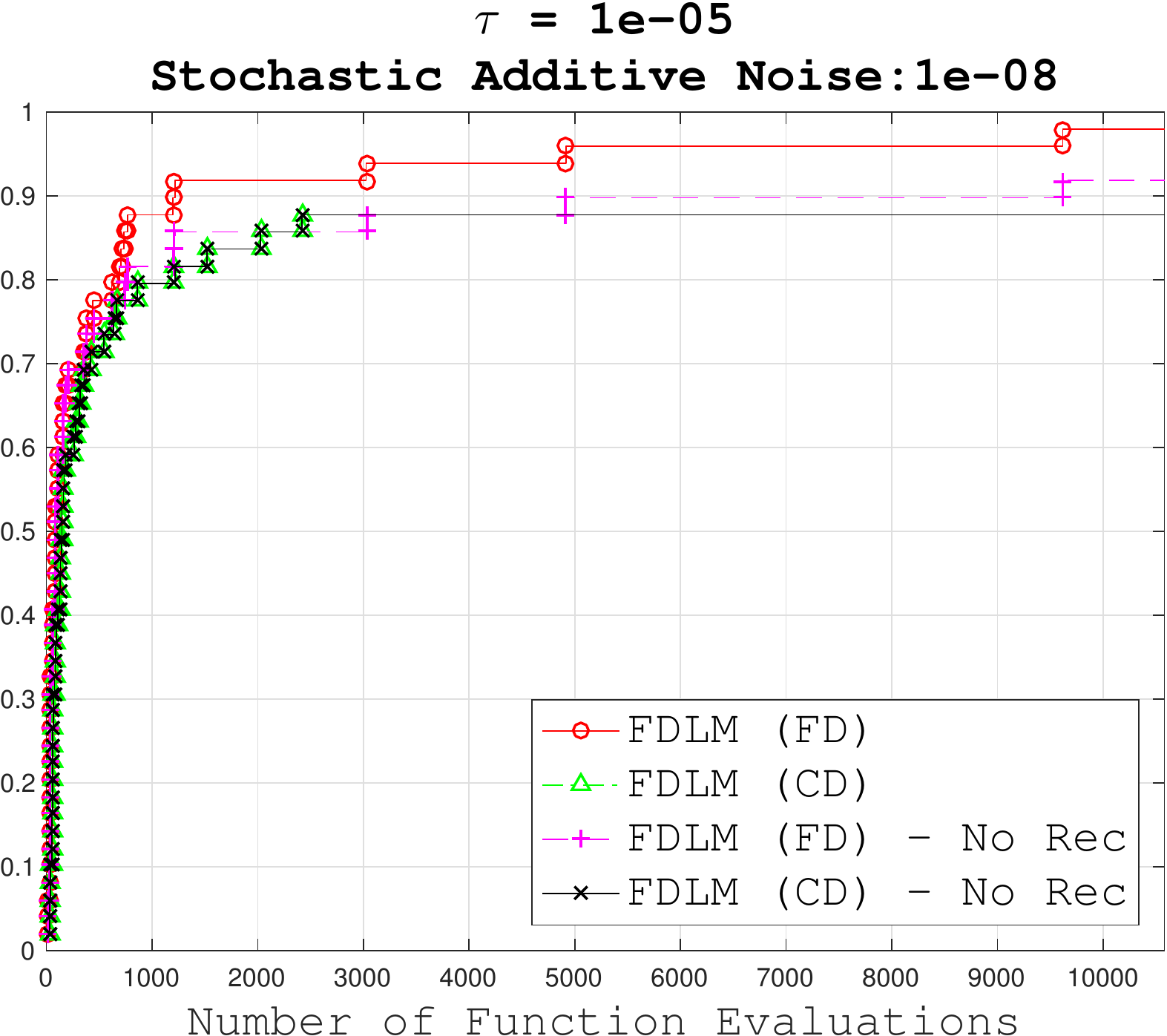}
\includegraphics[width=0.24\textwidth]{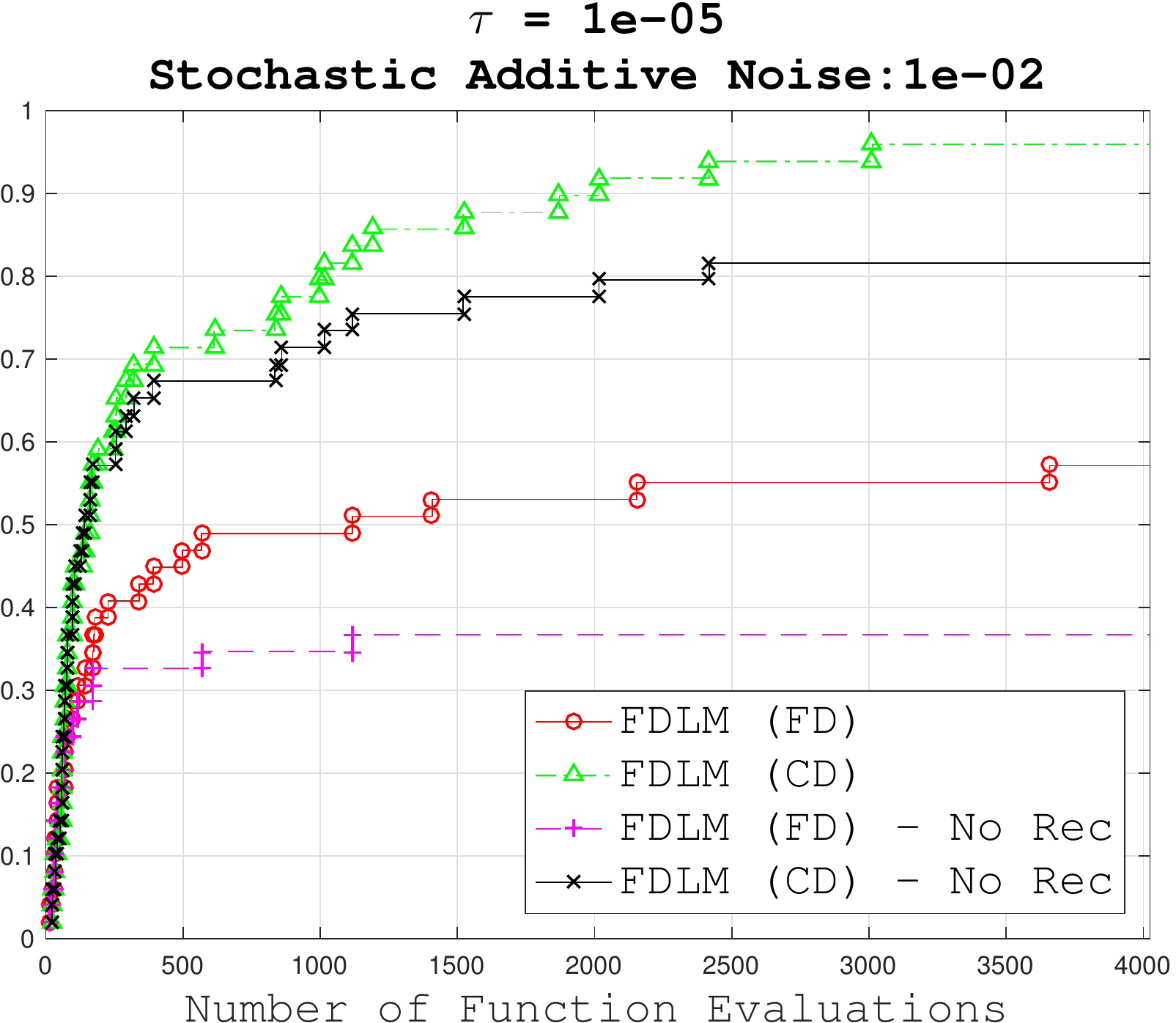}
\includegraphics[width=0.24\textwidth]{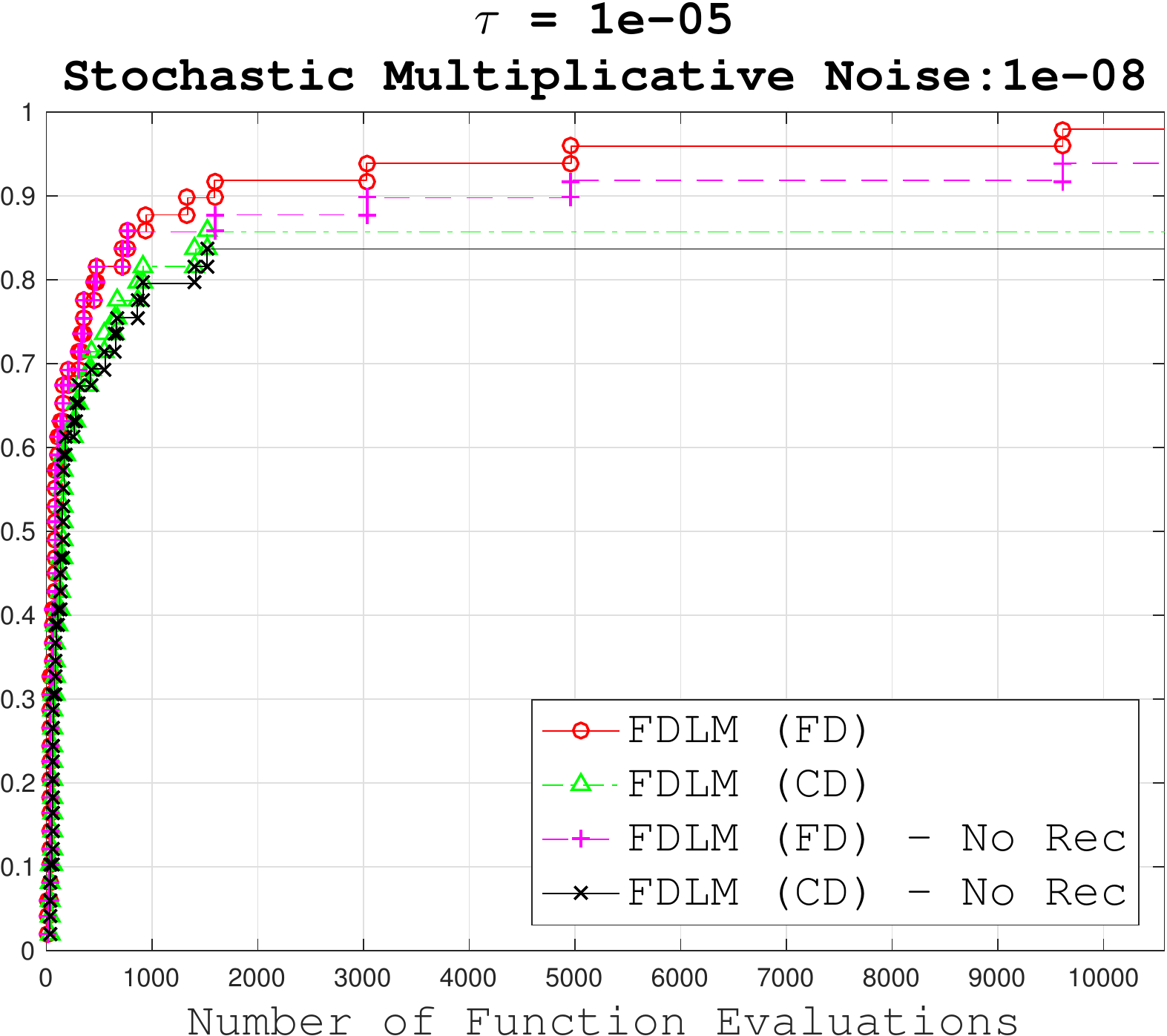}
\includegraphics[width=0.24\textwidth]{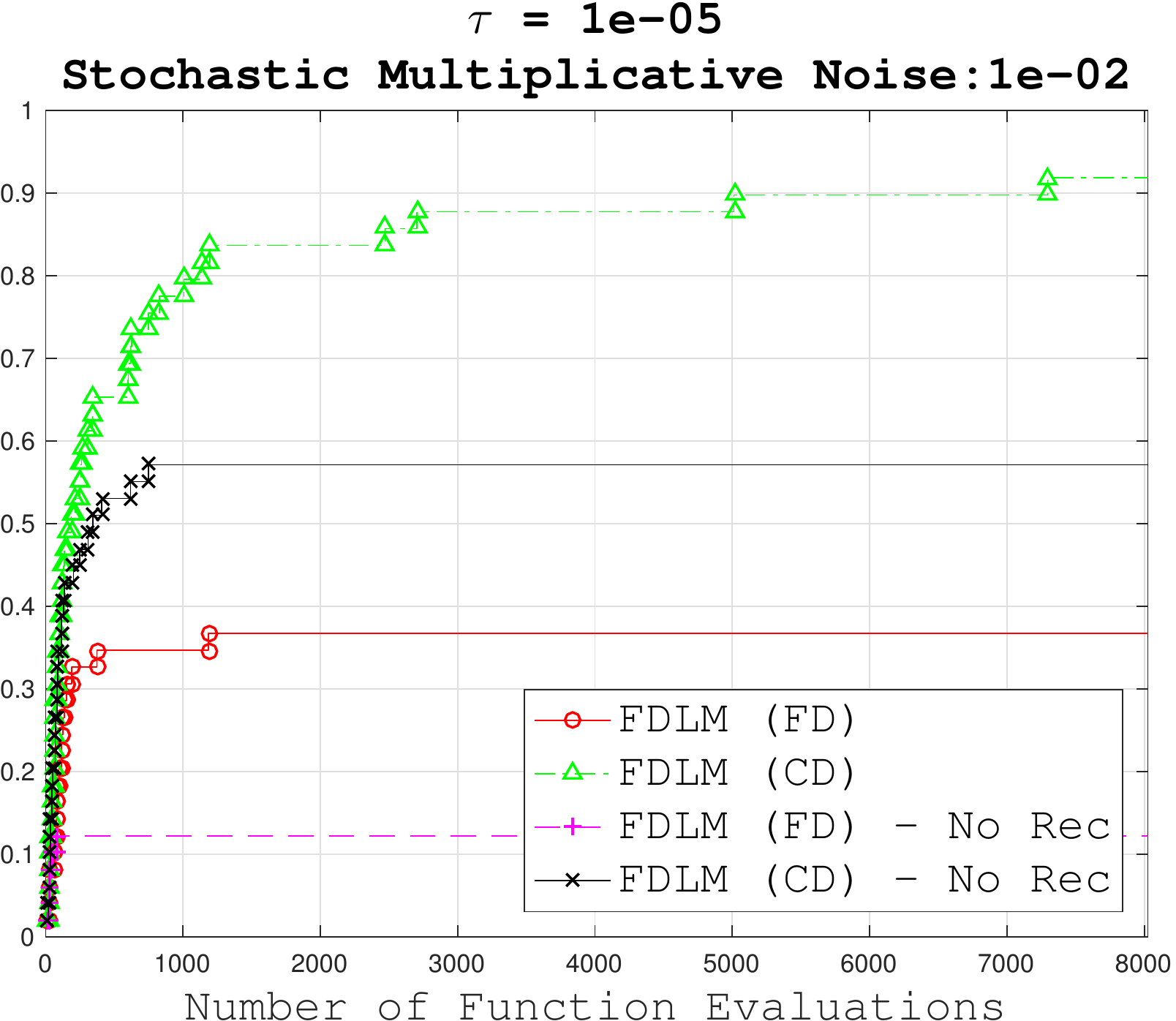}

\par\end{centering}
\caption{\small Performance of the FDLM method using forward and central differences with and without the \texttt{Recovery} procedure. The figure plots data profiles for stochastic additive and stochastic multiplicative noise (noise levels $10^{-8}$, $10^{-2}$). The first two plots show results for stochastic additive noise, and the last two plots show results for stochastic multiplicative noise.}
\label{rec_exp_data}
\end{figure}

\newpage

\section*{Acknowledgments} We are grateful to Jorge Mor\'e and Stefan Wild for many useful suggestions and insights on noise estimation and derivative-free optimization.

\bibliographystyle{siamplain}
\bibliography{references}

\begin{thebibliography}{10}

\bibitem{abramson2011nomad}
{\sc M.~A. Abramson, C.~Audet, G.~Couture, J.~E. Dennis~Jr, S.~Le~Digabel, and
  C.~Tribes}, {\em The {NOMAD} project}, 2011.

\bibitem{audet2006mesh}
{\sc C.~Audet and J.~E. Dennis~Jr}, {\em Mesh adaptive direct search algorithms
  for constrained optimization}, SIAM Journal on optimization, 17 (2006),
  pp.~188--217.

\bibitem{barton1992computing}
{\sc R.~R. Barton}, {\em Computing forward difference derivatives in
  engineering optimization}, Engineering optimization, 20 (1992), pp.~205--224.

\bibitem{bauschke2015derivative}
{\sc H.~H. Bauschke, W.~L. Hare, and W.~M. Moursi}, {\em A derivative-free
  comirror algorithm for convex optimization}, Optimization Methods and
  Software, 30 (2015), pp.~706--726.

\bibitem{bertsekas2015convex}
{\sc D.~P. Bertsekas}, {\em Convex Optimization Algorithms}, Athena Scientific,
  2015.

\bibitem{bethke1978genetic}
{\sc A.~D. Bethke}, {\em Genetic algorithms as function optimizers},  (1978).

\bibitem{bottou2017optimization}
{\sc L.~Bottou, F.~E. Curtis, and J.~Nocedal}, {\em Optimization methods for
  large-scale machine learning}, stat, 1050 (2017), p.~2.

\bibitem{choi2000superlinear}
{\sc T.~Choi and C.~T. Kelley}, {\em Superlinear convergence and implicit
  filtering}, SIAM Journal on Optimization, 10 (2000), pp.~1149--1162.

\bibitem{conn1997convergence}
{\sc A.~R. Conn, K.~Scheinberg, and P.~L. Toint}, {\em On the convergence of
  derivative-free methods for unconstrained optimization}, Approximation theory
  and optimization: tributes to MJD Powell,  (1997), pp.~83--108.

\bibitem{conn1998derivative}
{\sc A.~R. Conn, K.~Scheinberg, and P.~L. Toint}, {\em A derivative free
  optimization algorithm in practice}, in Proceedings of 7th
  AIAA/USAF/NASA/ISSMO Symposium on Multidisciplinary Analysis and
  Optimization, St. Louis, MO, vol.~48, 1998, p.~3.

\bibitem{conn2008geometry}
{\sc A.~R. Conn, K.~Scheinberg, and L.~N. Vicente}, {\em Geometry of sample
  sets in derivative-free optimization: polynomial regression and
  underdetermined interpolation}, IMA journal of numerical analysis, 28 (2008),
  pp.~721--748.

\bibitem{conn2009introduction}
{\sc A.~R. Conn, K.~Scheinberg, and L.~N. Vicente}, {\em Introduction to
  derivative-free optimization}, vol.~8, SIAM, 2009.

\bibitem{DennTorc91}
{\sc J.~E. Dennis and V.~Torczon}, {\em Direct search methods on parallel
  machines}, SIAM Journal on Optimization, 1 (1991), pp.~448--474.

\bibitem{DolMor01}
{\sc E.~D. Dolan and J.~J. Mor\'{e}}, {\em Benchmarking optimization software
  with performance profiles}, Mathematical Programming, Series~A, 91 (2002),
  pp.~201--213.

\bibitem{garmanjani2013smoothing}
{\sc R.~Garmanjani and L.~N. Vicente}, {\em Smoothing and worst-case complexity
  for direct-search methods in nonsmooth optimization}, IMA Journal of
  Numerical Analysis, 33 (2013), pp.~1008--1028.

\bibitem{gould-private}
{\sc N.~I.~M. Gould}.
\newblock (private communication), September 2016.

\bibitem{gray2006algorithm}
{\sc G.~A. Gray and T.~G. Kolda}, {\em Algorithm 856: Appspack 4.0:
  Asynchronous parallel pattern search for derivative-free optimization}, ACM
  Transactions on Mathematical Software (TOMS), 32 (2006), pp.~485--507.

\bibitem{hamming2012introduction}
{\sc R.~W. Hamming}, {\em Introduction to applied numerical analysis}, Courier
  Corporation, 2012.

\bibitem{Higham:MCT}
{\sc N.~J. Higham}, {\em The {Matrix Computation Toolbox}},
  \url{http://www.ma.man.ac.uk/~higham/mctoolbox}.

\bibitem{holland1975adaptation}
{\sc J.~H. Holland}, {\em Adaptation in natural and artificial systems: an
  introductory analysis with applications to biology, control, and artificial
  intelligence.}, U Michigan Press, 1975.

\bibitem{hooke1961direct}
{\sc R.~Hooke and T.~A. Jeeves}, {\em ``{D}irect search''solution of numerical
  and statistical problems}, Journal of the ACM (JACM), 8 (1961), pp.~212--229.

\bibitem{jones1993lipschitzian}
{\sc D.~R. Jones, C.~D. Perttunen, and B.~E. Stuckman}, {\em Lipschitzian
  optimization without the lipschitz constant}, Journal of Optimization Theory
  and Applications, 79 (1993), pp.~157--181.

\bibitem{kelley2011implicit}
{\sc C.~T. Kelley}, {\em Implicit filtering}, vol.~23, SIAM, 2011.

\bibitem{keskar2017limited}
{\sc N.~S. Keskar and A.~W{\"a}chter}, {\em A limited-memory quasi-newton
  algorithm for bound-constrained non-smooth optimization}, Optimization
  Methods and Software,  (2017), pp.~1--22.

\bibitem{kirkpatrick1983optimization}
{\sc S.~Kirkpatrick, C.~D. Gelatt, M.~P. Vecchi, et~al.}, {\em Optimization by
  simmulated annealing}, science, 220 (1983), pp.~671--680.

\bibitem{KoldLewiTorc03}
{\sc T.~G. Kolda, R.~M. Lewis, and V.~Torczon}, {\em Optimization by {D}irect
  {S}earch: new perspectives on some classical and modern methods}, SIAM
  Review, 45 (2003), pp.~385--482.

\bibitem{larson2013non}
{\sc J.~Larson and S.~M. Wild}, {\em Non-intrusive termination of noisy
  optimization}, Optimization Methods and Software, 28 (2013), pp.~993--1011.

\bibitem{lewis2009nonsmooth}
{\sc A.~S. Lewis and M.~L. Overton}, {\em Nonsmooth optimization via bfgs},
  Submitted to SIAM J. Optimiz,  (2009), pp.~1--35.

\bibitem{lewis2013nonsmooth}
{\sc A.~S. Lewis and M.~L. Overton}, {\em Nonsmooth optimization via
  quasi-newton methods}, Mathematical Programming, 141 (2013), pp.~135--163.

\bibitem{lewis2000direct}
{\sc R.~M. Lewis, V.~Torczon, and M.~W. Trosset}, {\em Direct search methods:
  then and now}, Journal of computational and Applied Mathematics, 124 (2000),
  pp.~191--207.

\bibitem{maggiar2015derivative}
{\sc A.~Maggiar, A.~W{\"a}chter, I.~S. Dolinskaya, and J.~Staum}, {\em A
  derivative-free trust-region algorithm for the optimization of functions
  smoothed via gaussian convolution using adaptive multiple importance
  sampling}, tech. report, Technical Report IEMS Department, Northwestern
  University, 2015.

\bibitem{MaraNoce00a}
{\sc M.~Marazzi and J.~Nocedal}, {\em Wedge trust region methods for derivative
  free optimization}, Mathematical Programming, Series~A, 91 (2002),
  pp.~289--305.

\bibitem{more2009benchmarking}
{\sc J.~J. Mor{\'e} and S.~M. Wild}, {\em Benchmarking derivative-free
  optimization algorithms}, SIAM Journal on Optimization, 20 (2009),
  pp.~172--191.

\bibitem{more2011estimating}
{\sc J.~J. Mor{\'e} and S.~M. Wild}, {\em Estimating computational noise}, SIAM
  Journal on Scientific Computing, 33 (2011), pp.~1292--1314.

\bibitem{more2012estimating}
{\sc J.~J. Mor{\'e} and S.~M. Wild}, {\em Estimating derivatives of noisy
  simulations}, ACM Transactions on Mathematical Software (TOMS), 38 (2012),
  p.~19.

\bibitem{nedic2001convergence}
{\sc A.~Nedi{\'c} and D.~Bertsekas}, {\em Convergence rate of incremental
  subgradient algorithms}, in Stochastic optimization: algorithms and
  applications, Springer, 2001, pp.~223--264.

\bibitem{NeldMead65}
{\sc J.~A. Nelder and R.~Mead}, {\em A simplex method for function
  minimization}, Computer Journal, 7 (1965), pp.~308--313.

\bibitem{nesterov2017random}
{\sc Y.~Nesterov and V.~Spokoiny}, {\em Random gradient-free minimization of
  convex functions}, Foundations of Computational Mathematics, 17 (2017),
  pp.~527--566.

\bibitem{mybook}
{\sc J.~Nocedal and S.~Wright}, {\em Numerical {O}ptimization}, Springer New
  York, 2~ed., 1999.

\bibitem{powell1972unconstrained}
{\sc M.~J. Powell}, {\em Unconstrained minimization algorithms without
  computation of derivatives}, tech. report, 1972.

\bibitem{Powe02}
{\sc M.~J. Powell}, {\em {UOBYQA}: unconstrained optimization by quadratic
  approximation}, Mathematical Programming, 92 (2002), pp.~555--582.

\bibitem{powell2006newuoa}
{\sc M.~J. Powell}, {\em The {NEWUOA} software for unconstrained optimization
  without derivatives}, in Large-scale nonlinear optimization, Springer, 2006,
  pp.~255--297.

\bibitem{rios2013derivative}
{\sc L.~M. Rios and N.~V. Sahinidis}, {\em Derivative-free optimization: a
  review of algorithms and comparison of software implementations}, Journal of
  Global Optimization, 56 (2013), pp.~1247--1293.

\bibitem{schttfkowski1987more}
{\sc K.~Schttfkowski}, {\em More test examples for nonlinear programming
  codes}, Lecture Notes in Econom. and Math. Systems, 282 (1987).

\bibitem{squire1998using}
{\sc W.~Squire and G.~Trapp}, {\em Using complex variables to estimate
  derivatives of real functions}, SIAM review, 40 (1998), pp.~110--112.

\bibitem{wild2008orbit}
{\sc S.~M. Wild, R.~G. Regis, and C.~A. Shoemaker}, {\em {ORBIT}: Optimization
  by radial basis function interpolation in trust-regions}, SIAM Journal on
  Scientific Computing, 30 (2008), pp.~3197--3219.

\bibitem{Wrig96}
{\sc M.~H. Wright}, {\em Direct search methods: {O}nce scorned, now
  respectable}, in Numerical Analysis 1995 (Proceedings of the 1995 Dundee
  Biennial Conference in Numerical Analysis), Addison Wesley Longman, 1996,
  pp.~191--208.

\end{thebibliography}
\end{document}